\documentclass[11pt,twoside, a4paper]{article}

\usepackage{nemish}

\usepackage{hyperref} 

\begin{document}
\title{Local laws for polynomials of Wigner matrices}

 \author{
L\'aszl\'o Erd\H{o}s\footnote{\hspace{0.15cm}Partially funded by ERC Advanced Grant RANMAT No. 338804} 
\\
{\small \begin{tabular}{c}{IST Austria}\\{lerdos@ist.ac.at} \end{tabular}} 
\and Torben Kr\"uger\footnote{\hspace{0.15cm}Partially supported by the Hausdorff Center for Mathematics \newline  Date: \today }
\addtocounter{footnote}{-2}\addtocounter{Hfootnote}{-2}\\ 
{\small \begin{tabular}{c} University of Bonn \\ torben.krueger@uni-bonn.de \end{tabular}}
\and Yuriy Nemish\footnotemark\\
{\small \begin{tabular}{c} IST Austria\\ yuriy.nemish@ist.ac.at\end{tabular}}
}
\date{} 

\maketitle
\thispagestyle{empty} 

\begin{abstract}  
  We consider general self-adjoint polynomials in several independent random matrices whose
entries are centered and have the same variance. We show that  under 
certain conditions the local law holds up to the optimal scale, i.e.,  the
eigenvalue density on scales just above the eigenvalue spacing follows the global density of
states which is determined by free probability theory. We prove that these conditions
hold for general homogeneous polynomials of degree two  and for symmetrized products of independent matrices with \emph{i.i.d.} entries, thus establishing the optimal bulk local law for
these classes of ensembles. In particular,  we generalize a similar result of Anderson 
for anticommutator. For more general polynomials our conditions are effectively checkable numerically.
   \end{abstract}


 \section{Introduction}
\label{sec:introduction}

Polynomials of random matrices have been subject of intensive research in the last thirty years. 
In the 1980's Voiculescu realized that random matrices and their polynomials can be used to  solve some basic problems in operator algebras of
free groups, which gave  birth to \emph{free probability theory}.  Roughly speaking, large independent random matrices serve as concrete approximants 
to free elements in abstract  noncommutative probability spaces, i.e. unital $C^*$-algebras with a tracial state.  In other words,
\emph{freeness} is the appropriate operator algebraic  analogue of independence in classical probability. 
A classical example for such  result is  Theorem 2.2 from \cite{Voic91}
showing that the trace of a self-adjoint polynomial $p(X_1,\ldots,X_k)$ in  $k$ independent $N\times N$ standard complex Gaussian
(GUE)  matrices converges in expectation and almost surely, as the size of the matrices  goes to infinity,
 to the trace of the polynomial $p(\semic_1,\ldots,\semic_k)$ in free semicircular variables.

Voiculescu's pioneering result  has since been extended in many directions. Convergence in operator norm was proved 
in  \cite{HaagThor05}, while convergence of the spectrum, in particular absence of outliers, was established 
in \cite{HaagSchuThor06}.  Another direction of generalizations was to replace Gaussian matrices with Wigner matrices, i.e. 
retain independence of the matrix elements while dropping the special distribution; for the first such result see \cite{Dyke93},
followed by many others, e.g.  \cite{Ande13,BeliBercCapi,BeliCapi17,CapiDoMa07,Male12} and references therein.
Yet another line of research 
 concerns certain qualitative  properties of the limiting spectral measure. For example, the limiting spectral measure 
 for self-adjoint polynomials does not contain atoms \cite{MaiSpeiWebe17,ShlySkou15} and for monomials it is even absolutely continuous \cite{CharShly16}.
  Very recently, the H\"{o}lder continuity of the cumulative distribution function was studied for polynomials \cite{BannMai} and rational functions \cite{MaiSpeiYin1} of random matrices.

 A common feature of all these results, as well as the scope of the underlying methods,
 is that they describe the spectrum of  $p(X_1, \ldots, X_k)$ on the global  scale, which is typically by a  factor
 $N$ larger than the scale of the  eigenvalue spacing. What happens on  scales in between? 
 Recent developments revealed that 
 the eigenvalue density of Wigner and related matrices on  \emph{mesoscopic} scales, i.e., scales involving $\sim N^{\gamma}$ eigenvalues for $0<\gamma<1$, 
also becomes deterministic in the large $N$ limit. Such results are commonly called \emph{local laws} and they have been
established in increasing generality for Hermitian matrices; with independent entries, see e.g.   \cite{AjanErdoKrug17,ErdoKnowYauYin13b,GotzNaumTikhTimu18,HeKnowRose18,TaoVu13},  
with general short range correlation structure for their matrix elements \cite{AjanErdoKrug19,ErdoKrugSchr18}, as well as for adjacency matrices for random regular graphs \cite{BaueHuanYau19,BaueHuanKnowYau17}.
 Local laws beyond mean field models, in particular for band matrices are especially challenging \cite{BaoErdo17, ErdoKnowYauYin13c,YangYin}. 

One of the main motivations for local laws is their key role in the proof of the Wigner-Dyson-Mehta conjecture
on the local spectral universality, see \cite{ErdoYau12}.
Recent developments on the local ergodicity of the Dyson Brownian motion (DBM) have demonstrated that local laws are 
the only model-dependent inputs for the universality proofs using the DBM,
see \cite{ErdoYauBook} for an overview and newer results in \cite{ErdoSchn17,LandYau17,LandYau}.

In this paper we prove  optimal local laws for self-adjoint polynomial models, thus connecting two large areas of recent research in random matrices.
We will combine methods from  free probability theory, most importantly the concept of linearization, with techniques developed
for local laws, such as large deviations and fluctuation averaging phenomenon. We point out that
mesoscopic spectral properties for general polynomials have not been studied before. Local laws have only been
established for a very few specific polynomials such as  (i) the anticommutator, $X_1X_2+X_2X_1$, of two  independent Wigner matrices in \cite{Ande15} 
and (ii) the (non-Hermitian) product $Y_1Y_2\ldots Y_k$  of several independent  \emph{i.i.d.} matrices in \cite{GotzNaumTikh17,Nemi_LLProducts}.

We now explain the method and some difficulties. The first  major obstacle is that the entries of
a general polynomial $P:=p(X_1,\ldots,X_k)$ of, say, independent $N\times N$ Wigner matrices, have a very complex non-local
correlation structure. This makes it impossible to apply the tools developed in \cite{AjanErdoKrug19}  or \cite{AndeZeit08} directly in the polynomial setting.
However, the well-known \emph{linearization trick}, originally developed in the context of automata theory \cite{Klee56,Schu61} and revived for use in random matrix theory \cite{HaagSchuThor06, HaagThor05},
transforms the polynomial model into a much larger random matrix $\Hb$ with 
a transparent correlation structure.  In fact, the linearized matrix  is a tensor linear combination of the independent Wigner matrices
with matrix coefficients whose  dimension $m\times m$ depends only on the polynomial $p$ and is independent of $N$.
This structure exactly corresponds to certain block matrices and more generally  \emph{Kronecker random matrices} introduced in  \cite{AltErdoKrugNemi_Kronecker}.
We remark that the linearization technique has been widely used in the free probability community
to study polynomials of random matrices 
on the global scale, see e.g. \cite{Ande13,BeliMaiSpei17,HaagThor05,HeltMaiSpei18,HeltMaccVinn06}
and  \cite[Chapter~5]{AndeGuioZeitBook} for a pedagogical introduction.

Local laws for  Kronecker matrix $\Hb$ have been studied in detail in \cite{AltErdoKrugNemi_Kronecker} by proving concentration of its resolvent
$(\Hb-z I_m\otimes I_N)^{-1}$  around the solution of corresponding matrix Dyson equation for spectral parameter $z$ in complex upper half-plane.
In contrast to the Kronecker case, to study the resolvent $(P-z)^{-1}$ of our polynomial,
we have to consider the \emph{generalized resolvent} of the linearized matrix $\Hb$, i.e., $(\Hb - z J\otimes I_N)^{-1}$, where $J$ is a rank-one $m\times m$
 matrix. Thus the results on the Kronecker matrices cannot be directly applied, in fact \emph{a priori} it is unclear whether the generalized resolvent is stable.
This is the second major obstacle in the study of polynomial models, and we overcome it by  simultaneously considering the generalized resolvent of $\Hb$ and its usual regularized version. It turns out that a certain nilpotency structure inherent for linearizations of polynomials yields the boundedness of the generalized resolvents
even after the regularization is removed. 

After these two key obstacles cleared, we can essentially use the local law established
for general  Kronecker matrices  in  \cite{AltErdoKrugNemi_Kronecker}
under two basic conditions: (i) the solution of the underlying Dyson equation is bounded and (ii) the stability operator is invertible. 
These conditions are verified for homogeneous polynomials  of degree two in Wigner matrices, substantially  generalizing the case of anticommutator studied by Anderson in \cite{Ande15}.
 We also verify them for the symmetrized product $Y_1 \cdots Y_k (Y_1 \cdots Y_k)^*$ of independent matrices with \emph{i.i.d.} entries.
 For more general polynomials, 
the validity of these conditions depends on the structure of the linearization but they are independent of $N$, so they are
numerically checkable. Notice that the linearization of a polynomial is not unique. In fact, any of 
the standard linearizations, obtained via a simple recursive procedure, typically has unnecessarily 
large dimension. It is much more effective to use the so-called \emph{minimal linearization}, which  is canonical 
\cite{BersReutBook,HeltMaccVinn06,BallGroeMala2005,BallGroeMala2006a,BallGroeMala2006b}, and we present numerical examples to demonstrate its advantages. Since both linearizations
are nilpotent, our theory equally applies to them.
We expect that for any self-adjoint polynomial there exists a linearization for which the conditions (i) and (ii) above hold everywhere in the bulk, i.e., where the density of states is bounded and bounded away from zero,
and in fact the minimal linearization is a natural  candidate. 

The task of dealing with the generalized resolvent of the linearization is inherent in other works on 
polynomials of random matrices that use the resolvent method, see  
\cite{Ande13,CapiDoMa07,HaagThor05,Male12}. In the most general setup,
 Anderson in  \cite{Ande13} used one of  the explicit standard linearizations 
to prove the  global law and the convergence of the norm for polynomials in Wigner matrices.  The structure of the standard
linearization allowed him to control the generalized resolvent directly from the resolvent of  $P$ via Schur complement formula.
A simpler version of this idea was presented in  \cite[Chapter~5.5]{AndeGuioZeitBook}. 
For the canonical minimal linearization such  simple a priori bound is not available. From algebraic point of view, the main novelty 
of our work is to identify a nilpotency structure in the minimal linearization and show that this structure is
sufficient to control the generalized resolvent.  From the analytic point of view,  we advocate the method of the stability
analysis of the Dyson equation combined with large deviation and fluctuation averaging estimates as
presented in \cite{AltErdoKrugNemi_Kronecker}, which itself is a natural extension  of many previous works on local 
laws for Wigner and  Wigner-type matrices.  This approach substitutes   the Poincar\'{e} inequality used in \cite{CapiDoMa07}
and the $L^p$ bounds used in \cite{Ande13} whose analogue for Wigner  and Wishart matrices go back to Bai and Silverstein \cite{Bai93,BaiSilv98}.

We close this introduction with a remark on local spectral universality. Our local law is optimal and it provides the necessary
 input for the customary proofs via the Dyson Brownian motion (DBM) as mentioned above. Thus we could easily prove bulk universality
 for polynomials that already  have a small additive GUE component. We cannot, however,  apply the 
 usual DBM argument to the linearized matrix since it would need to   assume that 
 a small global Gaussian component is present in  $\Hb$, but $\Hb$ has many zero blocks by construction.
 This fundamental difficulty has been overcome for certain band  matrices  \cite{BourErdoYauYin17,BourYauYin} which also has many zero entries.
 However, the specific  band structure was essential in those proofs. For the local spectral universality for a polynomial $P$
 the structure of the linearized
 matrix needs to  be exploited in a similar fashion. We note that apart from the trivial case of Hermitian polynomials of a single random
 matrix, currently the only nontrivial  universality results for polynomials are obtained for  very special cases and only for  Gaussian matrices
 by exploiting their determinantal structure, see, e.g.,  the  survey \cite{AkemIpse15} on products of large Gaussian random matrices.
 
 In Section~\ref{sec:polyAndLin} we introduce the concept of nilpotent linearization, the corresponding Dyson equation and
we present our main result together with  the conditions expressed in terms of the solution to the Dyson equation. Section~\ref{sec:aboutlinearizations}
is devoted to control the generalized resolvent by exploiting the nilpotent structure. In Section~\ref{sec:MDE} we present the existence
and uniqueness of the solution to  the Dyson equation by using
semicircular variables.
 In Section~\ref{sec:locallaw} we give a proof of the local law.
Finally, as an application,  in Section~\ref{sec:examples} we show the optimal bulk local law for general homogeneous polynomials of degree two in Wigner matrices and for symmetrized products of matrices with \emph{i.i.d.} entries.
Additional information on two different linearizations,
 as well as their numerical comparison
 are deferred to Appendix~\ref{sec:aLinCM}, while in Appendix~\ref{sec:semicircular} we collected some basic information   on
 semicircular variables for the reader's convenience.
 
\medskip

\emph{Acknowledgement.} The authors are grateful to Oskari Ajanki for his invaluable help at the initial stage of this project, to Serban Belinschi for useful discussions, to Alexander Tikhomirov for calling our attention to the model example in Section~\ref{sec:exampl-sing-valu} and to  the  anonymous referee for suggesting to simplify certain proofs.

\section{Main results}
\label{sec:polyAndLin}

\subsection{Linearization and Dyson equation in $C^*$-algebras}\label{sec:linDyson}

Fix $\alpha_*,\beta_*\in \NN$.
Let $\mathscr{A}$ be a unital $C^*$-algebra with  norm $\| \cdot  \|_{\Acc}$ and identity element $\cstarunit$, and let $x_1,\ldots,x_{\alpha_*},y_1,\ldots,y_{\beta_*} \in \mathscr{A}$ with $x_i^*=x_i$ for $1\leq i\leq \alpha_*$.
For any $n\in \NN$ and $\rb = (r_1,\ldots,r_n) \in \Acc^{n}$ we define $\|\rb \|: = \max_{1\leq i \leq n} \|r_i\|_{\Acc}$.
Denote by
\begin{equation*}
  \CC\la \xb,\yb,\yb^* \ra : = \CC\la x_1,\ldots,x_{\alpha_*},y_1,\ldots,y_{\beta_*},y_1^*,\ldots,y_{\beta_*}^* \ra  
\end{equation*}
the set of polynomials with complex coefficients in noncommutative elements $\{x_{\alpha},y_{\beta},y_{\beta}^*,1\leq\alpha\leq \alpha_*,  1\leq \beta\leq \beta_*\}$.
Let $p:= p(\xb, \yb, \yb^*) \in \CC \la \xb,\yb,\yb^* \ra$ and assume that $p$ is self-adjoint, i.e.,
\begin{equation*}
  (p(\xb, \yb, \yb^*))^*=p(\xb, \yb, \yb^*)
  .
\end{equation*}

It is a common and convenient practice to study the polynomials via their linearizations.
Linearization allows to transform polynomial model into a linear one, which is typically easier to analyze.
The price for doing this is the increased dimension of the model, which can quickly become prohibitive for more complicated polynomials.
\begin{defn}[Self-adjoint linearization]\label{def:saLin}
  Let $m\in\NN$ and let $\Lb\in (\CC\la \xb,\yb,\yb^* \ra)^{m\times m}$ be a matrix, whose matrix elements are polynomials of degree at most $1$.
  Suppose that
  \begin{equation} \label{eq:Lexpr1}
    \Lb  =
    \begin{pmatrix}
      \lambda & \ell^*
      \\
      \ell & \widehat{\Lb}
    \end{pmatrix},
  \end{equation}
  where $\widehat{\Lb}$ is the $(m-1)\times (m-1)$ submatrix of $\Lb$.
  We call $\Lb$ a \emph{self-adjoint linearization} (or simply \emph{linearization}) of $p\in \CC \la \xb, \yb, \yb^* \ra$ if $\Lb^*=\Lb$  and there exists $\epsD>0$ such that for all $\|\xb\|<\epsD, \|\yb\|< \epsD$, the matrix $\widehat{\Lb}$ is invertible and satisfies 
  \begin{equation}\label{eq:Lin11}
    p
    =
    \lambda - \ell^*\widehat{\Lb}^{\,-1}\ell
    .
  \end{equation}
  We will refer to $m$ as the \emph{dimension} of the linearization $\Lb$.
\end{defn}
Note that due to the property $\Lb^*=\Lb$ a self-adjoint linearization $\Lb$ can be written as
\begin{equation}\label{eq:linearization}
  \Lb
  =
  K_0\otimes\cstarunit - \sum_{\alpha=1}^{\alpha_*} K_\alpha\otimes x_\alpha - \sum_{\beta=1}^{\beta_*}( L_{\beta}\otimes y_\beta+L^*_{\beta} \otimes y^*_\beta)
  ,
\end{equation}
where $K_\alpha, L_\beta\in \CC^{m\times m}$ and $K_0^*=K_0$, $K_\alpha^*=K_\alpha$.
In this paper all linearizations are self-adjoint, so we will not stress self-adjointness all the times.

For each polynomial one can write many different linearizations.
In the related literature \cite{AndeGuioZeitBook,BersReutBook,HaagThor05,HeltMaccVinn06} one can distinguish two groups of methods used for constructing the linearizations of polynomial (and more generally rational) functions.
One group uses very explicit algorithms to build linearizations first for monomials, and then extending them to linear combinations of monomials.
These algorithms are well-know, but for the sake of completeness we will  give in  Appendix~\ref{sec:abiglinearization} a version of such an explicit linearization.
This is a standard construction that typically yields a linearization in very high dimension.
For many practical reasons it is better to work with smaller linearizations, which naturally leads to the notion of \emph{minimal} linearization.
\begin{defn}[Minimal linearization]
  A linearization of a polynomial  is called \emph{minimal} if it has the smallest dimension among all linearizations.
\end{defn}
Minimal linearization can be obtained by reducing the dimension of some previously constructed linearization (see e.g. \cite[Chapter~2.3]{BersReutBook}) and then using the symmetrization trick if needed  to restore self-adjointness \cite[Lemma~4.1 (3)]{HeltMaccVinn06}.
For completeness, as well as for the reader's convenience, in  Appendix~\ref{sec:minimallinearization} we present a somewhat different
 algorithmic procedure that  directly yields a minimal (self-adjoint) linearization from any (self-adjoint) linearization. 
 
Typically the dimension of a minimal linearization is significantly smaller compared to the standard
 linearization constructed in  Appendix~\ref{sec:abiglinearization} (see  Appendix~\ref{sec:aLinComp} for comparison), 
 which makes it much more convenient to work with if we want to study the model numerically.

In order to use linearizations for studying the resolvents of polynomials of random matrices it will be convenient to work with a special class of \emph{nilpotent} linearizations that we introduce  now. 
\begin{defn}[Nilpotent family] A family of matrices $\{ R_i \in \CC^{m\times m}  \; : \; i\in I\}$    is called \emph{nilpotent} if there exists an integer $n$ such that
  $R_{i_1} R_{i_2} \ldots R_{i_n}=0$ for any $n$-tuple of indices $(i_1, i_2,\ldots , i_n)\in I^n$.
\end{defn}
Define the  matrix $J := e_1 e_1^{t} \in \CC^{m\times m}$, where $e_1= (1,0,\ldots 0)^{t}\in \CC^m$, i.e. $J$ is an $m\times m$ matrix having the $(1,1)$-entry equal to $1$ and all the other entries equal to zero.
Let $\la \cdot\, , \cdot \ra\, : \, \CC^{m}\times \CC^{m} \rightarrow \CC$ be the usual scalar product in $\CC^{m}$ linear in the second variable.
For brevity we will denote $\llbr n \rrbr := \{1,\ldots , n\}$ for any $n \in \NN$.
\begin{defn}[Nilpotent linearization]\label{def:nilpotentlin}
  A linearization of a polynomial $p$ of the form \eqref{eq:linearization} is called \emph{nilpotent} if 
  \begin{itemize}
    \setlength\itemsep{0em}
  \item[(i)] $K_0$ is invertible;
  \item[(ii)]   $\la  e_1 , K_0^{-1} e_1\ra=1$
  \item[(iii)] The family of matrices
    \begin{equation*}
      \Big\{ \proj' K_\alpha K_0^{-1} \proj',\,  \proj'  L_\beta K_0^{-1} \proj', \,\proj'  L_\beta^* K_0^{-1} \proj' \; : \;
      \alpha \in \llbr \alpha_* \rrbr, \; \beta \in \llbr \beta_* \rrbr\Big\}
    \end{equation*}
    is nilpotent, where we set $\proj : = JK_0^{-1}$ and $\proj':=I-\proj$
  \end{itemize}
\end{defn}
One can easily see that (ii) is equivalent to $p(0,0,0)=\cstarunit$.
Indeed, let  $\la \cdot, \cdot \ra_{\Acc} \, : \, \CC^{m}\otimes \Acc \times \CC^{m}\otimes \Acc\rightarrow \Acc$ be an operator given by
  \begin{equation*}
    \la \lb , \rb \ra_{\Acc}
    :=
    \sum_{k=1}^{m} \lb_{k}^{*} \rb_{k}
  \end{equation*}
  where $\rb, \lb \in \CC^{m}\otimes \Acc$, $\rb = \sum_{k=1}^{m} e_k \otimes \rb_{k}$, $\lb = \sum_{k=1}^{m} e_k \otimes \lb_{k}$ and $e_k = (\delta_{ik})_{k=1}^{m}$.
  Then by the Schur complement formula and \eqref{eq:Lin11} we have that
  \begin{equation*}
    \la e_1 \otimes \cstarunit, \Lb^{-1} e_1\otimes \cstarunit \ra_{\Acc}
    =
    p^{-1}
    .    
  \end{equation*}
If we now take $\xb=\yb=0$, then $\Lb^{-1} = K_0^{-1}\otimes \cstarunit$ and thus $\la e_1, K_0^{-1}\, e_1 \ra \cstarunit = (p(0,0,0))^{-1}$.
Shifting the polynomial by a constant, without loss of generality, we may and will  assume in the rest of the paper that the constant term of the polynomial is 1, i.e. we write $p(\xb, \yb, \yb^*)=\cstarunit- q(\xb,\yb,\yb^*)$ for some polynomial 
$q(\xb,\yb,\yb^*)$  with $q(0,0,0)=0$.
Furthermore, note that $\pi$ is a projection by (ii) and $J=e_1 e_1^{t}$, but in general it is not an orthogonal projection.

We will show in Section~\ref{sec:Nilp} that for any polynomial of the form $p = \cstarunit - q$  both  linearizations constructed in  Appendix~\ref{sec:aLinCM}  belong to the class of nilpotent linearizations.
This property will be used to obtain an a priori bound for the \emph{generalized resolvent} of the linearization, that we define below.

Denote by $\Amxm$ the set of $m\times m$ matrices  with elements from $\Acc$.
We can look at $\Lb$ as an operator on $\Amxm$ equipped with the Banach space structure from $\mathscr{A}$.
For any $z\in\CC_+$  we will consider the \emph{generalized resolvent} of $\Lb$ defined  as $(\Lb-zJ \otimes \cstarunit)^{-1}$.

From the Schur complement formula
\begin{equation}\label{eq:schur}
  \big\la e_1 \otimes \cstarunit,  (\Lb-zJ \otimes \cstarunit )^{-1} e_1\otimes \cstarunit\big\ra_{\Acc}
  =
  \big(\lambda - z{\cstarunit} - \ell^*\widehat{\Lb}^{\,-1}\ell\big)^{-1}
  =
  (p-z{\cstarunit})^{-1}
\end{equation}
i.e.    the $(1,1)$-component of the generalized resolvent is the resolvent of $p$, viewed as an element of  $\mathscr{A}$.
In particular, if we take $\Acc$ to be $\CC^{N\times N}$, then the resolvent of a polynomial $p$ of matrices of size $N\times N$ is given by the {upper left $N\times N$ block} of the generalized resolvent of the corresponding linearization.

In Section~\ref{sec:atrivialbound} we show that generalized resolvent of a nilpotent linearization is well defined for all $z\in \CC_{+}$.
More precisely, define a norm $\|\, \cdot \,  \|$ on $\Amxm$  by 
\begin{equation*}
  \| \Rb \|
  :=
  \max_{1\leq k, l \leq m} \| \Rb_{kl} \|_{\Acc}
  ,
\end{equation*}
where $\Rb = \sum_{k,l=1}^{m} E_{kl}\otimes \Rb_{kl}$ and $E_{kl} = (\delta_{ik}\delta_{lj})_{1\leq i,j\leq m}$ is the standard basis in $\CC^{m\times m}$.
Then the following lemma holds.
\begin{lem} \label{lem:trivBound}  Let $q\in \CC\la \xb,\yb,\yb^* \ra$ be a self-adjoint polynomial with $q(0,0,0)=0$.
  Let $\Lb\in \Amxm$ be a nilpotent linearization of $\cstarunit - q(\xb,\yb, \yb^*)$
  Then there exist $C_1>0$ and $n_1\in \NN$, depending on $\Lb$, such that for all $z\in \CC_+$
  \begin{equation} \label{eq:trivResBound}
    \| ( \Lb -zJ\otimes \cstarunit)^{-1}\|
    \leq
    C_1 \bigg(1+\frac{1}{\Im z}\bigg)\, \big(1+\max_{1\leq \alpha\leq \alpha_*}\|x_{\alpha}\|_{\Acc}^{n_1}+\max_{1\leq \beta \leq \beta_*}\|y_{\beta}\|_{\Acc}^{n_1}\big)
    .
  \end{equation}
\end{lem}

Suppose now that we have a nilpotent linearization $\Lb$ in the form \eqref{eq:linearization}.
Define the linear map $\SuOp:\CC^{m\times m}\rightarrow \CC^{m\times m}$ by
\begin{equation}\label{eq:SuOp}
  \SuOp[R]
  =
  \sum_{\alpha=1}^{\alpha_*} K_\alpha R K_\alpha +\sum_{\beta=1}^{\beta_*} (L_{\beta} R L^*_{\beta} +L^*_{\beta} R L_{\beta}) 
  .
\end{equation}
For any $z\in \CC_+$ (spectral parameter) we consider the equation 
\begin{equation}
  \label{eq:MDE}
  - M^{\,-1} 
  =
  z J - K_0 + \SuOp[M]
\end{equation}
for the unknown matrix $M\in \CC^{m\times m}$.
We will always consider solutions with the side condition that $\Im M\ge 0$ where $\Im M = \frac{1}{2i}(M-M^*) $.
We call equation \eqref{eq:MDE} the \emph{Dyson equation for linearization (DEL)}.

Note that \eqref{eq:MDE}  is very similar to the matrix Dyson equations (MDE) extensively studied in
 the literature in connection with  large random matrices (see e.g. \cite{HeltRaFaSpei07} and \cite{AjanErdoKrug19}).
Their solutions typically give the deterministic part of the resolvent of a random matrix.
The main difference between \eqref{eq:MDE} and the MDE in \cite{AjanErdoKrug19} is that instead of the identity matrix, the spectral parameter $z$ 
appears with a coefficient matrix $J$ of smaller rank.
This makes \eqref{eq:MDE} much harder to analyse, in particular basic boundedness and stability properties do not follow directly from the structure of \eqref{eq:MDE} alone.
Nevertheless, the fact that \eqref{eq:MDE} comes from the linearization of a polynomial, especially that it is nilpotent still ensures its good properties. 

 For any  matrix  $R\in \CC^{m\times m}$ we denote by $\| R\|$ the operator norm induced by the Euclidean norm in $\CC^{m}$. 
The next lemma states the existence and uniqueness of the solution to \eqref{eq:MDE}, in particular
we may denote the solution $M=M(z)$, indicating its dependence on the spectral parameter.

\begin{lem}[Existence and uniqueness of solution of DEL]\label{lem:MDEexistence}
  Let $\Lb$ be a nilpotent linearization of the self-adjoint polynomial $\cstarunit - q(\xb, \yb,\yb^*)$ with $q(0,0,0)=0$  and let $\SuOp:\CC^{m\times m}\rightarrow \CC^{m\times m}$ be defined as in \eqref{eq:SuOp}.
  There exists a matrix-valued function $\GSol \,: \, \CC_{+}\rightarrow \CC^{m\times m}$ such that for all $z\in \CC_+$
  \begin{itemize}
    \setlength\itemsep{0em}
  \item[(i)] $\| \GSol (z)\| \leq C \, ( 1+1/\Im z)$ for some $C > 0$ independent of $z$;   
  \item[(ii)] $\GSol (z)$ depends analytically on $z$;
  \item[(iii)] $\Im \GSol (z) \geq 0$;
  \item[(iv)] $\GSol (z)$ satisfies the DEL \eqref{eq:MDE}.
  \end{itemize}
  This function is the unique solution of \eqref{eq:MDE} in the class of matrix-valued functions with $\Im M(z)\ge 0$  that are  analytic in the upper half-plane, i.e. if $M'\,: \, \CC_+ \rightarrow \CC^{m\times m}$ and $M'$ satisfies $(i)-(iv)$, then $M'=\GSol$.
\end{lem}
Lemma~\ref{lem:MDEexistence} will be proven in Section~\ref{sec:MDE}. In the rest of the paper, $\GSol=\GSol(z)$  will always denote the 
unique solution to \eqref{eq:MDE} obtained in Lemma~\ref{lem:MDEexistence}.

\begin{lem}[Stieltjes transform representation] \label{lem:StiltRep}
Let $\GSol(z) $ be the unique solution to DEL \eqref{eq:MDE} constructed in Lemma~\ref{lem:MDEexistence}.
  We then have the following: 
  \begin{itemize}
    \setlength\itemsep{0em}
  \item[(i)] For any $z\in \CC_+$
    \begin{equation}\label{eq:matSt}
      \GSol(z)
      =
      \GSol^{\infty}+\int_{\RR} \frac{V(dx)}{x-z}
      ,
    \end{equation}
    where $\GSol^{\infty} \in \CC^{m\times m}$ and $V(dx)$ is a (positive semidefinite)  matrix-valued measure on $\RR$ with compact support;
  \item[(ii)] For almost every $x \in \RR$ there exists the limit $\lim_{y \rightarrow 0_{ + }} \pi^{-1} \Im \GSol(x+\imunit y) = V(x) \in \CC^{m\times m}$; if the
   limit is finite on some interval $I\subset \RR$ everywhere, then  $V(dx)$ is absolutely continuous on $I$ and $V(dx)=V(x)dx$; 
  \item[(iii)]  There exists $C >0$ such that for any $z\in \CC_+$
    \begin{equation*}
      \Tr \Im \GSol(z)
      \leq
      C \la e_1, \Im \GSol(z) \, e_1 \ra
      .
    \end{equation*}
    In particular, we have that $\mathrm{supp} (V_{11}) = \mathrm{supp} (\Tr V)$.
  \end{itemize}
\end{lem}
This lemma will be proven in Section~\ref{sec:MDE}.

\subsection{Polynomials  and linearization of random matrices} \label{sec:RM}

In this section we specialize the setup from Section~\ref{sec:linDyson} to the matrix setup, i.e. to the case when $\Acc=\CC^{N\times N}$ for some $N\in \NN$ equipped with the usual matrix operator norm, induced by the Euclidean norm on $\CC^N$, and Hermitian conjugation to define the $C^*$-algebra structure.
To indicate this special case in the notation, instead of $x_1, x_2, \ldots y_1, y_2,  \ldots$ we will use capital letters, $X_1, X_2, \ldots$ and $Y_1, Y_2,\ldots$ for the $N\times N$ matrices.
Moreover, we assume that these matrices are random and independent.
The self-adjoint matrices $X_\alpha$ will be Wigner-type matrices, i.e. they have independent elements up to Hermitian symmetry,  while the matrices $Y_\beta$ will have independent entries without any restriction. We assume the matrix elements are centered and their variances are $1/N$. We collect these assumptions in the following list:

\begin{ass}
  Let $\Xb^{(N)}:=\{X_{\alpha}^{(N)}, \alpha \in \llbr \alpha_* \rrbr\}$ and $\Yb^{(N)}:=\{Y_{\beta}^{(N)}, \beta \in \llbr \beta_* \rrbr\}$ be two families of $N\times N$ random matrices such that
  \begin{itemize}
    \setlength\itemsep{0em}
  \item[(\textbf{H1})] the joint family {$\Xb^{(N)}\cup \Yb^{(N)}$ is independent;}
  \item[(\textbf{H2})] $X_{\alpha}^{(N)}$ are Hermitian random matrices having {independent} centered entries with variance $N^{-1}$;
  \item[(\textbf{H3})] $Y_{\beta}^{(N)}$ are (non-Hermitian) random matrices having {independent} centered entries with variance $N^{-1}$;
  \item[(\textbf{H4})] entries of $X_{\alpha}^{(N)}$ and $Y_{\beta}^{(N)}$ satisfy the moment bounds 
    \begin{equation*}
      \max_{i,j\in \llbr N \rrbr} \Big(\max_{\alpha\in \llbr \alpha_* \rrbr} \ME{ |\sqrt{N} X_{\alpha}^{(N)}(i,j)|^{p}}+ \max_{\beta\in \llbr \beta_* \rrbr} \ME{ |\sqrt{N} Y_{\beta}^{(N)}(i,j)|^{p} }\Big)
      \leq
      C_{p}
      .
    \end{equation*}
  \end{itemize}  
\end{ass}

Another set of assumptions concerns the properties of the solution of the Dyson equation  for linearization \eqref{eq:MDE}.
To this end we introduce the notions of the \emph{$\kappa$-bulk} and the \emph{stability operator}, which plays a crucial role in the analysis of the stability of the solution of \eqref{eq:MDE}.

\begin{defn}[Density of states]\label{def:ds} Let $\GSol(z)$ denote the unique solution of the DEL \eqref{eq:MDE}  given    in Lemma~\ref{lem:MDEexistence}.
   Define the function $\rho : \RR \rightarrow [0, +\infty]$
  \begin{equation*}
    \rho (E)
    :=
    \lim_{\eta\rightarrow 0_{ + }}\frac{1}{\pi} \la e_1, \Im \GSol(E+\imunit \eta)\, e_1 \ra
      ,
  \end{equation*}
  where the limit exists due to Lemma~\ref{lem:StiltRep} almost everywhere.  If the limit does not exist at $E$, we set $\rho(E)=\infty$ for convenience, to make the definition unambiguous. 
  We will refer to $\rho$  as the (absolutely continuous part of the) \emph{density of states} of $p$.
\end{defn}
 It will follow from the proof of Lemma~\ref{lem:MDEexistence} (see  \eqref{eq:41} and \eqref{eq:45} below) that $\rho (E)$ does not depend on the choice of linearization. 

\begin{defn}[Bulk, $\epsB$-bulk]\label{def:Bkappa}
  We say that $E\in\RR$ belongs to the \emph{bulk} if $0<\rho(E)<\infty$.
  For any $\epsB>0$ we define the set $B_{\epsB}:=\{E\in\RR \, : \, \epsB<\rho(E)<\epsB^{-1} \}$, which we will call the $\epsB$-\emph{bulk}.
\end{defn}
 We remark that Definition~\ref{def:ds} slightly differs from the standard definition used for the matrix Dyson equation in \cite{AjanErdoKrug19},
where the density of states was defined via the trace of $\Im \GSol$ as $\widetilde \rho (E) :=
    \lim_{\eta\rightarrow 0_{ +}}\frac{1}{\pi m} \Tr \Im \GSol(E+\imunit \eta)$ and  not only its (1,1)-component. The current definition is justified since 
  our main object is the polynomial $p$ and not its linearization $\Lb$.
Note, that it follows from $(iii)$ in Lemma~\ref{lem:StiltRep} that $\rho(E)$ and $\widetilde\rho(E)$ are comparable,
i.e.,  \textit{a posteriori}  the bulk could have  been defined using  $\rho$ instead of $\widetilde{\rho}$.

From now on we fix $\epsB>0$.

\begin{defn}[Stability operator]
  Let $\SuOp$ be defined as in \eqref{eq:SuOp} and $M$ obtained in Lemma~\ref{lem:MDEexistence}.
  Then the operator 
  \begin{equation*}
    \GStOp: \CC^{m\times m} \rightarrow \CC^{m\times m}
    , \quad
    \GStOp\,[R]
     :=
    R- \GSol \, \SuOp[R]\,\GSol 
  \end{equation*}
  is called the \emph{stability operator} corresponding to the DEL \eqref{eq:MDE}.
\end{defn}

\begin{ass}
  There exists a constant $C_3$, depending only on $\epsB$ and the polynomial $p$, such that for any $z\in \CC_+$ with $\Re z \in B_{\epsB}$ and $0 < \Im z < \infty$ we have
  \begin{itemize}
    \setlength\itemsep{0em}
  \item[(\textbf{M1})] $\|\GSol(z)\|\leq C_3$; 
  \item[(\textbf{M2})] $\|\GStOp^{-1}(z)\|\leq C_3$.
  \end{itemize}  
\end{ass}

The local law is formulated using the following the notion of \emph{stochastic domination}.
\begin{defn}[Stochastic domination]
  Let $\mathscr{D}\subset \CC$ and let $(\Phi_w^{(N)})_{N\in \NN}$ and $(\Psi_w^{(N)})_{N\in \NN}$, $w\in \mathscr{D}$, be two sequences of nonnegative random variables.
  Then we say that $\Phi$ is stochastically dominated by $\Psi$ uniformly on $\mathscr{D}$  if for all $\epsE, D > 0$ there exists $C(\epsE, D)>0$ such that for all $N\in \NN$
  \begin{equation*}
    \Pr{ \Phi_w^{(N)} \geq N^{\epsE} \Psi_w^{(N)}}
    \leq
    \frac{C(\epsE, D)}{N^{D}}
  \end{equation*}
  with $C(\epsE, D)$ independent of $N$ and $w$.
  In this case we write $\Phi \prec \Psi$. 
\end{defn}

We are now ready to state our main result.
\begin{thm}[Local law for polynomials] \label{thm:mainPoly}
   Let $p\in \CC \la \xb, \yb, \yb^* \ra$ be a self-adjoint polynomial with $p(0,0,0)=\cstarunit$  and let $\Lb$ be a nilpotent linearization of $p$ be defined as in \eqref{eq:linearization}.
  Let $\GSol(z)$ be a solution of the corresponding \emph{DEL} \eqref{eq:MDE} constructed as in Lemma~\ref{lem:MDEexistence}.
  Suppose that the families of random matrices $\Xb^{(N)}, \Yb^{(N)}$ satisfy conditions \textup{(\textbf{H1})-(\textbf{H4})} and that $\GSol(z)$ satisfies \textup{(\textbf{M1})-(\textbf{M2})}  for some fixed $\epsB>0$. 
  Then  the local law holds for $p(\Xb^{(N)},\Yb^{(N)}, [\Yb^{(N)}]^*)$ in the $\epsB$-bulk up to the optimal scale, i.e., for any $\gamma>0$
  \begin{equation}
    \label{eq:locallawpoly}
    \max_{i,j\in\llbr N \rrbr} \|\PRes_{ij}(z)-\la e_1, \GSol(z)\, e_1 \ra \delta_{ij}\|
    \prec
    \sqrt{\frac{1}{N\Im z}}
    ,\qquad
    \Big\| \frac{1}{N} \sum_{i=1}^{N}\PRes_{ii}(z)-\la e_1, \GSol(z)\, e_1 \ra \Big\|
    \prec
    \frac{1}{N\Im z}
  \end{equation}
  uniformly for $z\in D_{\kappa, \gamma}$ with $D_{\kappa, \gamma}:= \{z\in \CC\; : \;\Re z\in B_{\epsB}, \; N^{-1+\gamma} \leq \Im z \leq 1\}$, where $\PRes(z)$ is the resolvent matrix of the polynomial 
  \begin{equation*}
    \PRes(z)
    :=
    \Big(p(\Xb^{(N)},\Yb^{(N)}, [\Yb^{(N)}]^*)-z\otimes I_N\Big)^{-1}
    .
  \end{equation*}
\end{thm}
 Note that the typical distance between two adjacent eigenvalues in the bulk is of order $N^{-1}$. 
Thus the exponent  in the bound $\Im z \geq N^{-1+\gamma}$  is the lowest possible 
that allows for a deterministic limit of the resolvent.
In \eqref{eq:locallawpoly} we formulated the local law in the entrywise
 and in the tracial sense, but it is easy to extend the first result to a more general \emph{anisotropic sense}
 that approximates  $\langle \ub, g(z) \vb\rangle$  for any deterministic vectors $\ub, \vb\in \CC^N$
 by adapting the method from \cite[Section 7]{BloeErdoKnowYauYin14} or \cite[Section~6.1]{AjanErdoKrug17} to Kronecker random matrices in the spirit of \cite{AltErdoKrugNemi_Kronecker}.

We now comment on the assumptions \textup{(\textbf{M1})-(\textbf{M2})}. We expect that these hold for an
appropriate linearization for any self-adjoint polynomial, but this remains an open question in full generality.
 However, in Section~\ref{sec:example} we prove \textup{(\textbf{M1})-(\textbf{M2})} for a general homogeneous polynomial of degree two in Wigner matrices.
For other polynomials we remark that these two assumptions can be  checked numerically since they 
require the solution $M(z)$ of the Dyson equation for a linearization \eqref{eq:MDE} that can be computed by an effective fixed point iteration.
The numerics can be speeded up by reducing the  dimension of the \emph{DEL}, e.g.
by  considering the minimal linearization instead of the standard one, see
Appendix~\ref{sec:aLinComp} for some examples.

 Local laws provide information that can be used to estimate with relatively high precision the  locations of individual eigenvalues of the corresponding random matrix, as well as to show the delocalization of its eigenvectors.
These results have been obtained many times in the literature, therefore we state them without proofs and refer the interested reader to, e.g.,
\cite[Section~5]{AjanErdoKrug17}.
\begin{cor}[Bulk rigidity]
  Let $\lambda_i$, $1\leq i\leq N$, be the eigenvalues of $p(\Xb^{(N)},\Yb^{(N)},[\Yb^{(N)}]^*)$ in the increasing order.
  For each $E\in B_{\epsB}$ denote by $\iota(E)$ the index of the eigenvalue that is typically close to $E$, i.e.,
  \begin{equation}\label{eq:iota}
    \iota(E)
    :=
    \Big\lceil N \int_{-\infty}^{E} \rho(dx) \Big\rceil
    .
  \end{equation}
  Then
  \begin{equation*}
    \sup \{ |\lambda_{\iota(E)}-E| \, : \, E\in B_{\epsB}\}
    \prec
    \frac{1}{N}
    .
  \end{equation*}
\end{cor}
\begin{cor}[Delocalization of bulk eigenvectors]
  For $1\leq i\leq N$ denote by  $\mathbf{u}_i\in \CC^{N}$ the normalized eigenvector of $p(\Xb^{(N)},\Yb^{(N)},[\Yb^{(N)}]^*)$ that corresponds to the eigenvalue $\lambda_i$.
  Then for any deterministic unit vector $\mathbf{b}\in \CC^{N}$ and $E\in B_{\epsB}$ we have
    \begin{equation*}
    |\mathbf{b} \cdot \mathbf{u}_{\iota(E)}|
    \prec
    \frac{1}{\sqrt{N}}
    ,
  \end{equation*}
  where $\iota(E)$ is defined as in \eqref{eq:iota}.
\end{cor}

Although the main focus  of this paper is the local law and its consequences, we remark that our method also gives an optimal $1/N$ speed of convergence 
of the empirical spectral distribution of any self-adjoint polynomial $p(\Xb^{(N)},\Yb^{(N)}, [\Yb^{(N)}]^*)$ to its limiting density on the global scale. More precisely, we have the following:

\begin{pr}[Speed of convergence]\label{pr:speed}
   Let $p\in \CC \la \xb, \yb, \yb^* \ra$ be a self-adjoint polynomial with $p(0,0,0)=\cstarunit$   and let $\rho$ be the density of states.
     Suppose that the families of random matrices $\Xb^{(N)}, \Yb^{(N)}$ satisfy conditions \textup{(\textbf{H1})-(\textbf{H4})}.
 Let $\lambda_1, \ldots, \lambda_N$ be the eigenvalues of $p(\Xb^{(N)},\Yb^{(N)},[\Yb^{(N)}]^*)$ and let $f$ be a smooth function on $\RR$. Then
 \begin{equation}\label{fp}
    \bigg| \frac{1}{N} \sum_{i=1}^N f(\lambda_i) - \int_\RR f(x)\rho(dx)\bigg| \prec \frac{1}{N}.
  \end{equation}
 In particular, we have
 \begin{equation}\label{pp}
    \bigg| \frac{1}{N} \Tr p(\Xb^{(N)},\Yb^{(N)},[\Yb^{(N)}]^*)  - \int_\RR x\rho(dx)\bigg| \prec \frac{1}{N}.
 \end{equation}
  \end{pr}
Note that  this result does not assume the  conditions \textup{(\textbf{M1})-(\textbf{M2})}. In fact,  
\eqref{pp} shows that the speed of convergence in the customary definition of asymptotic freeness of
the random variables $(\Xb^{(N)},\Yb^{(N)},[\Yb^{(N)}]^*)$ is of order $1/N$.

 In the rest of the paper, whenever this does not cause any confusion, we will suppress the $N$-dependence in $\Xb$, $\Yb$ and other $N$-dependent objects.

\section{Linearizations: nilpotency and a priori bound}
\label{sec:aboutlinearizations}

In this section we prove that the linearizations  of polynomials   constructed in  Appendix~\ref{sec:aLinCM} possess some nice properties.
More precisely, we show in Lemmas \ref{lem:bignilp} and \ref{lem:anilpotency} that both the standard and the minimal  linearizations  are nilpotent, and then, in Section~\ref{sec:atrivialbound}, we prove that the bound \eqref{eq:trivResBound} holds for the generalized resolvents of any nilpotent linearization. 

 Note, that in  Appendix~\ref{sec:aLinCM} we consider linearizations of noncommutative polynomials in \emph{self-adjoint} variables only.
We start this section with a short remark explaining why this is indeed enough. 

Define the real and imaginary parts of an element $a\in\Acc$ as
  \begin{equation*}
    \Re a
    :=
    \frac{a+a^*}{2}
    ,\quad
    \Im a
    :=
    \frac{a-a^*}{2\,\imunit}
  \end{equation*}
  so that $\Re a$ and $\Im a$ are self-adjoint and $a=\Re a +\imunit \Im a$.
  Then \eqref{eq:linearization} can be rewritten as
  \begin{equation}\label{eq:aLsum3}
    \Lb
    =
    K_0\otimes \cstarunit - \sum_{\gamma=1}^{\gamma_*} K_{\gamma}\otimes x_{\gamma}
    ,
  \end{equation}
  where $\gamma_*:= \alpha_*+2\beta_*$ and we defined for $\beta\in \llbr \beta_*\rrbr$
  \begin{equation}\label{eq:aNewK}
    x_{\alpha_*+\beta}
    :=
    \sqrt{2}\Re x_{\beta}
    ,\quad
    x_{\alpha_*+\beta_*+\beta}
    :=
    \sqrt{2}\Im x_{\beta}
    ,\quad
    K_{\alpha_*+\beta}
    :=
    \sqrt{2}\Re L_{\beta}
    ,\quad
    K_{\alpha_*+\beta_*+\beta}
    :=
    -\sqrt{2} \Im L_{\beta}
    .
  \end{equation}
  We can now use formulas \eqref{eq:linearization}, \eqref{eq:aLsum3} and \eqref{eq:aNewK} to switch between linearizations of $q\in \CC \la \xb, \yb,\yb^*\ra $ and $\tilde{q} \in \CC \la \xb, \Re \yb, \Im \yb \ra$ with $\tilde{q}(\xb,\Re \xb, \Im \xb)=q(\xb,\yb,\yb^*)$.
  Clearly $q(0,0,0)=0$ is equivalent to $\tilde q(0,0,0)=0$ which we will assume in the sequel.
  Therefore, in the current section and Section~\ref{sec:MDE}, with a slight abuse of notation, by defining $\xb = (x_{\gamma},\, \gamma \in \llbr \gamma_*\rrbr)$, it will be enough to consider only self-adjoint polynomials only of the form $\tilde q(\xb)$ with $\tilde q(0)=0$ and  linearizations $\Lb$ of $\tilde{q}(\xb)$ of the form \eqref{eq:aLsum3}.
  In Section~\ref{sec:locallaw}, we will go back to the linearization \eqref{eq:linearization}.

\subsection{Joint nilpotency} \label{sec:Nilp}

In the next lemma we show that the standard linearization constructed in  Appendix~\ref{sec:abiglinearization} is nilpotent.
\begin{lem}[Nilpotency of the standard linearization]\label{lem:bignilp}
  Let  $\tilde{q} \in \CC \la \xb  \ra$  be a self-adjoint polynomial satisfying $\tilde q(0)=0$.
  Let
  \begin{equation*}
    \Lb =     K_0\otimes \cstarunit - \sum_{\gamma=1}^{\gamma_*} K_{\gamma}\otimes x_{\gamma}
  \end{equation*}
   be a  linearization of $\cstarunit - \tilde{q}$ constructed in  Appendix~\ref{sec:abiglinearization}.
  Then $\Lb$  is  nilpotent.
\end{lem}
\begin{proof}
 First of all, note that (i) and (ii) in the definition of the nilpotent linearization follow directly from \eqref{eq:apermut}. 
Thus, in order to finish the proof we need to show that the family
\begin{equation*}
        \Big\{ \proj' K_\gamma K_0^{-1} \proj', \; : \;  \gamma \in \llbr \gamma_* \rrbr \Big\}
\end{equation*}
is nilpotent, where, we recall,
  \begin{equation*}
    \proj
    :=
    J K_0^{-1}
    ,\quad
    \proj'
    :=
    I-\proj
    .
  \end{equation*}

  From the representation \eqref{eq:apermut}  we have that $\proj = J K_0^{-1} = J= K_0^{-1}J$, hence $\proj$ commutes with $K_0^{-1}$.
  This implies that for any $\gamma\in \llbr \gamma_* \rrbr$
  \begin{equation*}
    \proj' K_{\gamma} K_0^{-1} \proj'
    =
    \proj' K_{\gamma}(\proj + \proj') K_0^{-1} \proj'
    =
    \proj' K_{\gamma} \proj' K_0^{-1} \proj'
    =
    \Theta \widehat{K}_{\gamma}  \widehat{K}_0^{-1} \Theta^{-1}
    ,
  \end{equation*}
  where  we recall the structure of $K_0$ and $ K_\gamma$, explicitly indicating  their minors after separating the  first row and column:
  \begin{equation*}
    K_0
    = \left(
      \begin{array}{c|ccc}
        1 & 0 & \cdots & 0
        \\ \hline
        0 & & &
        \\
        \vdots & & \Theta \widehat{K}_0 &
        \\
        0 & & &
      \end{array}
    \right)
    ,\quad
    K_{\gamma}
    = \left(
      \begin{array}{c|ccc}
        * &* & \cdots & *
        \\ \hline
        * & & &
        \\
        \vdots & & \Theta \widehat{K}_{\gamma} &
        \\
        *& & &
      \end{array}
    \right)
    ,
  \end{equation*}
  with $\widehat{K}_0,\widehat{K}_{\gamma}\in \{0,1\}^{(m-1)\times (m-1)}$ and $\widehat{K}_0$ being a permutation matrix.
  Stars indicate arbitrary unspecified matrix elements. 
  
  The key observation is the following relation between the location of nonzero matrix elements of $\widehat{K}_0$ and $\widehat{K}_{\gamma}$
  \begin{equation}\label{eq:alphaPerm}
    \mbox{if }\quad
    \widehat{K}_0 = (e_{\tau_1}, e_{\tau_2}, \ldots, e_{\tau_{m-1}})^{t}\quad
    \mbox{ then }\quad
    \widehat{K}_{\gamma} = (c^{\gamma}_1 e_{\tau_2}, c^{\gamma}_2 e_{\tau_3}, \ldots, c^{\gamma}_{m-2} e_{\tau_{m-1}},0)^{t}
  \end{equation}
  where $\tau$ is the permutation on $\llbr m-1 \rrbr$ determined by the permutation matrix  $\widehat{K}_0$,  $e_\tau$ is the $\tau$-th coordinate vector in $\CC^{m-1}$ and  $c_i^{\gamma} \in \{0,1\}$ are some constants.
  In other words \eqref{eq:alphaPerm} says that an entry of $\widehat{\Lb}$ may  contain $x_{\gamma}$ only if the entries just below it and on the right side contain $\cstarunit$.
  For the proof, notice that this  rule is immediate for the basic block of $\Lb$ for  monomials \eqref{eq:abasicblock}  and it  remains valid after taking the conjugate transpose or applying any of the rules (R1)-(R3).

  Next, the fact that $\widehat{K}_0$ is a symmetric permutation matrix implies that $\widehat{K}_0^{-1} = \widehat{K}_0$, which means that $\widehat{K}_0^{-1} = (e_{\tau_1}, e_{\tau_2}, \ldots, e_{\tau_{m-1}})$.
  Therefore,
  \begin{equation*}
    \widehat{K}_{\gamma}\widehat{K}_0^{-1}
    =
    (c^{\gamma}_1 e_{\tau_2}, c^{\gamma}_2 e_{\tau_3}, \ldots, c^{\gamma}_{m-2} e_{\tau_{m-1}},0)^{t} (e_{\tau_1}, e_{\tau_2}, \ldots, e_{\tau_{m-1}})
    =
    \sum_{i=1}^{m-2} c_i^{\gamma}E_{i,i+1} 
  \end{equation*}
  is strictly upper-triangular.
  A family of strictly upper-triangular matrices is obviously nilpotent.
  This finishes the proof of the lemma.
\end{proof}

\begin{lem}[Nilpotency of the minimal linearization]\label{lem:anilpotency}
  Let  $\tilde{q} \in \CC \la \xb  \ra$  be a self-adjoint polynomial satisfying $\tilde q(0)=0$.
  Let
  \begin{equation*}
    \Lb =     K_0\otimes \cstarunit - \sum_{\gamma=1}^{\gamma_*} K_{\gamma}\otimes x_{\gamma}
  \end{equation*}
  be a minimal linearization of $\cstarunit - \tilde{q}$.
  Then $\Lb$ is nilpotent. 
\end{lem}
\begin{proof}
  By \eqref{eq:aKmin}, \eqref{eq:aKB} and \eqref{eq:111} properties (i) and (ii) from the definition of the nilpotent linearization are satisfied.
   Thus it is left to show that the family of matrices
   \begin{equation*}
        \Big\{ \proj' K_\gamma K_0^{-1} \proj', \; : \;  \gamma \in \llbr \gamma_* \rrbr \Big\}
\end{equation*}
is  nilpotent. 
  Define for brevity $A_{\gamma}:=K_{\gamma}K_0^{-1}$. 
  Assume that  $\|\xb \| \leq \epsD $ for $\epsD>0$ small enough, so that
  \begin{equation}\label{eq:aTBexp0}
    \Big\la K_0^{-1}e_1 \otimes \cstarunit , \Big(I\otimes \cstarunit - \sum_{\gamma=1}^{\gamma_*} A_{\gamma} \otimes x_{\gamma}\Big)^{-1}   e_1 \otimes \cstarunit \Big\ra_{\Acc}
    =
    \frac{1}{\cstarunit - p(\xb)}
  \end{equation}
  and the objects on both sides can be expanded into a convergent geometric series.
  Using the notation 
  \begin{equation*}
    \pproj
    :=
    \proj\otimes \cstarunit
    ,\quad
    \pproj'
    :=
    I\otimes \cstarunit - \pproj
    ,
  \end{equation*}
  and defining the trace operator $ \la \,\cdot \,\ra_{\Acc} : \Amxm \rightarrow \Acc$ by
  \begin{equation*}
    \la \Rb \ra_{\Acc}
    :=
    \sum_{k=1}^{m} \Rb_{kk}
    \quad
    \mbox{ for }
    \Rb \in \Amxm
    ,
  \end{equation*}
  equality \eqref{eq:aTBexp0} can be rewritten as
  \begin{equation*}
   \Big\la \pproj \Big(I\otimes \cstarunit - \sum_{\gamma=1}^{\gamma_*} A_{\gamma} \otimes x_{\gamma}\Big)^{-1}  \Big\ra_{\Acc}
    =
    \frac{1}{\cstarunit - \tilde{q}(\xb)}
    .
  \end{equation*}

  Now using the geometric series expansion for $(I\otimes \cstarunit - \sum_{\gamma=1}^{\gamma_*} A_{\gamma} \otimes x_{\gamma})^{-1} $ we have
  \begin{align}
    \frac{1}{\cstarunit - \tilde{q}(\xb)}
    &=
      \Big\la  \pproj \, \Big(I\otimes \cstarunit + \sum_{k=1}^{\infty} \sum_{(\alpha_1, \ldots , \alpha_k)\in \llbr \gamma_*\rrbr^k} A_{\alpha_1} \cdots A_{\alpha_k}\otimes x_{\alpha_1}\cdots x_{\alpha_k} \Big) \Big\ra_{\Acc} \nonumber
    \\\label{eq:TBexp2}
    &=
      \la \proj \ra \otimes \cstarunit + \sum_{k=1}^{\infty} \sum_{(\alpha_1, \ldots , \alpha_k)\in \llbr \gamma_*\rrbr^k} \la \proj  A_{\alpha_1} \cdots A_{\alpha_k} \ra \otimes x_{\alpha_1}\cdots x_{\alpha_k}
      ,
  \end{align}
  where $\la \, \cdot \, \ra$ denotes the usual trace operator, i.e., $\la R \ra = \Tr R$ for $R\in \CC^{m\times m}$.
  Since the polynomial $\tilde{q}$ has no constant term, we can write it as $\tilde{q}(\xb) = \sum_{\beta=1}^{\infty} \tilde{q}_{\beta}(\xb)$, where 
  $\tilde{q}_\beta$ is a homogeneous polynomial of degree $\beta$. Clearly this summation is finite since $\tilde{q}_\beta\equiv 0$
  whenever $\beta$ is larger than the degree of $\tilde{q}$. In other words, 
  $\tilde{q}_1$ denotes the linear part of $\tilde{q}$,  $\tilde{q}_2$ the quadratic  part, etc. 
  Then $(\cstarunit - \tilde{q}(\xb))^{-1}$ can be expanded as 
  \begin{equation}\label{eq:TBexp3}
    \frac{1}{\cstarunit - \tilde{q}(\xb)}
    =
    \cstarunit +\sum_{\ell=1}^{\infty} \tilde{q}^{\ell}(\xb)
    =
    \cstarunit +\sum_{\ell=1}^{\infty} \sum_{\beta_1,\ldots,\beta_{\ell} =1}^{\mathrm{deg}(\tilde{q})} \tilde{q}_{\beta_1} \cdots \tilde{q}_{\beta_{\ell}}
    .
  \end{equation}
  By construction we know that $\la \proj \ra =1$.
  If we now compare \eqref{eq:TBexp2} and \eqref{eq:TBexp3} recursively degree by degree, then
  from degree one terms we get that
  \begin{equation*}
    \tilde{q}_1
    =
    \sum_{\gamma=1}^{\gamma_*} \la \proj A_{\gamma} \ra x_{\gamma}
    .
  \end{equation*}
  Similarly, from comparing the degree two terms we have
  \begin{equation*}
    \tilde{q}_2+\tilde{q}_1 \tilde{q}_1
    =
    \sum_{\alpha_1, \alpha_2 =1}^{\gamma_*} \la \proj A_{\alpha_1} A_{\alpha_2}  \ra x_{\alpha_1} x_{\alpha_2}
    ,
  \end{equation*}
  so that
  \begin{equation*}
    \tilde{q}_2
    =
    \sum_{\alpha_1, \alpha_2 =1}^{\gamma_*} \la \proj A_{\alpha_1} A_{\alpha_2}  \ra x_{\alpha_1} x_{\alpha_2} - \sum_{\alpha_1=1}^{\gamma_*} \la \proj A_{\alpha_1} \ra x_{\alpha_1} \sum_{\alpha_2=1}^{\gamma_*} \la \proj A_{\alpha_2}  \ra x_{\alpha_2}
    =
    \sum_{\alpha_1,\alpha_2=1}^{\gamma_*} \la \proj A_{\alpha_1} \proj' A_{\alpha_2}  \ra x_{\alpha_1} x_{\alpha_2}
    ,
  \end{equation*}
  where we used the following factorization property, based upon $J=e_1^* e_1$: for any $B_1, B_2 \in \CC^{m\times m}$
  \begin{equation}\label{eq:trCommu} 
    \la \proj B_{1}  \ra \la \proj B_{2}  \ra
    =
    \la e_1, K_0^{-1} \, B_1\, e_1 \ra \la e_1,  K_0^{-1}\, B_2 \, e_1 \ra
    =
    \la \proj B_1 \proj B_2 \ra
    .
  \end{equation}
  Next, from comparing the degree three terms in \eqref{eq:TBexp2} and \eqref{eq:TBexp3} we get
  \begin{equation*}
    \tilde{q}_3 + \tilde{q}_2\tilde{q}_1 + \tilde{q}_1 \tilde{q}_2 + \tilde{q}_1\tilde{q}_1 \tilde{q}_1
    =
    \sum_{\alpha_1,\alpha_2,\alpha_3=1}^{\gamma_*} \la \proj A_{\alpha_1}A_{\alpha_2} A_{\alpha_3} \ra x_{\alpha_1} x_{\alpha_2} x_{\alpha_3} 
  \end{equation*}
  and thus
  \begin{align*}
    \tilde{q}_3
    & =
      \sum_{\alpha_1,\alpha_2,\alpha_3=1}^{\gamma_*} \la \proj A_{\alpha_1}(\proj+\proj')A_{\alpha_2}(\proj+\proj') A_{\alpha_3} \ra x_{\alpha_1} x_{\alpha_2} x_{\alpha_3}
      -\sum_{\alpha_1,\alpha_2,\alpha_3=1}^{\gamma_*} \la \proj A_{\alpha_1}\proj'A_{\alpha_2} \ra \la \proj A_{\alpha_3}  \ra x_{\alpha_1} x_{\alpha_2} x_{\alpha_3}
    \\ & \quad
         - \sum_{\alpha_1,\alpha_2,\alpha_3=1}^{\gamma_*} \la \proj A_{\alpha_1}  \ra \la \proj A_{\alpha_2}\proj'A_{\alpha_3} \ra  x_{\alpha_1} x_{\alpha_2} x_{\alpha_3}
         -\sum_{\alpha_1,\alpha_2,\alpha_3=1}^{\gamma_*} \la \proj A_{\alpha_1} \ra \la \proj A_{\alpha_2} \ra \la \proj A_{\alpha_3}  \ra x_{\alpha_1} x_{\alpha_2} x_{\alpha_3}
    \\
    &=
      \sum_{\alpha_1,\alpha_2,\alpha_3=1}^{\gamma_*} \la \proj A_{\alpha_1}(\proj+\proj')A_{\alpha_2}(\proj+\proj') A_{\alpha_3} \ra x_{\alpha_1} x_{\alpha_2} x_{\alpha_3}
      -\sum_{\alpha_1,\alpha_2,\alpha_3=1}^{\gamma_*} \la \proj A_{\alpha_1}\proj'A_{\alpha_2} \proj A_{\alpha_3}  \ra x_{\alpha_1} x_{\alpha_2} x_{\alpha_3}
    \\ &\quad
         - \sum_{\alpha_1,\alpha_2,\alpha_3=1}^{\gamma_*} \la \proj A_{\alpha_1} \proj  A_{\alpha_2}\proj'A_{\alpha_3} \ra  x_{\alpha_1} x_{\alpha_2} x_{\alpha_3}
         -\sum_{\alpha_1,\alpha_2,\alpha_3=1}^{\gamma_*} \la \proj A_{\alpha_1}\proj  A_{\alpha_2}\proj  A_{\alpha_3}  \ra x_{\alpha_1} x_{\alpha_2} x_{\alpha_3}
    \\ &=
         \sum_{\alpha_1,\alpha_2,\alpha_3=1}^{\gamma_*} \la \proj A_{\alpha_1}\proj'A_{\alpha_2}\proj' A_{\alpha_3} \ra x_{\alpha_1} x_{\alpha_2} x_{\alpha_3}
         ,
  \end{align*}
  where again, similarly as for quadratic terms, we used \eqref{eq:trCommu} to change the order of multiplication and taking trace.

  Now we prove the general formula for $\tilde{q}_\ell$ by induction on the degree $\ell$.
  Suppose that for any $k < \ell$
  \begin{equation}\label{eq:aIndHyp}
    \tilde{q}_k
    =
    \sum_{\alpha_1,\ldots ,\alpha_k=1}^{\gamma_*} \la \proj A_{\alpha_1}\proj'A_{\alpha_2}\proj'\cdots A_{\alpha_{k-1}}\proj' A_{\alpha_k} \ra x_{\alpha_1} \cdots  x_{\alpha_k}
    .
  \end{equation}
  Then from comparing the degree $\ell$ terms in \eqref{eq:TBexp2} and \eqref{eq:TBexp3} we get
  \begin{align*}
    &\sum_{\alpha_1,\ldots ,\alpha_{\ell}=1}^{\gamma_*} \la  \proj A_{\alpha_1}A_{\alpha_2}\cdots A_{\alpha_{\ell-1}} A_{\alpha_{\ell}} \ra x_{\alpha_1} \cdots  x_{\alpha_{\ell}}
    \\ & \quad
         = \sum_{j_1,\ldots,j_{\ell-1}\in \{0,1\}} \sum_{\alpha_1,\ldots ,\alpha_{\ell}=1}^{\gamma_*} \la \pi A_{\alpha_1}\kappa_{j_1} A_{\alpha_2}\kappa_{j_2}\cdots A_{\alpha_{\ell-1}}\kappa_{j_{\ell-1}} A_{\alpha_{\ell}} \ra x_{\alpha_1} \cdots  x_{\alpha_{\ell}}
    \\ & \quad
         =
         \tilde{q}_{\ell}+\tilde{q}_{\ell-1}\tilde{q}_1+\tilde{q}_{\ell-2}\tilde{q}_2+\tilde{q}_{\ell-2} \tilde{q}_1 \tilde{q}_1 +\cdots + \tilde{q}_1 \tilde{q}_1\cdots \tilde{q}_1
  \end{align*}
  where $\kappa_0= \proj'$ and $\kappa_1=\proj$.
  Using the factorization property \eqref{eq:trCommu} and the induction hypothesis \eqref{eq:aIndHyp} one can see that for the terms in the last sum can be written as
  \begin{equation*}
    \tilde{q}_{i_1}\cdots \tilde{q}_{i_t}
    =
    \sum_{\alpha_1,\ldots ,\alpha_{\ell}=1}^{\gamma_*} \la \pi A_{\alpha_1}\kappa_{j_1} A_{\alpha_2}\kappa_{j_2}\cdots A_{\alpha_{\ell-1}}\kappa_{j_{\ell-1}} A_{\alpha_{\ell}} \ra x_{\alpha_1} \cdots  x_{\alpha_{\ell}}
  \end{equation*}
  with
  \begin{equation*}
    j_{s} =\left\{
      \begin{array}{ll}
        1,& s\in \{i_1,i_1+i_2,\ldots, i_1+i_2+\cdots + i_{t-1}\},
        \\
        0, &\mbox{otherwise}
      \end{array}
    \right.  
  \end{equation*}
  and $i_s<m$.
  Therefore, we deduce by induction that for all $\ell \in \NN$
  \begin{equation*}
    \tilde{q}_{\ell}
    =
    \sum_{\alpha_1,\ldots ,\alpha_{\ell}=1}^{\gamma_*} \la \proj A_{\alpha_1}\proj'A_{\alpha_2}\proj'\cdots A_{\alpha_{\ell-1}}\proj' A_{\alpha_{\ell}} \ra x_{\alpha_1} \cdots x_{\alpha_{\ell}}
    .
  \end{equation*}
  In particular, if  $\tilde{q}$ is a polynomial of degree $\ell^*$, then for any $\ell>\ell^*$
  \begin{equation}\label{eq:anilpotencytrace}
    \la \proj A_{\alpha_1}\proj'A_{\alpha_2}\proj'\cdots A_{\alpha_{\ell-1}}\proj' A_{\alpha_{\ell}} \ra
    =
    0
    .
  \end{equation}
  Since $\Lb$ is a minimal linearization and  $K_0, K_{\gamma} \in \CC^{m\times m}$, by Proposition~\ref{prop:min} we have that
  \begin{equation*}
    \mathrm{span} \Big(\bigcup_{\overline{\alpha} \in \Icc} A_{\overline{\alpha}}  e_1 \Big)
      =
      \CC^m
      ,\quad
    \mathrm{span} \Big(\bigcup_{\overline{\alpha}\in \Icc} A^*_{\overline{\alpha}} K_0^{-1} e_1 \Big)
      =
      \CC^m
      ,
    \end{equation*}
    where $\Icc$, $A_{\overline{\alpha}}$ and $A^*_{\overline{\alpha}}$ are defined in \eqref{eq:aIcc} and \eqref{eq:Ralpha}.
      For $\overline{\alpha}= (\alpha_1, \ldots, \alpha_k) \in \Icc$ denote 
  \begin{equation*}
    \tilde{r}_{\overline{\alpha}}
    :=
    A_{\alpha_1} \proj' \cdots \proj' A_{\alpha_k} e_1
    ,\quad
    \tilde{l}_{\overline{\alpha}}
    :=
    A_{\alpha_1}^* {\proj'}^* \cdots {\proj'}^* A_{\alpha_k}^* K_0^{-1}e_1
    .
  \end{equation*}
  Then we can show that in fact
  \begin{equation}\label{eq:aeq121}
    \mathrm{span} \Big(\bigcup_{\overline{\alpha}\in \Icc} \tilde{r}_{\overline{\alpha}} \Big)
      =
      \CC^m
      ,\quad
    \mathrm{span} \Big(\bigcup_{\overline{\alpha}\in \Icc} \tilde{l}_{\overline{\alpha}} \Big)
      =
      \CC^m
      .
  \end{equation}
  Indeed, using the fact that for any $(\alpha_1, \ldots, \alpha_{\ell})\in \llbr \gamma_*\rrbr^{\ell}$ and any $B\in \CC^{m\times m}$
  \begin{equation*}
    A_{\alpha_1}\proj' \cdots \proj' A_{\alpha_{k-1}} \proj B  e_1
    =
    A_{\alpha_1} \proj' \cdots  \proj' A_{\alpha_{k-1}} e_1 \la \proj B \ra
  \end{equation*}
  it can be easily seen that
  \begin{equation*}
    A_{\overline{\alpha}}  e_1 
    =
    A_{\alpha_1} (\proj+\proj') A_{\alpha_{2}}\cdots (\proj+\proj')  A_{\alpha_{\ell}} e_1
    =
    \tilde{r}_{\overline{\alpha}} + u
    ,
  \end{equation*}
  where 
  \begin{equation*}
    u
    \in
    \mathrm{span} \Big(\{\tilde{r}_{\emptyset}\}\cup \bigcup_{k=1}^{\ell-1} \bigcup_{\overline{\beta}\in \llbr \gamma_*\rrbr^{k}}  \tilde{r}_{\overline{\beta}} \Big)
    ,
  \end{equation*}
  This means that for any $\ell\in \NN$
  \begin{equation*}
    \mathrm{span} \Big(\{e_1 \}\cup \bigcup_{k=1}^{\ell} \bigcup_{\overline{\alpha}\in \llbr \gamma_*\rrbr^{k}}  A_{\overline{\alpha}} e_1 \Big)
    \subset
    \mathrm{span} \Big(\{ \tilde{r}_{\emptyset}\}\cup \bigcup_{k=1}^{\ell} \bigcup_{\overline{\alpha}\in \llbr \gamma_*\rrbr^{k}} \tilde{r}_{\overline{\alpha}} \Big)
  \end{equation*}
  which implies the first equality in \eqref{eq:aeq121}.
  The second equality can be shown similarly.

  After all these preparations, we are ready to prove the nilpotency.
  Fix $\ell>\ell^*$ and $(\gamma_1, \ldots, \gamma_{\ell})\in \llbr \gamma_*\rrbr^{\ell}$, where $\ell^*$ denotes the degree of $\tilde{q}$.
  Then for any $\overline{\alpha}, \overline{\beta} \in \Icc$ of lengths $k_{\alpha}$ and $k_{\beta}$ correspondingly, by \eqref{eq:anilpotencytrace} we have
  \begin{align*}
    &\la \tilde{l}_{\overline{\alpha}}, \proj'A_{\gamma_1}\proj'\cdots A_{\gamma_{\ell}} \proj' \tilde{r}_{\overline{\beta}}\ra
    \\
    &\quad    =
      \big\la \proj A_{\alpha_{k_\alpha}} \proj' A_{\alpha_{k_\alpha-1}} \proj' \cdots A_{\alpha_1} \proj' A_{\gamma_1} \proj' \cdots A_{\gamma_{\ell}} \proj' A_{\beta_1} \proj' A_{\beta_2} \proj' \cdots A_{\beta_{k_\beta}} \big\ra
      = 0
      ,
  \end{align*}
  which together with \eqref{eq:aeq121} implies that $\proj'A_{\gamma_1}\proj'\cdots A_{\gamma_{\ell}} \proj'=0$.
  This completes the proof of the lemma.
\end{proof}

\subsection{A priori bound}\label{sec:atrivialbound}
The  a priori bound on the generalized resolvent of any nilpotent linearization was formulated in Lemma~\ref{lem:trivBound}. Now we give its proof using the Schur complement formula.
\begin{proof}[Proof of Lemma~\ref{lem:trivBound}]
  First of all,  with the definition 
  \begin{equation}\label{eq:aTscDef}
    \Tb (z)
    :=
    \Big(I\otimes \cstarunit -z\proj\otimes \cstarunit - \sum_{\gamma=1}^{\gamma_*} A_{\gamma}\otimes  x_{\gamma}\Big)^{-1},
  \end{equation}
  we can rewrite the generalized resolvent as
  \begin{equation}\label{eq:rewrite}
    \Big( (K_0-zJ)\otimes \cstarunit - \sum_{\gamma=1}^{\gamma_*} K_\gamma\otimes  x_{\gamma} \Big)^{-1}
    =
    (K_0^{-1}\otimes \cstarunit) \,  \Tb (z)
    .
  \end{equation}
  Taking the quadratic form of this identity at $e_1$, we have
  \begin{equation*}
    \la  K_0^{-1} e_1 \otimes \cstarunit, \, \Tb (z) \,e_1 \otimes \cstarunit  \ra_{\Acc} =  
    \Big\la e_1 \otimes \cstarunit,\Big((K_0-zJ)\otimes \cstarunit - \sum_{\gamma=1}^{\gamma_*} K_{\gamma}\otimes  x_{\gamma}\Big)^{-1} e_1 \otimes \cstarunit \Big\ra_{\Acc}.
  \end{equation*}
  From the definition of the linearization  and \eqref{eq:schur}, the right hand side  is just the resolvent  $((1-z) \cstarunit - \tilde{q}(\mathbf{ x}))^{-1}$, hence
  \begin{equation}\label{eq:aSchurC110}
    \la  K_0^{-1} e_1 \otimes \cstarunit, \,\Tb (z) \,e_1 \otimes \cstarunit  \ra_{\Acc}
    =
    \frac{1}{\cstarunit - \tilde{q}(\mathbf{ x})-z\cstarunit}
    .
  \end{equation}
  After multiplying this identity by $e_1 \otimes \cstarunit$ on the left and  $(K_0^{-1}e_1)^* \otimes \cstarunit$ on the right we obtain 
  \begin{equation}\label{eq:aSchurC11}
    \pproj  \Tb  (z) \pproj
    =
    \proj \otimes \frac{1}{\cstarunit - \tilde{q}(\mathbf{ x})-z\cstarunit} 
    ,
  \end{equation}
  recalling $\pi= JK_0^{-1}$ and $J= e_1^* e_1$.

  With $\pproj=\proj\otimes \cstarunit$, and $\pproj'= I\otimes \cstarunit - \pproj$, we now define
  \begin{equation}
    \label{eq:21}
    \Sb
    :=
    \pproj' + \sum_{k=1}^{\infty} \Big(\pproj' \Big(\sum_{\gamma=1}^{\gamma_* } A_{\gamma}\otimes x_{\gamma} \Big) \pproj'\Big)^{k}
    =
    \pproj' + \sum_{k=1}^{\infty} \Big(\sum_{\gamma=1}^{\gamma_* } \proj'A_{\gamma}\proj'\otimes x_{\gamma} \Big)^{k}
  \end{equation}
  where the series are convergent, in fact finite, by the joint nilpotency of the family of matrices $\{ \proj' A_{\gamma} \proj', 1\leq \gamma \leq \gamma_*  \}$ (see Lemma~\ref{lem:anilpotency}).
  In particular, there exists a  $k^*\in \NN$ such that 
  \begin{equation}\label{eq:Sbound}
    \|\Sb\|\le C \big(1+ \max_\gamma \| x_\gamma\|^{k^*}_{\Acc}\big) <\infty.
  \end{equation}
  Notice that $\Sb$  is the inverse of  $\pproj' \Tb  (z) \pproj'$ on the range of $\pproj$, i.e. 
  \begin{equation*}
    \pproj' \Big( I\otimes \cstarunit -z\pproj - \sum_{\gamma=1}^{\gamma_* } A_{\gamma}\otimes  x_{\gamma} \Big)\, \pproj' \Sb
    =
    \Sb \, \pproj' \Big(I\otimes \cstarunit -z\pproj - \sum_{\gamma=1}^{\gamma_* } A_{\gamma}\otimes  x_{\gamma} \Big)\, \pproj'
    =
    \pproj'
    .
  \end{equation*}

  By the generalized Schur complement formula for $ \Tb  (z)  =(\pproj+\pproj') \Tb (z) \,(\pproj+\pproj')$  we have
  \begin{align}
    \pproj \, \Tb  (z) \, \pproj'
    &=
      -\Big(\proj\otimes \frac{1}{\cstarunit - \tilde{q}(\mathbf{ x})-z\cstarunit}\Big) \Big(I\otimes \cstarunit -z\pproj - \sum_{\gamma=1}^{\gamma_*} A_{\gamma}\otimes  x_{\gamma}\Big) \, \pproj' \Sb \nonumber
    \\ \label{eq:aSchurC12}
    &
      =
      -\Big(\proj\otimes\frac{1}{\cstarunit - \tilde{q}(\mathbf{ x})-z\cstarunit}\Big) \Big(\sum_{\gamma=1}^{\gamma_*} \proj A_{\gamma}\proj'\otimes  x_{\gamma}\Big) \, \Sb
      ,\\ \label{eq:aSchurC21}
    \pproj' \Tb  (z) \, \pproj
    &
      =
      -\Sb \Big(\sum_{\gamma=1}^{\gamma_*} \proj' A_{\gamma}\proj\otimes  x_{\gamma} \Big)\Big(\proj \otimes \frac{1}{\cstarunit - \tilde{q}(\mathbf{ x})-z\cstarunit} \Big) 
      ,\\\label{eq:aSchurC22}
    \pproj' \Tb (z) \pproj'
    &=
      \Sb+\Sb \Big(\sum_{\gamma=1}^{\gamma_*} \proj' A_{\gamma}\proj\otimes  x_{\gamma}\Big)\Big(\proj \otimes \frac{1}{\cstarunit - \tilde{q}(\mathbf{ x})-z\cstarunit}\Big)  \Big(\sum_{\gamma=1}^{\gamma_*} \proj A_{\gamma}\proj'\otimes  x_{\gamma}\Big) \, \Sb
      .
  \end{align}
  Since $\tilde{q}(\mathbf{ x})$ is self-adjoint, we have a bound on the inverse of $\cstarunit - \tilde{q}(\mathbf{ x})-z\cstarunit$
  \begin{equation}\label{eq:aTBsc}
    \Big\|\frac{1}{\cstarunit - \tilde{q}(\mathbf{ x})-z\cstarunit} \Big\|_{\Acc}
    \leq
    \frac{1}{\eta}
    .
  \end{equation}
  Using now \eqref{eq:aTBsc}, the boundedness of $S$ and formulas \eqref{eq:aSchurC11}-\eqref{eq:aSchurC22} it can be seen that there exists $C>0$ such that $\|T\|\leq C(1+\eta^{-1})$.
  The bound \eqref{eq:trivResBound} now follows from \eqref{eq:rewrite} and \eqref{eq:Sbound}. 
\end{proof}

\begin{rem} \label{rem:33}
  Let $\pi_{\#}\in \CC^{m\times m}$ be  an arbitrary  rank $1$ projection, $\pi_{\#}':= I_m - \pi_{\#}$, $\tilde{\pi}_{\#}:= \pi_{\#}\otimes \cstarunit$, $\tilde{\pi}_{\#}':= \pi_{\#}'\otimes \cstarunit$ and suppose that
  \begin{equation}
    \label{eq:35}
    \tilde{\pi}_{\#} (\Lb - zJ\otimes \cstarunit)^{-1} \tilde{\pi}_{\#}
    =
    \pi_{\#} \otimes \frac{1}{\cstarunit - \tilde{q}(\xb) - z\cstarunit}
    ,
  \end{equation}
  where $(\Lb - zJ\otimes \cstarunit)^{-1}$ is invertible due to \eqref{eq:trivResBound}, and the right-hand side is well-defined since $z\in\CC_{+}$ and $\tilde{q}(\xb)$ is self-adjoint.
  Using the invertibility of $\Lb - zJ\otimes \cstarunit$ , we have trivially that
  \begin{equation}
    \label{eq:36}
    (\tilde{\pi}_{\#} + \tilde{\pi}_{\#}')(\Lb - zJ\otimes \cstarunit) (\tilde{\pi}_{\#} + \tilde{\pi}_{\#}')(\Lb - zJ\otimes \cstarunit)^{-1} (\tilde{\pi}_{\#} + \tilde{\pi}_{\#}')
    =
    I_{m}\otimes \cstarunit
    ,
  \end{equation}
  and, in particular,
  \begin{align}
    \label{eq:37}
    \tilde{\pi}_{\#}'(\Lb - zJ\otimes \cstarunit) (\tilde{\pi}_{\#} + \tilde{\pi}_{\#}')(\Lb - zJ\otimes \cstarunit)^{-1} \tilde{\pi}_{\#} 
    &=
    0_{m\times m}\otimes \cstarunit
    ,
    \\ \label{eq:38}
    \tilde{\pi}_{\#}'(\Lb - zJ\otimes \cstarunit) (\tilde{\pi}_{\#} + \tilde{\pi}_{\#}')(\Lb - zJ\otimes \cstarunit)^{-1} \tilde{\pi}_{\#}'
    &=
      \tilde{\pi}_{\#}'
      .
  \end{align}
  It can be easily checked from  \eqref{eq:35} and  \eqref{eq:37}-\eqref{eq:38} that
  \begin{equation}
    \label{eq:34}
    \tilde{\pi}_{\#}' (\Lb - zJ\otimes \cstarunit)\tilde{\pi}_{\#}'\Big( (\Lb - zJ\otimes \cstarunit)^{-1} - (\Lb - zJ\otimes \cstarunit)^{-1} \big(\pi_{\#}\otimes (\cstarunit - \tilde{q}(\xb) - z\cstarunit)\big) (\Lb - zJ\otimes \cstarunit)^{-1}\Big)\tilde{\pi}_{\#}'
    =
    \tilde{\pi}_{\#}'
    ,
  \end{equation}
  so that $\tilde{\pi}_{\#}' (\Lb - zJ\otimes \cstarunit)\tilde{\pi}_{\#}'$ is invertible  on $\mathrm{Ran}\, \tilde{\pi}_{\#}'$  and, moreover,  its inverse  satisfies the bound
  \begin{equation}
    \label{eq:39}
     \big\| \big(\tilde{\pi}_{\#}' (\Lb - zJ\otimes \cstarunit)\tilde{\pi}_{\#}' \big)^{-1}\big\|
    \leq
    \big\| (\Lb - zJ\otimes \cstarunit)^{-1} \big\| \, \Big( 1 + \| \cstarunit - \tilde{q}(\xb) - z\cstarunit \|_{\Acc} \big\| (\Lb - zJ\otimes \cstarunit)^{-1} \big\| \Big)
    .
  \end{equation}
  In particular, if we take $\pi_{\#} = J$, then by the definition of the linearization and the Schur complement formula
  \begin{equation}
    \label{eq:40}
    (J\otimes \cstarunit) (\Lb - zJ\otimes \cstarunit)^{-1} (J\otimes \cstarunit)
    =
     J\otimes \frac{1}{ \cstarunit - \tilde{q}(\xb) - z\cstarunit }
    ,
  \end{equation}
  and, therefore, $\widehat{\Lb}$, defined as in \eqref{eq:Lexpr1}, is invertible and satisfies the bound \eqref{eq:39}.

  Notice also, that $\widehat{\Lb}$ is independent of $z$, so by taking $C_{2}$ to be equal to the evaluation of the bound on the right-hand side of \eqref{eq:39} at, say, $z=\imunit$, we conclude that
  \begin{equation}
    \label{eq:54}
    \big\| \widehat{\Lb}^{\, -1} \big\| \leq C_{2}
  \end{equation}
   with $C_{2}$ depending only on $\|x_{\gamma}\|_{\Acc}$, $1\leq \gamma\leq \gamma_{*}$ and the linearization $\Lb$.
\end{rem}

\section{Solution to the polynomial Dyson equation}
\label{sec:MDE}

Before starting the proof of Lemma~\ref{lem:MDEexistence} we observe that the linear map $\SuOp$ can be written using only self-adjoint matrices.
Indeed, if we define (compare with \eqref{eq:aNewK})
\begin{equation} \label{eq:symmetrized}
  K_{\alpha_*+\beta}
  :=
  \sqrt{2}\Re L_{\beta}
  ,\quad
  K_{\alpha_*+\beta_*+\beta}
  :=
  -\sqrt{2}\Im L_{\beta}
  ,\quad
  1\leq \beta\leq \beta_*
  ,
\end{equation}
then for any $R\in \CC^{m\times m}$
\begin{equation*}
  \SuOp[R]=\sum_{\alpha=1}^{\alpha_*+2\beta_*} K_\alpha R K_{\alpha}
  .
\end{equation*}
Therefore, the Dyson equation for linearization  \eqref{eq:MDE} can also be written as
\begin{equation}\label{eq:MDEothersym}
  -  M^{\,-1} 
  =
  zJ-K_0 + \sum_{\alpha=1}^{\gamma_*} K_{\alpha} M K_{\alpha}
  .
\end{equation}
where we introduced $\gamma_*:= \alpha_*+2\beta_*$ for brevity.

In the sequel we will use the following notations for comparison relations.
Let $\mathscr{D}\subset \CC$ and let $(\phi_w^{(N)})_{N\in \NN}$ and $(\psi_w^{(N)})_{N\in \NN}$, $w\in \mathscr{D}$, be two sequences of complex-valued functions on $\mathscr{D}$.
We will write $\phi_w^{(N)} \lesssim \psi_w^{(N)}$ (or simply $\phi \lesssim \psi$) if there exists $C>0$ depending only on the polynomial $p$ such that $\phi_w^{(N)} \leq C \psi_w^{(N)}$ uniformly for $w\in \mathscr{D}$ and $N\in \NN$.
If $\phi \lesssim \psi$ and $\psi \lesssim \phi$ then we will write $\phi \sim \psi$.

Also, from now on we will always denote the real and imaginary parts of the spectral parameter $z$ by $E$ and $\eta$ correspondingly, i.e., $z:=E+\imunit \eta$.
\begin{proof}[Proof of Lemma~\ref{lem:MDEexistence}]
  \emph{Existence.}
   Let  $\{\semic_1, \ldots, \semic_{\gamma_*}\}$ be a family of free semicircular variables in a $C^*$--probability space $(\Scc,\tau)$ (see  Appendix~\ref{sec:freeIntro}).
  Define
  \begin{equation*}
    \Lb_{\mathrm{sc}}
    :=
    K_0 \otimes \freeunit - \sum_{\gamma=1}^{\gamma_*} K_{\gamma}\otimes \semic_{\gamma}
    ,
  \end{equation*}
  and for $z\in \CC_{+}$ define a function $  \GSol_{\mathrm{sc}}(z) : \CC_+ \rightarrow \CC^{m\times m}$ by  
  \begin{equation}\label{eq:MscDef}
    \GSol_{\mathrm{sc}}(z)
    :=
    (\id\otimes \tau)\big(\Lb_{\mathrm{sc}}-zJ\otimes \freeunit \big)^{-1}
    .
  \end{equation}
  The subscript in $\Lb_{\mathrm{sc}}$ and $\GSol_{\mathrm{sc}}$ refers to the semicircular elements.

  We will show that the function $\GSol_{\mathrm{sc}}$ is well-defined on $\CC_{+}$ and satisfies \emph{(i)-(iv)} of  Lemma~\ref{lem:MDEexistence}.
  We now introduce some notation that will be used throughout the proof.
  Let $\proj$ and $ \proj'$ denote as before projections on $\CC^{m\times m}$ given by  $\proj = JK_0^{-1}$, $\proj'= I-\proj$, and  let $\pproj$ and $\pproj'$ be projections on $\Cmxm$ defined by
  \begin{equation*}
    \pproj
    :=
    \proj\otimes \freeunit
    ,\quad
    \pproj'
    :=
    I\otimes \freeunit - \pproj
    =
    \proj'\otimes \freeunit
    .
  \end{equation*}
  Define also the matrices $A_{\gamma}:= K_{\gamma} K_0^{-1}$, $\gamma\in\llbr \gamma_* \rrbr$.
  Notice that the  nilpotency of $\Lb$ implies that $\{ \proj' A_\gamma\proj'\}_{\gamma=1}^{\gamma^*}$ is a nilpotent family.

  We first  show that $\GSol_{\mathrm{sc}} (z) $ is well-defined and  properties $(i)$-$(iii)$ hold. To see this, we apply Lemma~\ref{lem:trivBound} with $\Acc = \Scc$ to $(\Lb_{\mathrm{sc}}-zJ\otimes \freeunit )^{-1}$.
  Then from \eqref{eq:trivResBound} (assuming only self-adjoint variables) we obtain that for any $z\in \CC_+$
  \begin{equation}\label{scbound}
    \|(\Lb_{\mathrm{sc}}-zJ\otimes \freeunit)^{-1}\|
    \lesssim
    1+\frac{1}{\eta}
    ,
  \end{equation}
  Moreover, simple computation shows that 
  \begin{equation*}
    \Im \GSol_{\mathrm{sc}}(z)= \eta \, (\id\otimes \tau) \left((\Lb_{\mathrm{sc}}-\overline{z}J\otimes \freeunit)^{-1}  (J\otimes \freeunit) ( \Lb_{\mathrm{sc}} - z J \otimes \freeunit)^{-1}\right)
    ,    
  \end{equation*}
  which yields that $\Im \GSol_{\mathrm{sc}} (z)$ is positive semi-definite.

  In order to prove that $\GSol_{\mathrm{sc}}(z)$ satisfies the \emph{DEL} \eqref{eq:MDE}, consider  its regularizations, i.e.,  a family of matrix-valued functions $\{\GSol_{\mathrm{sc},\,\imunit \epsH_{k}}\}_{k\in \NN}$ given by
  \begin{equation}
    \label{eq:1}
    \GSol_{\mathrm{sc},\,\imunit \epsH_{k}}(z)
    :=
    (\id\otimes \tau)\big(\Lb_{\mathrm{sc}}-(zJ + \imunit \epsH_{k} I)\otimes \freeunit \big)^{-1}
  \end{equation}
  with $z \in \CC_{+}$ and $\epsH_{k} = k^{-1}$.
  For any fixed $k\in\NN$, the imaginary part of $zJ + \imunit \epsH_{k} I$ is positive definite, therefore it follows from \cite[Lemma~5.4]{HaagThor05}  (see also \cite[Proposition~4.1]{Lehn99})  that the function $\GSol_{\mathrm{sc},\,\imunit \epsH_{k}} (z) $ is analytic in $\CC_{+}$ and satisfies the \emph{self-consistent} (or \emph{Matrix Dyson}) equation
  \begin{equation}
    \label{eq:8}
    - \Big[\GSol_{\mathrm{sc},\,\imunit \epsH_{k}}(z)\Big]^{-1} 
    =
    K_{0} - (zJ + \imunit \epsH_{k} I) + \Gamma[\GSol_{\mathrm{sc},\,\imunit \epsH_{k}}(z)]
    ,
  \end{equation}
  which can be viewed as a regularized version of the \emph{DEL} \eqref{eq:MDE}.
   Using the \emph{a priori} bound \eqref{scbound}, for any fixed $\eta = \Im z$ we have
  \begin{equation}
    \label{eq:18}
    |\epsH_{k}|\, \big\|(\Lb_{\mathrm{sc}}-zJ\otimes \freeunit)^{-1}\big\|
    <
    \frac{1}{2} 
  \end{equation}
   if $k\in \NN$ is large enough.
  Therefore, 
  the resolvent identity implies that for $k\in\NN$ large enough  (depending on $\eta = \Im z$) 
  \begin{equation}
    \label{eq:30}
    (\Lb_{\mathrm{sc}}-(zJ+\imunit \epsH_{k})\otimes \freeunit)^{-1}
    =
    \Big(I\otimes \freeunit-\imunit \epsH_{k} \,(\Lb_{\mathrm{sc}}-zJ\otimes \freeunit)^{-1}\Big)^{-1}(\Lb_{\mathrm{sc}}-zJ\otimes \freeunit)^{-1}
    ,
  \end{equation}
  so that by definition \eqref{eq:1} the function $\GSol_{\mathrm{sc},\,\imunit \epsH_{k}}(z)$ satisfies a  $\epsH_k$-independent bound
  \begin{equation}
    \label{eq:13}
    \| \GSol_{\mathrm{sc},\,\imunit \epsH_{k}}(z) \|
    \lesssim
    1 + \frac{1}{\eta}
    ,\quad
     \forall \, k \geq k_{0}(\eta) 
    .
  \end{equation}
  On the other hand, the trivial resolvent bound implies that
  \begin{equation}
    \label{eq:32}
    \| \GSol_{\mathrm{sc},\,\imunit \epsH_{k}}(z) \|
    \leq
    \frac{1}{\epsH_{k}}
    .
  \end{equation}
  Therefore, it is  easy to see that the family of function $\{\GSol_{\mathrm{sc},\,\imunit \epsH_{k}} : \CC_{+} \rightarrow  \CC^{m\times m}  \}_{k\in\NN}$ is locally uniformly bounded, and thus, by the Montel's theorem, is normal.
  From the a priori bounds \eqref{scbound} and \eqref{eq:13} and the resolvent identity we have that for $k\in\NN$ large enough
  \begin{equation}
    \label{eq:33}
    \|\GSol_{\mathrm{sc},\,\imunit \epsH_{k}}(z) - \GSol_{\mathrm{sc}}(z) \|
    \lesssim
    \frac{1}{k}\bigg(1 + \frac{1}{\eta}\bigg)^{2}
    ,
  \end{equation}
  which yields the pointwise limit $\lim_{k\rightarrow \infty} \GSol_{\mathrm{sc},\,\imunit \epsH_{k}}(z) = \GSol_{\mathrm{sc}}(z)$ for $z\in \CC_{+}$.
  The normality of $\{\GSol_{\mathrm{sc},\,\imunit \epsH_{k}}\}_{k\in \NN}$ implies that $\GSol_{\mathrm{sc}}(z)$ is analytic on $z\in \CC_{+}$.
  By rewriting \eqref{eq:8} as
  \begin{equation}
    \label{eq:16}
    (K_{0} - (zJ + \imunit \epsH_{k} I))\GSol_{\mathrm{sc},\,\imunit \epsH_{k}}(z) + \Gamma[\GSol_{\mathrm{sc},\,\imunit \epsH_{k}}(z)] \GSol_{\mathrm{sc},\,\imunit \epsH_{k}}(z) + 1
    =
    0
  \end{equation}
  and taking the limit $k\rightarrow \infty$, we obtain that $\GSol_{\mathrm{sc}}(z)$ satisfies the \emph{DEL} \eqref{eq:MDE}.  
  
  \emph{Uniqueness.} 
  Suppose that $M_1, M_2 \, : \CC_+ \rightarrow  \CC^{m\times m} $ are two analytic solutions of \eqref{eq:MDEothersym} satisfying $\Im M_1(z) \geq 0$ and $\Im M_2(z) \geq 0$.
  It is easy to see that both $M_{1,2}(z)$
   are solutions of \eqref{eq:MDEothersym} on $\CC_{+}$ if and only if for all $z\in \CC_{+}$ functions $M^{\mathrm{ns}}_{1,2}(z):= K_0 M_{1,2} (z)$ satisfy
    \begin{equation}\label{eq:MDEotherR1}
      M^{\mathrm{ns}}_{1,2}(z)
      =
      \frac{1}{1-z}\proj + \proj' + M^{\mathrm{ns}}_{1,2}(z) \sum_{\gamma=1}^{\gamma_*} A_{\gamma} M^{\mathrm{ns}}_{1,2}(z) A_{\gamma} \Big(\frac{1}{1-z}\proj + \proj'\Big)
      .
    \end{equation}      
  If we recursively replace $M^{\mathrm{ns}}_{1,2}(z)$ in the RHS by the expression given in the RHS of \eqref{eq:MDEotherR1}, we obtain a series which is convergent for large $z$ due to nilpotency of the linearization.
  Indeed, if we assume for simplicity that $\gamma_*=1$, then  $M^{\mathrm{ns}}_{1,2}(z)$ can be rewritten as
  \begin{equation}\label{eq:Mexpan}
    M^{\mathrm{ns}}_{1,2}(z)
    =
    \sum_{\ell=0}^{\infty} C_{\ell} \Big(\Big(\frac{1}{1-z}\proj + \proj'\Big)A_{1}\Big)^{2\ell} \Big(\frac{1}{1-z}\proj + \proj'\Big)
  \end{equation}
  where $C_{\ell}$  denotes the $\ell$th Catalan number.
  Since $C_{\ell} = \frac{1}{\ell +1}\binom{2\ell}{\ell} \le 4^{\ell}$, we conclude that the RHS of \eqref{eq:Mexpan} contains $\OO{16^{\ell}}$ products of type
  \begin{equation}\label{eq:MexpanProd}
    \sigma_{1} A_1 \cdots \sigma_{2\ell} A_1 \sigma_{2\ell +1}
  \end{equation}
  with $\sigma_{i}\in \{\frac{1}{1-z} \proj, \proj'\}$.
  Now we collect all terms of type \eqref{eq:MexpanProd} that behave asymptotically like $(1-z)^{-i}$ for some $i\in \NN$ as $|z|\rightarrow \infty$.
  By the nilpotency of $\{ \proj' A_{\gamma} \proj', 1\leq \gamma\leq \gamma_*\}$ there exists $k\in \NN$ such that $\proj' A_{\gamma_1} \proj'\cdots \proj' A_{\gamma_k} \proj'=0$.
  Therefore, the maximum $\ell$ for which there exists a product  \eqref{eq:MexpanProd} that behaves asymptotically as $(1-z)^{-i}$ is less than $k(i+1)/2$.
  This implies that if $|1-z|>2 (4\| A_1 \|)^{k}$, then the RHS in \eqref{eq:Mexpan} converges and thus $M^{\mathrm{ns}}_{1}(z)=    M^{\mathrm{ns}}_{2}(z)$ on $\CC_+$ by analiticity.
  If $\gamma_* > 1$, then it can be shown similarly that on the set $|z|> 2(4 \gamma_* \max_{1\leq \gamma \leq \gamma_*}\|A_{\gamma}\|)^{k}$ functions $M^{\mathrm{ns}}_{1}(z)$ and $M^{\mathrm{ns}}_{2}(z)$ coincide, which again implies $M^{\mathrm{ns}}_{1}(z)= M^{\mathrm{ns}}_{2}(z)$ on $\CC_+$.
 This concludes the proof that  $\GSol_{\mathrm{sc}}(z)$ defined in  \eqref{eq:MscDef} is the unique solution to the DEL \eqref{eq:MDE}, i.e. $M(z)=\GSol_{\mathrm{sc}}(z)$. 
\end{proof}

\begin{proof}[Proof of Lemma~\ref{lem:StiltRep}]
  $\GSol(z)$ is a matrix-valued Herglotz function, therefore, from \cite[(1.1)-(1.3)]{GeszTsek00} it has the following representation
  \begin{equation}\label{eq:StieltRep}
    \GSol(z)
    =
    B_1 z + B_0 + \int_{\RR} \Big( \frac{1}{x-z} - \frac{x}{1+x^2}\Big)  V(dx) 
    ,
  \end{equation}
  where $B_1, B_0 \in \CC^{m\times m}$, $B_1=\lim_{\eta\uparrow \infty}(\frac{1}{\imunit \eta} \GSol(\imunit \eta))$ and $ V(dx) $ is a matrix-valued measure satisfying
  \begin{equation*}
    \int_{\RR} \frac{ \la \cb ,  V(dx)  \,\cb \ra}{1+x^2}
    <
    \infty
    \quad
    \mbox{and}\quad
    \int_{I}  V(dx) 
    \geq 0
  \end{equation*}
  for any $\cb \in \CC^{m}$ and Borel $I \subset \RR$.
  
  By definition \eqref{eq:MscDef} and the conclusion of the existence part of the proof of  Lemma~\ref{lem:MDEexistence}, we know that $\GSol(z)$ can be written as
\begin{equation} \label{eq:41}
  \GSol(z)=\GSol_{\mathrm{sc}}(z)
  =
  (\id\otimes \tau)\big( \Lb_{\mathrm{sc}}-zJ\otimes \freeunit\big)^{-1}
  .
\end{equation}
  Similarly as in \eqref{eq:Lexpr1}, express $\Lb_{\mathrm{sc}}$ in a 2 by 2 block form
  \begin{equation}
    \label{eq:42}
    \Lb_{\mathrm{sc}}
    =
    \begin{pmatrix}
      \lambda_{\mathrm{sc}} & \ell_{\mathrm{sc}}^*
      \\
      \ell_{\mathrm{sc}} & \widehat{\Lb}_{\mathrm{sc}}
    \end{pmatrix}
  \end{equation}
  by separating its first row and column,  so that $\freeunit - \tilde{q}(\mathbf{s}) = \lambda_{\mathrm{sc}} - \ell_{\mathrm{sc}}^* \widehat{\Lb}_{\mathrm{sc}}^{\,-1}\ell_{\mathrm{sc}}$.
  Now apply to $(\Lb_{\mathrm{sc}} - zJ\otimes \freeunit)^{-1}$ the Schur complement formula with respect to its $(1,1)$ component in the form
  \begin{align}
     \label{eq:43}
        \begin{pmatrix}
      \lambda_{\mathrm{sc}} - z \freeunit & \ell_{\mathrm{sc}}^*
      \\
      \ell_{\mathrm{sc}} & \widehat{\Lb}_{\mathrm{sc}}
    \end{pmatrix}^{-1}
                               &=
    \left(
      \begin{array}{cc}
        0 & 0
        \\
        0 & \widehat{\Lb}_{\mathrm{sc}}^{\,-1}
      \end{array}
    \right)
    +
    \left(
      \begin{array}{cc}
        \freeunit & 0
        \\
        - \widehat{\Lb}_{\mathrm{sc}}^{\,-1} \ell_{\mathrm{sc}} & 0
      \end{array}
    \right)
    \left(
      \begin{array}{cc}
        \frac{1}{\freeunit - \tilde{q}(\mathbf{s}) - z \freeunit}& 0
        \\
        0 & 0                                                                
      \end{array}
    \right)
        \left(
      \begin{array}{cc}
        \freeunit & - \ell_{\mathrm{sc}}^{*} \widehat{\Lb}_{\mathrm{sc}}^{\,-1} 
        \\
         0 & 0
      \end{array}
             \right)
    \\ \label{eq:44}
                                          &=
    \left(
      \begin{array}{cc}
        0 & 0
        \\
        0 & \widehat{\Lb}_{\mathrm{sc}}^{\,-1}
      \end{array}
    \right)
    -
    \left(
      \begin{array}{cc}
        \freeunit & 0
        \\
        -\widehat{\Lb}_{\mathrm{sc}}^{\,-1} \ell_{\mathrm{sc}} & 0
      \end{array}
    \right)
    \left(
      \begin{array}{cc}
        \sum\limits_{i=0}^{\infty}\frac{(\freeunit - \tilde{q}(\mathbf{s}))^{i}}{z^{i+1}}& 0
        \\
        0 & 0                                                                
      \end{array}
    \right)
        \left(
      \begin{array}{cc}
        \freeunit & -\ell_{\mathrm{sc}}^{*} \widehat{\Lb}_{\mathrm{sc}}^{\,-1} 
        \\
         0 & 0
      \end{array}
             \right)
             ,
  \end{align}
  where $\widehat{\Lb}_{\mathrm{sc}}$ is invertible and satisfies the bound $\| \widehat{\Lb}_{\mathrm{sc}}^{\, -1}\| \leq C_{2}$, for $C_{2}>0$ depending only on the linearization $\Lb$ (see  \eqref{eq:54} in  Remark~\ref{rem:33}).
  The series in \eqref{eq:44} is clearly convergent for $|z|> \| \freeunit - \tilde{q}(\mathbf{ s})\|_{\Scc}$.
  The expansion \eqref{eq:44} together with \eqref{eq:41} immediately imply that $\| \GSol(\imunit \eta)\| \leq C$ for some $C>0$ and all $\eta>1$ large enough, from which it follows that $B_1  = 0$ in \eqref{eq:StieltRep}.
  
  By Definition~\ref{def:saLin} of self-adjoint linearizations, the submatrix $\widehat{\Lb}_{\mathrm{sc}}$ is self-adjoint, i.e., $\Im \widehat{\Lb}_{\mathrm{sc}} = 0$.
  Therefore, \eqref{eq:44} implies that
  \begin{equation}
    \label{eq:45}
    \Im\, (\Lb_{\mathrm{sc}} - zJ\otimes \freeunit)^{-1} 
    =
    -
    \left(
      \begin{array}{cc}
        \freeunit & 0
        \\
        -\widehat{\Lb}_{\mathrm{sc}}^{\,-1} \ell_{\mathrm{sc}} & 0
      \end{array}
    \right)
    \left(
      \begin{array}{cc}
        \Im  \sum\limits_{i=0}^{\infty}\frac{(\freeunit - \tilde{q}(\mathbf{s}))^{i}}{z^{i+1}}& 0
        \\
        0 & 0                                                                
      \end{array}
    \right)
        \left(
      \begin{array}{cc}
        \freeunit & -\ell_{\mathrm{sc}}^{*} \widehat{\Lb}_{\mathrm{sc}}^{\,-1} 
        \\
         0 & 0
      \end{array}
    \right)
    .
  \end{equation}
  From the properties of scalar-valued Herglotz functions (formula S1.1.9 in \cite{KacKrei74}), polarization (as in the proof of Lemma~5.3 in \cite{GeszTsek00}) and \eqref{eq:41} we obtain that
  \begin{align}
    \nonumber
    \int_{\RR}  V(dx) 
    &=
      \lim_{\eta\uparrow \infty} \eta \Im \GSol(\imunit \eta)
    \\
    &= \label{eq:46}
      (\id \otimes \tau)
          \left(
      \begin{array}{cc}
        \freeunit & 0
        \\
        -\widehat{\Lb}_{\mathrm{sc}}^{\,-1} \ell_{\mathrm{sc}} & 0
      \end{array}
    \right)
    \left(
      \begin{array}{cc}
        1 & 0
        \\
        0 & 0                                                                
      \end{array}
    \right)
        \left(
      \begin{array}{cc}
        \freeunit & -\ell_{\mathrm{sc}}^{*} \widehat{\Lb}_{\mathrm{sc}}^{\,-1} 
        \\
         0 & 0
      \end{array}
             \right)
             ,
  \end{align}
  where the second line follows from the boundedness of $\widehat{\Lb}_{\mathrm{sc}}^{\, -1}$ and the expansion in \eqref{eq:45}, for which only the term corresponding to $i=0$ does not vanish in the limit. 
  In particular, \eqref{eq:46} implies that $\| \int_{\RR} V(dx) \| < \infty$ and  $V_{11}(\RR) = 1$, i.e., that $V_{11}(dx)$ is a probability measure.

  Since $V(dx)$ is a positive semidefinite  matrix-valued   measure,  the inequality  $\| \int_{\RR} f(x) V(dx)\| \leq \max_{x\in \RR} |f(x)| \,\| \int_{\RR}V(dx)\|$ holds for all measurable functions $f:\RR \rightarrow \CC$.
  Therefore, the boundedness of $\| \int_{\RR} V(dx) \|$ implies that both integrands in \eqref{eq:StieltRep}  separately  have finite integrals. 
  By taking the limit $z\rightarrow \infty$ in \eqref{eq:StieltRep} with $B_1=0$, we obtain that
  \begin{equation*}
    \lim_{z\rightarrow \infty} \GSol(z)
    =
    B_0 - \int_{\RR} \frac{x\,  V(dx) }{1+x^2}
    .
  \end{equation*}
  On the other hand, taking the same limit in \eqref{eq:44} leads to
  \begin{equation}
    \label{eq:47}
    \lim_{z\rightarrow \infty} \GSol(z)
    =
    (\id \otimes \tau)
    \left(
      \begin{array}{cc}
        0 & 0
        \\
        0 & \widehat{\Lb}_{\mathrm{sc}}^{\, -1}
      \end{array}
      \right)
    .
  \end{equation}
  Combining the two limits above and putting $M^{\infty}:= (\id \otimes \tau)
    \left(
      \begin{array}{cc}
        0 & 0
        \\
        0 & \widehat{\Lb}_{\mathrm{sc}}^{\, -1}
      \end{array}
      \right)$ yields the representation \eqref{eq:matSt}.

      Writing now $\Im \GSol(z)$ as
  \begin{align} \nonumber
  \Im \GSol(z)
  &=
    (\id\otimes \tau) \big( \Im \big( \Lb_{\mathrm{sc}}-zJ\otimes \freeunit\big)^{-1} \big)
    \\
  &=\label{eq:25}
  (\id\otimes \tau) \Big( \eta \big( \Lb_{\mathrm{sc}}-zJ\otimes \freeunit\big)^{-1} \, (J\otimes \freeunit) \, \big( \Lb_{\mathrm{sc}}-\overline{z}J\otimes \freeunit\big)^{-1} \Big)
\end{align}
and using consequently \eqref{eq:41} and \eqref{eq:44} we get that 
\begin{equation}
  \label{eq:19}
  \| \Im \GSol(E+\imunit \eta)\|
  \lesssim
  \frac{\eta }{\big(E-\|\freeunit - \tilde{q}(\mathbf{\semic})\|_{\Scc} \big)^{2}}
  ,
\end{equation}
and thus
\begin{equation}
  \label{eq:20}
  \lim_{\eta\rightarrow 0_{+}} \Im \GSol(E+\imunit \eta) = 0
\end{equation}
if $|E| > \|\freeunit - \tilde{q}(\mathbf{\semic})\|_{\Scc}$.
Therefore, by dominated convergence,
for any $a<b<-\| \freeunit - \tilde{q}(\mathbf{\semic})\|_{\Scc}$ or $\| \freeunit - \tilde{q}(\mathbf{\semic})\|_{\Scc} < a < b$
\begin{equation}
  \label{eq:24}
  \limsup_{\eta\downarrow 0}\int_{a}^{b} \Im \GSol(E + \imunit \eta) dE
  =
  0
  .
\end{equation}
By the Stieltjes inversion formula (see, e.g., \cite[formula (1.4)]{GeszTsek00}), we conclude that $\mathrm{supp}\, V \subset [-\| \freeunit - \tilde{q}(\mathbf{\semic})\|_{\Scc},\| \freeunit - \tilde{q}(\mathbf{\semic})\|_{\Scc}]$.
This finishes the proof of part $(i)$ of Lemma~\ref{lem:StiltRep}.

Part $(ii)$ of Lemma~\ref{lem:StiltRep} follows from, e.g., \cite[Lemma~5.4]{GeszTsek00} and the Stieltjes transform representation \eqref{eq:matSt} of $\GSol(z)$.

In order to prove part $(iii)$, we compute the diagonal entries of $\Im \GSol(z)$ using  \eqref{eq:41}  and the equality \eqref{eq:43}.
More precisely, from \eqref{eq:43} we have that
\begin{equation}
  \label{eq:48}
      \Im\, (\Lb_{\mathrm{sc}} - zJ\otimes \freeunit)^{-1} 
    =
    \left(
      \begin{array}{cc}
        \freeunit & 0
        \\
        -\widehat{\Lb}_{\mathrm{sc}}^{\,-1} \ell_{\mathrm{sc}} & 0
      \end{array}
    \right)
    \left(
      \begin{array}{cc}
        \Im  \frac{1}{\freeunit - \tilde{q}(\mathbf{s}) -z \freeunit}& 0
        \\
        0 & 0                                                                
      \end{array}
    \right)
        \left(
      \begin{array}{cc}
        \freeunit & -\ell_{\mathrm{sc}}^{*} \widehat{\Lb}_{\mathrm{sc}}^{\,-1} 
        \\
         0 & 0
      \end{array}
    \right)
    .
  \end{equation}
  The polynomial $\tilde{q}$ is self-adjoint, thus
  \begin{equation}
    \label{eq:49}
    \Im  \frac{1}{\freeunit - \tilde{q}(\mathbf{s}) -z \freeunit}
    =
    \eta \frac{1}{(\freeunit - \tilde{q}(\mathbf{s}) -z \freeunit)(\freeunit - \tilde{q}(\mathbf{s}) - \overline{z} \freeunit)}
    .
  \end{equation}
  Applying now \eqref{eq:48} and \eqref{eq:49} to the free probability representation of $\GSol(z)$ from \eqref{eq:MscDef}, we get that
  \begin{equation}
    \label{eq:50}
    \la e_1, \Im \GSol(z)\, e_1 \ra
    =
    \eta \,\tau \,\Big( \frac{1}{(\freeunit - \tilde{q}(\mathbf{s}) -z \freeunit)(\freeunit - \tilde{q}(\mathbf{s}) - \overline{z} \freeunit)} \Big)
  \end{equation}
  and
  \begin{equation}
    \label{eq:51}
    \la e_{i}, \Im \GSol(z)\, e_{i} \ra
    =
    \eta \,\tau \,\bigg( \Big(\widehat{\Lb}_{\mathrm{sc}}^{\,-1}\ell_{\mathrm{sc}}\Big)_{i-1}  \frac{1}{(\freeunit - \tilde{q}(\mathbf{s}) -z \freeunit)(\freeunit - \tilde{q}(\mathbf{s}) - \overline{z} \freeunit)} \Big( \widehat{\Lb}_{\mathrm{sc}}^{\,-1} \ell_{\mathrm{sc}}\Big)_{i-1}^{*} \bigg)
  \end{equation}
  for $2\leq i \leq m$.
 In particular,  \eqref{eq:50}  shows that the imaginary part of the upper left entry of $\GSol(z)$ is independent of the linearization.

Now, applying  (G.7) from \cite{AndeGuioZeitBook} to \eqref{eq:51}, using submultiplicativity of the norm (see  \ref{sec:freeIntro}) and $\|\semic_{\alpha}\|_{\Scc}=2$, we get that for all $2\leq i\leq m$,
\begin{equation}
  \label{eq:52}
  \la e_{i}, \Im \GSol(z)\, e_{i} \ra
  \leq
  \Big\| \Big(\widehat{\Lb}_{\mathrm{sc}}^{\,-1}\ell_{\mathrm{sc}}\Big)_{i-1} \Big\|_{\Scc}^{2} \, \eta \,\tau \,\bigg(   \frac{1}{(\freeunit - \tilde{q}(\mathbf{s}) -z \freeunit)(\freeunit - \tilde{q}(\mathbf{s}) - \overline{z} \freeunit)}  \bigg)
  .
\end{equation}
Combining \eqref{eq:50}, \eqref{eq:51} and \eqref{eq:52} we end up with the following bound for $\Tr \Im \GSol(z)$
\begin{equation}
  \label{eq:53}
  \Tr \Im \GSol(z)
  \leq
  \la e_{1}, \Im \GSol(z)\, e_{1} \ra \,\bigg( 1+ \sum_{i=2}^{m} \Big\| \Big(\widehat{\Lb}_{\mathrm{sc}}^{\,-1}\ell_{\mathrm{sc}}\Big)_{i-1} \Big\|_{\Scc}^{2} \bigg)
  .
\end{equation}
  Therefore, if $\lim_{\eta \rightarrow 0}\la e_1, \Im \GSol(E+\imunit \eta) e_1 \ra = 0$  for some $E\in \RR$,  then $\lim_{\eta\rightarrow 0} \Tr\Im \GSol(E+\imunit \eta) = 0$, and similarly if $\lim_{\eta\rightarrow 0} \Tr\Im \GSol(E+\imunit \eta) = \infty$ then $\lim_{\eta \rightarrow 0}\la e_1, \Im \GSol(E+\imunit \eta) e_1 \ra = \infty$.
  We conclude that $\mathrm{supp} ( V_{11}) = \mathrm{supp} (\Tr V)$.
\end{proof}

\begin{lem}[Stability of the solution of the DEL] For fixed $\kappa>0$ and under assumptions (\textbf{M1}) and (\textbf{M2}), there exists $\epsilon>0$ such that uniformly for all $z \in \CC_+$ with $\Re z \in B_\kappa$ the following holds true:
    \begin{itemize}
    \setlength\itemsep{0em}
  \item[(i)] For any $R \in \CC^{m\times m}$ , $\|R\| < \epsilon$, the matrix equation
      \begin{equation*}
    - \GSol^{\, - 1} 
    =
    zJ-K_0 + \SuOp [ \GSol] + R
  \end{equation*}
  has a solution, which we denote by $\GSol(R)$;
  \item[(ii)] For any $R_1, R_2 \in \CC^{m\times m}$, $\|R_1\| < \epsilon$, $\|R_2\| < \epsilon$, we have
  \begin{equation}\label{eq:SolStab}
    \| M (R_1)-M (R_2)\|
    \leq
    C \|R_1 - R_2\|
    .
  \end{equation}
  \end{itemize}  
\end{lem}
\begin{proof}
  This follows easily from (\textbf{M1}) and (\textbf{M2}) (see e.g. proof of the Corollary 3.8 in \cite{AltErdoKrugNemi_Kronecker})
\end{proof}

\section{Proof of the local law}
\label{sec:locallaw}

In order to establish the local law for the polynomials we will rely heavily on the linearization technique described in the previous sections. 
More precisely, given a self-adjoint polynomial $p=p(\Xb, \Yb, \Yb^{*})$ in the variables $\Xb$, $\Yb$ and $\Yb^{*}$, we consider one of its nilpotent linearizations $\Lb$ as defined in Section~\ref{sec:linDyson}.
Its generalized resolvent  will give the necessary information on the resolvent of $p$ via \eqref{eq:schur}.
So from now on our main object of interest will be the linearized random matrix $\Hb$ defined by
\begin{equation}\label{eq:Hdef}  
  \Hb
  =
  K_0\otimes I_N - \sum_{\alpha=1}^{\alpha_*} K_\alpha\otimes X_\alpha - \sum_{\beta=1}^{\beta_*}\big( L_{\beta}\otimes Y_\beta+L^*_{\beta}\otimes  Y_\beta^{*} \big)
  .
\end{equation} 
This matrix plays the role of $\Lb$ in Section~\ref{sec:linDyson}, but we use a different letter to stress that we  are in the random matrix setup. 
We denote by $I_N$ the unit element of $\mathscr{A}=\CC^{N\times N}$.
We remark that random matrices of the form \eqref{eq:Hdef}, in particular their \emph{resolvents}, have been extensively studied in \cite{AltErdoKrugNemi_Kronecker}
where they were called   {\it Kronecker matrices}.

We will denote the \emph{generalized resolvent} of $\Hb$ by $\GRes(z):=(\Hb-zJ\otimes I_N)^{-1}$.
By $\GRes_{kl} \in \CC^{N\times N}$ and $\Gm_{ij} \in \CC^{m\times m}$ we will denote the coefficient of $\GRes$ in the standard bases of $\CC^{m \times m}$ and $\CC^{N\times N}$ correspondingly, i.e.,
\begin{equation*}
  \GRes
  =
  \sum_{k,l=1}^{m} E_{kl}\otimes \GRes_{kl}
  =
  \sum_{i,j=1}^{N} \Gm_{ij} \otimes E_{ij}
\end{equation*}
with $E_{ij}=E_{ij}^{(n)}:= (\delta_{ki}\delta_{jl})_{k,l=1}^{n}\in \CC^{n\times n}$ for corresponding $n\in \NN$.
More generally, for any $\Rb \in \CC^{m\times m}\otimes \CC^{N\times N}$ we will denote its coefficients in the standard basis of $\CC^{N\times N}$ by $P_{ij} (\Rb), 1\leq i, j \leq N$, so that
\begin{equation*}
  \Rb
  =
  \sum_{i,j=1}^{N} P_{ij}(\Rb)\otimes E_{ij}
  .
\end{equation*}
In particular, we have $P_{ij}(\GRes) = \Gm_{ij}$.
Here is our main technical result.

\begin{thm}[Local law for the linearization] \label{thm:main}
  Let $p\in \CC \la \xb, \yb, \yb^* \ra$ be a self-adjoint polynomial with $p(0,0,0)=1$  and let $\Lb$ be a nilpotent linearization of $p$ be defined as in \eqref{eq:linearization}.
  Let $\GSol(z)$ be a solution of the corresponding \emph{DEL} \eqref{eq:MDE} constructed as in Lemma~\ref{lem:MDEexistence}.
  Let $\Hb$ be defined as in \eqref{eq:Hdef}.
  Suppose that the families of random matrices $\Xb, \Yb$ satisfy conditions \textup{(\textbf{H1})-(\textbf{H4})} and that $\GSol(z)$ satisfies \textup{(\textbf{M1})-(\textbf{M2})}  for some fixed $\epsB>0$. 
  Then the local law holds for $\Hb$  in the $\epsB$-bulk up to the optimal scale, i.e.,  for any $\gamma>0$ we have 
  \begin{equation}
    \label{eq:locallaw}
    \max_{i,j\in\llbr N \rrbr} \|\Gm_{ij}(z)-\GSol(z)\delta_{ij}\|
    \prec
    \sqrt{\frac{1}{N \Im z}}
    ,\qquad
    \Big\|\frac{1}{N} \sum_{i=1}^{N}\Gm_{ii}(z)-\GSol(z)\Big\|
    \prec
    \frac{1}{N \Im z}
  \end{equation}
  uniformly for $z\in D_{\kappa, \gamma}$ with $D_{\kappa, \gamma}:= \{z\in \CC\; : \;\Re z\in B_{\epsB}, \; N^{-1+\gamma} \leq \Im z \leq 1\}$.
\end{thm}

\begin{proof}[Proof of Theorem~\ref{thm:mainPoly}]
It follows immediately from Theorem~\ref{thm:main} and the Schur complement formula \eqref{eq:schur}.  
\end{proof} 
The rest of this section is devoted to the proof of Theorem~\ref{thm:main}.
Throughout this section we will use regularizations of $\GRes$ and $\GSol$.
For $z,\omega \in \CC_{+}$ define
\begin{equation*} 
  \Ress{\omega}(z)
  =
  \big( \Hb - zJ\otimes I_N - \omega I\otimes I_N \big)^{-1}
\end{equation*}
and let $\Soll{\omega}(z)$ be the solution of the regularized DEL 
\begin{equation}\label{eq:MregA}
  - \big(\Soll{\omega}(z)\big)^{\, -1} 
  =
  zJ + \omega I - K_0 + \SuOp[\Soll{\omega}(z)]
\end{equation}
that is analytic in $z$ and $\omega$ and has positive definite imaginary part.  Note that the existence of a solution to \eqref{eq:MregA} was shown earlier \cite{Voic95,HeltRaFaSpei07} and analyticity in $\omega$ and $z$ can be inferred from this general theory as well.  
As an alternative, analyticity in $\omega$ and $z$ can also be seen directly from an application of the implicit function theorem by differentiating \eqref{eq:MregA}. This also demonstrates the role of the stability operator from assumption (\textbf{M2}) for the regularity of the solution $\Soll{\omega}(z)$.
In fact, differentiating yields
\begin{equation}\label{Mderiv}
  - \big(\Soll{\omega}(z)\big)^{\,-1}  \partial \Soll{\omega}(z)  \, \big(\Soll{\omega}(z) \big)^{\,-1} 
  =
  K + \SuOp[\partial \Soll{\omega}(z)]
  ,
\end{equation}
where either $\partial = \partial_\omega$ with $K=I$ or $\partial = \partial_z$ with $K=J$, depending on the variable we are interested in. 
After rearranging the terms, the above equation can be rewritten as
\begin{equation*}
  \GStOp_{\omega}[\partial \Soll{\omega}(z)]
  =
  -\Soll{\omega}(z)K\Soll{\omega}(z)
  ,
\end{equation*}
where
  \begin{equation*}
    \GStOp_{\omega}\,:\, \CC^{m\times m} \rightarrow \CC^{m\times m}
    ,\quad
    \GStOp_{\omega}(R)
    :=
    R- \Soll{\omega}(z)\, \Gamma[R] \, \Soll{\omega}(z)
    .
  \end{equation*}
  From \cite[Lemma~3.7]{AltErdoKrugNemi_Kronecker} and the trivial  bound $\|\Soll{\omega}(z)\|\leq (\Im \omega)^{-1}$ we have that $\| \GStOp_{\omega}^{-1}\| \leq (\Im \omega)^{-10}$.
  Using implicit function theorem we conclude that $\|\partial \Soll{\omega}(z) \| \leq (\Im \omega)^{-12}$ and that $\Soll{\omega}(z)$ is analytic in $\omega$ and  $z$. 

  The next lemma collects some properties of the regularizations $\Ress{\omega}(z)$ and $\Soll{\omega}(z)$.
  \begin{lem}\label{lem:RegBnd}
    There exists $C>0$ such that 
    \begin{itemize}
      \setlength{\itemsep}{0em}
    \item[(i)] uniformly on $E\in \RR$, $\eta > 0 $ and $\epsF \geq 0 $
      \begin{equation}
        \label{eq:GregB}
        \|\Ress{\imunit \epsF}(z)\| \leq C \, \bigg( 1+\frac{1}{\eta } \bigg)
        ;
      \end{equation}
    \item[(ii)] if \textup{(\textbf{M1})} holds, then uniformly on $E\in B_{\epsB}$, $\eta \geq 0$ and $\epsF \geq 0 $
      \begin{equation}
        \label{eq:MregB}
        \| \Soll{\imunit \epsF}(z)\|\leq C
        ,\quad
        \|(\Soll{\imunit \epsF}(z))^{-1}\| \leq C (1+|z| + \epsF)
        ;
      \end{equation}
    \item[(iii)] if additionally \textup{(\textbf{M2})} holds, then uniformly on $E\in B_{\epsB}$, $0\leq \eta \leq 1$ and $\epsF \geq 0 $
      \begin{equation}
        \label{eq:LregB}
        \| (\GStOp_{\imunit \epsF}(z))^{-1}\| \leq C
        .
      \end{equation}
    \end{itemize}
    \end{lem}
    \begin{proof}
      Firstly, by specializing Lemma~\ref{lem:trivBound} for $\Acc = \CC^{m\times m}\otimes \CC^{N\times N}$, $\xb=\Xb $, $\yb=\Yb$ and $\Lb= \Hb$, we obtain that there exists $C_1>0$ such that
      \begin{equation*}
        \| \GRes (z) \|
        \leq
        C_1 \bigg( 1+\frac{1}{\eta} \bigg)
      \end{equation*}
      By the resolvent identity, for any $E\in B_{\epsB}$, $\eta\geq 0$ and $\epsF\geq 0$ we have that
      \begin{equation*}
        \Ress{\imunit \epsF}(z)
        =
        \GRes(z) + \imunit \epsF\,\Ress{\imunit \epsF}(z)\,\GRes(z)
        ,
      \end{equation*}
      therefore from the trivial bound $\|\Ress{\imunit \epsF}(z) \| \leq \epsF^{-1}$ we obtain
      \begin{equation*}
        \|\Ress{\imunit \epsF}(z)\|
        \leq
        2\| \GRes(z)\|
        .
      \end{equation*}
      By the stability of the solution of the DEL \eqref{eq:SolStab} and (\textbf{M1}), there exists $C_2>0$ such that for any $E\in B_{\epsB}$, $\eta\geq 0$ and $0\leq \epsF \leq 1$
      \begin{equation*}
        \| \Soll{\imunit \epsF}(z)\|
        \leq
        C_2
        .
      \end{equation*}
      On the other hand, if we apply the trivial bound $\Soll{\imunit \epsF}(z) \leq \epsF^{-1}$ for $\epsF \geq 1$, we obtain that
      \begin{equation}\label{eq:MregB2}
        \|\Soll{\imunit \epsF}(z)\|
        \leq
        \max\{1, C_2\}
        =:
        C_3
      \end{equation}
      for $E\in B_{\epsB}$, $\eta\geq 0$ and $\epsF\geq 0$.
      Now, using \eqref{eq:MregA} and \eqref{eq:MregB2}, there exists $C_4>0$ such that for all $E\in B_{\epsB}$, $\eta\geq 0$ and $\epsF\geq 0$
      \begin{equation}\label{eq:MinvRegB1}
        \|(\Soll{\imunit \epsF}(z))^{-1}\|
        \leq
        C_4(1+|z| + \epsF)
        .
      \end{equation}
      To obtain \eqref{eq:LregB} note that
      \begin{equation}
        \|\GStOp_{\imunit \epsF} - \GStOp \|
        \leq
        \|\Soll{\imunit \epsF}(z)-\GSol(z)\|\, \big(\|\Soll{\imunit \epsF}(z)\| + \|\GSol(z)\| \big)\, \|\SuOp \|
        \leq
        C_5 \epsF
      \end{equation}
      for some $C_5>0$.
      Therefore by (\textbf{M2}) there exists $\epsilon_1>0$ and $C_6>0$ such that for $0\leq \epsF \leq \epsilon_1$
    \begin{equation*}
      \|(\GStOp_{\imunit \epsF})^{-1}\|
      =
      \|\GStOp^{-1}(I-(\GStOp-\GStOp_{\imunit \epsF})\GStOp^{-1})^{-1}\|
      \leq
      2\|\GStOp^{-1}\|
      \leq
      C_6
      .
    \end{equation*}
    By the definition of $\GStOp_{\imunit \epsF}(z)$ and the trivial bound $\|\Soll{\imunit \epsF}(z)\|\leq \epsF^{-1}$, there exists $\epsilon_2>0$ such that for $\epsF \geq \epsilon_2$
    \begin{equation*}
      \| (\GStOp_{\imunit \epsF})^{-1}\|
        \leq
        \frac{1}{1-\|\Soll{\imunit \epsF}(z)\|^2 \|\SuOp\|}
        \leq
        2
        .
      \end{equation*}
      Finally, by \cite[Lemma~3.7]{AltErdoKrugNemi_Kronecker}, compactness of $B_{\epsB}$, \eqref{eq:MregB2} and \eqref{eq:MinvRegB1} there exists $C_7>0$ such that for all $E\in B_{\epsB}$, $0\leq \eta \leq 1$ and $\epsilon_1 \leq \epsF \leq \epsilon_2$
      \begin{equation*}
        \| (\GStOp_{\imunit \epsF})^{-1}\|
        \leq
        \frac{\| \Soll{\imunit \epsF}(z)\|^2 \|(\Soll{\imunit \epsF}(z))^{-1}\|^{9}}{(\mathrm{dist}(\mathrm{supp} (\rho_{z}),\imunit \epsF))^{8}}
        \leq
        C_7
        ,
      \end{equation*}
      where $\rho_z(x):= \lim_{u\downarrow 0} (\pi m)^{-1}\Tr \Im \Soll{ x + \imunit u}(z)$.
To finish the proof, take $C>\max\{2,2C_1,C_3,C_4,C_6,C_7\}$.
\end{proof}

Now we state the local law for the regularized resolvent.

\begin{lem}
  \label{lem:LLgoodB}
  Uniformly for $E\in B_{\epsB}$, $0\leq \eta \leq 1$ and $\epsF\geq N^{-1+\gamma}$
  \begin{equation}\label{eq:LLgoodB}
    \max_{i,j\in\llbr N \rrbr} \|P_{ij}(\Ress{\imunit \epsF}(E+\imunit \eta))-\Soll{\imunit \epsF}(E+\imunit \eta)\, \delta_{ij}\|
    \prec
    \sqrt{\frac{1}{N\epsF}}
    .
  \end{equation}
\end{lem}
\begin{proof}
  Follows from Lemma~B.1 in \cite{AltErdoKrugNemi_Kronecker}.
  Indeed, by \eqref{eq:MregB} and \eqref{eq:LregB} for all $E\in B_{\epsB}$, $0\leq \eta \leq 1$ and $\epsF \geq N^{-1+\gamma}$ we have
  \begin{equation*}
    \max_{i,j\in\llbr N \rrbr} \|P_{ij}(\Ress{\imunit \epsF}(E+\imunit \eta))-\Soll{\imunit \epsF}(E+\imunit \eta)\, \delta_{ij}\|
    \prec
      \frac{1}{1+\epsF}\sqrt{\frac{\|\Soll{\imunit \epsF}(E+\imunit \eta)\|}{N\epsF}}+\frac{1}{(1+\epsF^2)N}+\frac{1}{(1+\epsF^2)N\epsF}
    .
  \end{equation*}
  The fact that $\Soll{\imunit \epsF}(E+\imunit \eta)$ is bounded by \eqref{eq:MregB} yields \eqref{eq:LLgoodB}.
\end{proof}
We are ready to prove the main theorem. 
\begin{proof}[Proof of Theorem~\ref{thm:main}]
  By \cite[Lemma~4.4]{AltErdoKrugNemi_Kronecker} and Lemma~\ref{lem:RegBnd}, for $E\in B_{\epsB}$, $0\leq \eta \leq 1$ and $\epsG\geq 0$
  \begin{equation*}
    \max_{1\leq i \leq N}\|P_{ii}(\Ress{\imunit \epsG}(z)) - \Soll{\imunit \epsG}(z) \| \, \chi(\Lambda^{\epsG} \leq \vartheta^{\epsG})
    \prec
    \frac{1}{\sqrt{N}}+\Lambda^{\epsG}_{\mathrm{hs}}+\| ( \Soll{\imunit \epsG}(z))^{-1}\|\, \Big(\Lambda^{\epsG}_{\mathrm{w}}\Big)^{2}
    ,
  \end{equation*}
  where $\chi(A)$ denotes the indicator function of an event $A$ and we introduced
  \begin{align*} 
  &\Lambda_{\mathrm{hs}}^{\epsG}(z)
  :=
  \frac{1}{N}\Big(\Tr \Ress{\imunit \epsG}(z)^* \Ress{\imunit \epsG}(z)\Big)^{1/2}
    ,
    \\
  &\Lambda_{\mathrm{w}}^{\epsG}(z)
  :=
  \frac{1}{\sqrt{2N}}\max_{i}\Big(\Tr P_{ii}[\Ress{\imunit \epsG}(z)^*\Ress{\imunit \epsG}(z) +\Ress{\imunit \epsG}(z)\, \Ress{\imunit \epsG}(z)^*]\Big)^{1/2},
   \\
  &
   \Lambda^{\epsG}(z)
  :=
  \max_{i,j\in\llbr N \rrbr} \|P_{ij}(\Ress{\imunit \epsG}(z))-\Soll{\imunit \epsG}(z)\delta_{ij}\|\,,
\end{align*}
   as well as 
    \begin{equation*}
    \vartheta^{\epsG}
    :=
    \frac{1}{4(\|(\GStOp_{\epsG})^{-1}\|\; \| \Soll{\imunit \epsG}(z)\| \; \|\SuOp\| + \|( \Soll{\imunit \epsG}(z))^{-1}\|)}
    .
  \end{equation*}
  To estimate $\Lambda_{\mathrm{hs}}^{\epsG}$ note that $\Lambda_{\mathrm{hs}}^{\epsG}=N^{-1}\normHS{\Ress{\imunit \epsG}(z)}$,  where for any $n\in \NN$ we denote by $\normHS{\,\cdot \,}: \CC^{n \times n} \rightarrow [0, +\infty)$ the usual Hilbert-Schmidt norm, i.e., for any $R\in \CC^{n\times n}$
  \begin{equation*}
    \normHS{R}^2
    =
    \Tr R^* R
    .
  \end{equation*}
  By the resolvent identity $\Ress{\imunit \epsG}(z)=\Ress{\imunit (\eta+\epsG)}(E)-\imunit \eta  \, \Ress{\imunit \epsG}(z)  \, ((I_m-J)\otimes I_N) \, \Ress{\imunit (\eta+\epsG)}(E)  $, therefore
    \begin{align*}
      \normHS{\Ress{\imunit \epsG}(z)}
      &\leq
       \normHS{\Ress{\imunit (\eta+\epsG)}(E) } +\eta\,\normOp{\Ress{\imunit \epsG}(z)} \, \normHS{\Ress{\imunit (\eta+\epsG)}(E)}
      \\
      &\leq
        \normHS{ \Ress{\imunit (\eta+\epsG)}(E) } + C\, \normHS{\Ress{\imunit (\eta+\epsG)}(E) }
      \\
      &\lesssim
        \normHS{\Ress{\imunit(\eta+\epsG)}(E)}
        ,
    \end{align*}
    where we used \eqref{eq:GregB} to obtain the bound $\eta\normOp{\Ress{\imunit \epsG}(z)}\leq C$ for some $C>0$ uniformly on $E\in B_{\epsB}$, $0 < \eta \leq 1$ and $\epsG \geq 0$.
      Since $\Ress{\imunit(\eta+\epsG)}(E)$ is a resolvent with spectral parameter $\imunit (\eta + \epsG)$, we can apply to it the Ward identity, which together with Lemma~\ref{lem:LLgoodB} gives
    \begin{equation}\label{eq:HSbound}
      \frac{1}{N}\normHS{\Ress{\imunit(\eta + \epsG)}(E)}
      =
      \left( \frac{ \Tr \Im \Ress{\imunit (\eta + \epsG)}(E) }{N^2(\eta + \epsG)}\right)^{1/2}
      \prec
      \sqrt{\frac{1}{N(\eta + \epsG)}}
    \end{equation}
    uniformly for $E\in B_{\epsB}$, $N^{-1+\gamma} \leq \eta \leq 1$ and $\epsG\geq 0$.
    In order to estimate $\Lambda_{\mathrm{w}}^{\epsG}$, we introduce the norm $\| \, \cdot \, \|_{\mathrm{w}} : \CC^{m\times m} \otimes \CC^{N\times N} \rightarrow [0, +\infty)$ given by
    \begin{equation*}
      \| \Rb \|_{\mathrm{w}}^2
      =
      \max_{1\leq i \leq N} \Tr P_{ii}(\Rb \Rb^*)
      .
    \end{equation*}
    One can easily see that $\Lambda_{\mathrm{w}}^{\epsG}\sim N^{-1/2}\| \Ress{\imunit \epsG}(z) \|_{\mathrm{w}}$.
    Then similarly as for $\normHS{\, \cdot \,}$,
        \begin{align*}
     \normW{\Ress{\imunit \epsG}(z)}
      &\leq
        \normW{ \Ress{\imunit (\eta+\epsG)}(E) } + C\, \normW{\Ress{\imunit (\eta+\epsG)}(E) }
      \\
      &\lesssim
        \normW{\Ress{\imunit(\eta+\epsG)}(E)}
        .
    \end{align*}
        By applying again the Ward identity and Lemma~\ref{lem:LLgoodB} we obtain that uniformly for $E\in B_{\epsB}$, $N^{-1+\gamma}\leq \eta \leq 1$ and $\epsG \geq 0$
    \begin{equation*}
      \frac{1}{N}\normW{\Ress{\imunit (\eta + \epsG)}(E)}^2
      \lesssim
      \max_{1\leq i\leq N}\frac{\Im \Tr P_{ii}(\Ress{\imunit(\eta+\epsG)}(E))}{N(\eta+\epsG)}
      \prec
      \frac{1}{N(\eta+\epsG)}
      .
    \end{equation*}
    Together with \eqref{eq:HSbound} and \eqref{eq:MregB} this implies that
    \begin{equation}\label{eq:errTbndEpsG2}
      \max_{1\leq i \leq N}\|P_{ii}(\Ress{\imunit \epsG}(z)) - \Soll{\imunit \epsG}(z) \| \,\chi(\Lambda^{\epsG} \leq \vartheta^{\epsG})
    \prec
    \frac{1}{\sqrt{N}}+\sqrt{\frac{1}{N(\eta+\epsG)}}
  \end{equation}
  for $E\in B_{\epsB}$, $N^{-1+\gamma}\leq \eta \leq 1$ and $\epsG \geq 0$.

  For any $1\leq i, j \leq N, i\neq j,$ and $z\in \CC_{+}$ we have that 
  \begin{equation}
    \label{eq:31}
    \| P_{ij}(\Ress{\imunit \epsG}(z))\| \, \chi(\Lambda^{\epsG} \leq \vartheta^{\epsG})
    \prec
    \|P_{ii}(\Ress{\imunit \epsG}(z))\|\,\Lambda_{\mathrm{w}}^{\epsG}
    .
  \end{equation}
  This statement was proven in the form $|G_{ij}|\, \chi \prec | G_{jj}|\, \chi \Lambda_{\mathrm{w}}$ in the course of the proof of \cite[Lemma~4.3]{AltErdoKrugNemi_Kronecker} (see \cite[Eq. (4.23)]{AltErdoKrugNemi_Kronecker} and the discussion after it).
   The right-hand side of \eqref{eq:31} can be bounded by $ \|\Soll{\imunit (\eta + \epsG)}(E)\| \, \normW{\Ress{\imunit(\eta+\epsG)}(E)} + \Lambda^{ \epsG} \normW{\Ress{\imunit(\eta+\epsG)}(E)}$.
    The second term can be absorbed into $\Lambda^{ \epsG}$, so by using \eqref{eq:errTbndEpsG2} and \eqref{eq:MregB} we end up with the bound
    \begin{equation*}
      \max_{1\leq i,j \leq N}\|P_{ij}(\Ress{\imunit \epsG}(z)) - \Soll{\imunit \epsG}(z)\delta_{ij} \| \, \chi(\Lambda^{\epsG} \leq \vartheta^{\epsG})
    \prec
    \frac{1}{\sqrt{N}}+\sqrt{\frac{1}{N(\eta+\epsG)}}
  \end{equation*}
  uniformly on $E\in B_{\epsB}$, $N^{-1+\gamma}\leq \eta \leq 1$ and $\epsG \geq 0$

    Since $\vartheta^{\epsG}(z) \gtrsim \epsG^{-1}$ by \eqref{eq:MregB} and \eqref{eq:LregB}, and  $\Lambda^{ \epsG}(z) \prec \epsG^{-2}$ by \cite[Lemma~4.4, (i)]{AltErdoKrugNemi_Kronecker}, we can choose $\epsG_1>0$ such that for all $E\in B_{\epsB}$
    \begin{equation}\label{eq:LetaEpsG}
      \Lambda^{ \epsG_1}(E+\imunit)
      \leq
      \vartheta^{\epsG_1}(E+\imunit)
      .
    \end{equation}
    Then by \cite[Lemma~A.2]{AjanErdoKrug17} \eqref{eq:LetaEpsG} holds not only for $\epsG=\epsG_1$, but \emph{a.w.o.p.} for all $0\leq \epsG \leq \epsG_1$.
    In particular,  we will have that \emph{a.w.o.p.}
    \begin{equation*}
      \Lambda(E+\imunit)
      \leq
      \vartheta(E+\imunit)
      .
    \end{equation*}
    On the other hand, if we take $\epsG=0$ in \eqref{eq:errTbndEpsG2} we will get that for $E\in B_{\epsB}$ and $0\leq \eta \leq 1$
      \begin{equation*}
      \max_{1\leq i \leq N}\|\Gm_{ii}(z) - \GSol(z) \| \, \chi(\Lambda \leq \vartheta)
    \prec
    \frac{1}{\sqrt{N}}+\sqrt{\frac{1}{N\eta}}
    .
  \end{equation*}
  Applying \cite[Lemma~A.2]{AjanErdoKrug17} to $\Lambda(E+\imunit \eta)$ and $\vartheta(E+\imunit \eta)$ we get that for $E\in B_{\epsB}$ and $N^{-1+\gamma} < \eta \leq 1$
    \begin{equation*}
      \Lambda(E+\imunit \eta)
      \leq
      \vartheta(E+\imunit \eta)
      ,
    \end{equation*}
    which yields the first inequality in \eqref{eq:locallaw}.

    To prove the averaged local law we will use the fluctuations averaging mechanism, proof of which in a suitable form can be found in \cite[Proposition~4.6]{AltErdoKrugNemi_Kronecker},
    see also Section~10 of  \cite{ErdoYauBook} related previous proofs.
    To this end, we introduce conditional expectation with respect to the $i$th rows and columns of the matrices $X_{\alpha}$ and $Y_{\beta}$
    \begin{equation*}
      \EE_{i}[\, \cdot \,]
      :=
      \EE \Big[\, \cdot \, \Big| \Big\{X_{\alpha}(k,l), Y_{\beta}(k,l)\, : \, \alpha\in \llbr \alpha_*\rrbr, \beta\in \llbr \beta_*\rrbr, k,l\in \llbr N \rrbr \setminus \{i\}\Big\}\Big]
      ,
    \end{equation*}
    and a family of operators
    \begin{equation*}
      \Qcc_{i}[\, \cdot \,]
      :=
      \mathrm{Id}[\, \cdot \,] - \EE_{i}[\, \cdot \, ]
      .
    \end{equation*}
    By the Schur complement formula (see e.g. \cite[Section~4.1]{AltErdoKrugNemi_Kronecker}), for all $i\in \llbr N \rrbr$ we have that
 \begin{equation}\label{eq:Gii Schur expansion}
      -  \frac{1}{\Gm_{ii}} 
      =
      zJ - P_{ii}(\Hb) + \sum_{k,l \neq i} P_{ik}(\Hb) P_{kl}(\GRes^{\{i\}}) P_{li}(\Hb)
      ,
    \end{equation}
    where
    \begin{equation*}
      \GRes^{\{i\}}(z):
      =
      \Big( K_0\otimes I_N - \sum_{\alpha=1}^{\alpha_*} K_\alpha\otimes X_\alpha^{\{i\}} - \sum_{\beta=1}^{\beta_*}\big( L_{\beta}\otimes Y_\beta^{\{i\}}+L^*_{\beta}\otimes  Y_\beta^{\{i\}*} \big) - zJ \otimes I_N \Big)^{-1}
    \end{equation*}
    and matrices $X_\alpha^{\{i\}}$ and $Y_\alpha^{\{i\}}$ are obtained from $X_\alpha$ and $Y_\alpha$, respectively, by replacing their $i$th rows and columns by zero.
    We use the $1/G_{ii}$ notation for the inverse of the $m\times m$ matrix $G_{ii}$. 
    Taking the expectation $\mathbb{E}_i$ on both sides of \eqref{eq:Gii Schur expansion}, using \eqref{eq:Hdef}, the assumptions (\textbf{H1})-(\textbf{H3}) about the distribution of $\Xb^{(N)}$ and $\Yb^{(N)}$, as well as the independence of  $\GRes^{\{i\}}$ from $P_{ik}(\Hb) $ and $P_{ki}(\Hb)$, yields the equation
\begin{equation}\label{eq:perturbed DEL}
-\frac{1}{G_{ii}} = -\mathbb{E}_i\frac{1}{G_{ii}}-\Qcc_{i}\frac{1}{\Gm_{ii}} = z J-K_0+\SuOp\,\bigg[\frac{1}{N}\sum_{j=1}^NG_{jj}\bigg]+D_i\,,
\end{equation}
for $(G_{ii})_{i=1}^N$. Here, the  random error term $D_i=D_i(z)$ takes the form
\begin{equation*}
        D_i=-   \Qcc_{i}\frac{1}{\Gm_{ii}}
        + \SuOp\,\bigg[\frac{1}{N}\sum_{j\neq i}\big  (P_{jj}\,\GRes^{\{i\}} -\Gm_{jj}\big)\bigg] 
        - \SuOp\,\bigg[\frac{\Gm_{ii}(z)}{N}\bigg]\,.
\end{equation*}
We subtract  \eqref{eq:perturbed DEL} from the DEL \eqref{eq:MDE} and multiply  the result from the left by its solution $M$ and from the right by $G_{ii}$ to see that
 the difference between $\Gm_{ii}(z)$ and $\GSol(z)$ can be written as
    \begin{align}\label{identity}
      \Gm_{ii}(z)-\GSol(z)
      &=
        \GSol(z)\,\SuOp\,\bigg[\frac{1}{N}\sum_{j=1}^{N}\big(\Gm_{jj}(z)-\GSol(z)\big)\bigg]\,\Gm_{ii}(z)
      \\
      &\quad
        -  \GSol(z) \, \Qcc_{i}\,\bigg[\frac{1}{\Gm_{ii}(z)}\bigg] \,\Gm_{ii}(z)  \nonumber
       \\
      &\quad
        + \GSol(z)\, \SuOp\, \bigg[\frac{1}{N}\sum_{j\neq i}\big  (P_{jj}(\GRes^{\{i\}}(z))  -\Gm_{jj}(z)\big)\bigg]\,\Gm_{ii}(z) \nonumber
       \\
      &\quad
        - \GSol(z)\, \SuOp\, \bigg[\frac{\Gm_{ii}(z)}{N}\bigg]\,\Gm_{ii}(z) \nonumber
        .
    \end{align}
    Note, that by using the large deviation bounds (see e.g. \cite[Lemma~4.3]{AltErdoKrugNemi_Kronecker}) we have
    \begin{equation*}
      \bigg \|\, \Qcc_{i}\bigg[\frac{1}{\Gm_{ii}(z)}\bigg]\, \bigg\|
      \prec
      \sqrt{\frac{1}{N\eta}}
      .
    \end{equation*}
    After taking the average over $i\in \llbr N \rrbr$, rearranging the terms in \eqref{identity}  and using the entry-wise local law from \eqref{eq:locallaw}, (\textbf{M1}), boundedness of $\SuOp$ and 
     the formula
    \begin{equation*}
      \Gm_{jj}(z)- P_{jj}\big(\GRes^{\{i\}}(z)\big) 
      =
      \Gm_{ji}(z)\,\frac{1}{\Gm_{ii}(z)}\,\Gm_{ij}(z)
      ,\quad
       j\neq i 
      ,
    \end{equation*}
    we obtain
    \begin{equation}\label{eq:AvLL1}
      \GStOp \,\bigg[\frac{1}{N}\sum_{i=1}^{N}\Gm_{ii}(z)-\GSol(z) \bigg]        
      =
          - \GSol(z) \, \frac{1}{N}\sum_{i=1}^{N}\Qcc_{i}\,\bigg[\frac{1}{\Gm_{ii}(z)}\bigg] \,\GSol(z) +O_{ \prec }\bigg(\frac{1}{N}+\frac{1}{N\eta}\bigg)
        ,
    \end{equation}
    where $O_{\prec}(N^{-1}+(N\eta)^{-1})$ collects terms stochastically dominated by $N^{-1} + (N\eta)^{-1}$ .
    Applying again the entry-wise local law \eqref{eq:locallaw}  and \eqref{eq:MregB}, we have that uniformly on $E\in B_{\epsB}$ and $N^{-1+\gamma} < \eta \leq 1$
    \begin{equation}\label{eq:FlAvCond}
      \max_{i,j\in \llbr N \rrbr}\bigg\| \frac{1}{\GSol(z)}\Gm_{ij}(z)-\delta_{ij}  I_m  \bigg\|
      \prec
      \sqrt{\frac{1}{N\eta}}
      \leq
      N^{-\gamma}
      .
    \end{equation}
    Inequality \eqref{eq:FlAvCond} allows us to improve a bound on the first term on the RHS of \eqref{eq:AvLL1} by using the fluctuation averaging (see \cite[Proposition~4.6]{AltErdoKrugNemi_Kronecker}), which gives that for $E\in B_{\epsB}$ and $N^{-1+\gamma} < \eta \leq 1$
    \begin{equation*}
      \frac{1}{N}\sum_{i=1}^{N}\Qcc_{i}\,\bigg[\frac{1}{\Gm_{ii}(z)}\bigg] \, \GSol(z) 
      \prec
      \frac{1}{N\eta}
      .
    \end{equation*}
    Now the boundedness of $\GStOp^{-1}$ from (\textbf{M2}) yields the second inequality in \eqref{eq:locallaw}.
    This completes the proof of Theorem~\ref{thm:main}.
\end{proof}

Finally we prove the speed of convergence in the global law:

\begin{proof}[Proof of Proposition~\ref{pr:speed}] First we prove \eqref{fp} for bounded test functions. In this case, by  the Helffer-Sj\"ostrand formula
(see, e.g.  Section 11.2 of \cite{ErdoYauBook}) it is sufficient to prove \eqref{fp} for all functions of the 
form $f(x)= (x-z)^{-1}$ with any fixed $z\in \CC_+$. Consider any nilpotent linearization of $p$ and let $M(z)$ be
the solution of DEL \eqref{eq:MDE}.  Notice that 
\begin{equation}\label{MM}
 \| \GSol(z)\| \le C\,\bigg( 1+ \frac{1}{\Im z}\bigg), \qquad  \| \GStOp^{-1}\| \le C\,(1+|z|)^2\,\bigg( 1+ \frac{1}{\Im z}\bigg)^2,
\end{equation}
where the first bound was obtained in Lemma~\ref{lem:MDEexistence} (i).  The second bound is a consequence of
 the identity
\begin{equation}\label{eq:identity}
\GStOp^{-1}[R] =(\id\otimes \tau)\, \bigg( \frac{1}{\Lb-zJ\otimes \freeunit}\,\big(M^{-1}RM^{-1} \otimes \freeunit \big)\,\frac{1}{\Lb-zJ\otimes \freeunit}\bigg),
\end{equation}
the a priori bound $\| \GSol^{-1}\|\le C(1+|z|)$, see \eqref{eq:MinvRegB1}, and  the bound
\eqref{eq:trivResBound} applied to semicircular elements as in  \eqref{scbound}.
The identity \eqref{eq:identity} follows immediately by expressing the derivative of the function 
$$
 \Phi(A):= (\id\otimes \tau) (\Lb - (zJ-A)\otimes \freeunit )^{-1}, 
$$ 
at $A=0$ 
in two different ways. The derivative on the free probability level gives the right hand side of \eqref{eq:identity}, while
the derivative on  the level of
the Dyson equation \eqref{eq:MDE} perturbed with $A$ gives the left hand side, see \eqref{Mderiv} for a similar calculation.

Thus \eqref{MM} shows that  the analogues of the conditions \textup{(\textbf{M1})-(\textbf{M2})} hold
 away from the real axis, i.e. on any compact set $\{ z\in \CC\; : \; |z|\le C^*, \Im z\ge c^*\}$
with fixed positive thresholds $c^*, C^*$. Inspecting the proof of the local law in Section~\ref{sec:locallaw}, we see
that the entire argument goes through under  these modified assumptions. We leave the details to the reader.

Finally, using the boundedness of all  $\Xb$ and $\Yb$ matrices with very high probability, a standard cutoff argument
yields \eqref{fp} for arbitrary smooth function.
\end{proof}

\section{Examples}
\label{sec:examples}
In this section we prove optimal bulk local law (in the sense of Theorem~\ref{thm:mainPoly}) for two concrete families of polynomials of random matrices, namely, for the eigenvalues of  quadratic forms in Wigner matrices and 
for eigenvalues of symmetrized products (i.e. singular values of products) of matrices with \emph{i.i.d.} entries.

\subsection{Local law for homogeneous polynomials of degree two in Wigner matrices}
\label{sec:example}
   Consider a family of noncommutative self-adjoint polynomials of degree 2 in $\gamma_*\geq 2$ variables given by
   \begin{equation*}
     \tilde{q}(x_1,\ldots,x_{\gamma_*}) =
     \xb^{t} \,\Xi \, \xb
     ,
   \end{equation*}
   where $\xb = (x_1,\ldots, x_{\gamma_*})^{t}$ and $\Xi$ is a Hermitian $\gamma_*\times \gamma_*$ matrix.
   We will assume that $\Xi$ is invertible.
   Note that if we take $\Xi = \left(
     \begin{array}{cc}
       0 & 1
       \\
       1 & 0
     \end{array}
\right)$, then $\tilde{q}(x_1, x_2) = x_1 x_2 + x_2 x_1$, the anticommutator, which was studied by Anderson in \cite{Ande15}.
   
   Suppose that $\Lb$ is the minimal linearization of $\cstarunit-\tilde{q}(\xb)$ and suppose that for any $\epsB>0$ the assumptions (\textbf{M1}) and (\textbf{M2}) hold for $\Lb$ and the corresponding solution of the \emph{DEL} everywhere in the $\epsB$-bulk.
   This, together with Theorem~\ref{thm:mainPoly}, would imply that the optimal local law holds for the polynomial $\cstarunit-\tilde{q}$ everywhere in the $\epsB$-bulk.
   Therefore, in order to prove the local law it is enough to find a minimal linearization of $\cstarunit-\tilde{q}$ that satisfies (\textbf{M1}) and (\textbf{M2}).

   Before proceeding to the proof, we fix a specific linearization of the polynomial $\cstarunit-\tilde{q}$, which is particularly suitable for our computations.
   More precisely, let $\Lb(\xb) = K_0 - \sum_{\gamma = 1}^{\gamma_*} K_{\gamma} x_{\gamma}$ with
\begin{equation}\label{eq:ExLin}
  K_0
  =
  \left(
    \begin{array}{c|ccc}
      1 & 0 & \cdots & 0
      \\ \hline
      0 & & &
      \\
      \vdots & &\Xi^{-1} &
      \\
      0 & & &
    \end{array}
  \right)
  ,\quad
  K_{\gamma}
  =
  \left(
    \begin{array}{c|ccc}
      0 & & \hat{e}_{\gamma}^{t} &
      \\ \hline
        & 0 & \cdots & 0
      \\
       \hat{e}_{\gamma} & \vdots& & \vdots
      \\
       & 0 & \cdots & 0
    \end{array}
  \right)
  ,
\end{equation}
where $\{\hat{e}_{\gamma}\, :  1\leq \gamma\leq \gamma_{*}\}$ denotes the canonical basis of $\CC^{\gamma_*}$. In particular the operator $\Gamma$ is given by the formula
\begin{equation} \label{Gamma for deg 2 polynomial}
\Gamma 
\left(
\begin{array}{c|c}
\alpha& a^t
\\\hline
b & A
\end{array}
\right) = 
\left(
\begin{array}{c|c}
\tr A& b^t
\\\hline
a & \alpha I_{\gamma_*}
\end{array}
\right) \,,
\end{equation}
for $\alpha \in \CC$, $a,b \in \CC^{\gamma_*}$ and $A \in \CC^{\gamma_* \times \gamma_*}$.  
One can easily see that $\Lb(\xb)$ gives a linearization of $\cstarunit-\tilde{q}$.
Moreover, this linearization is in fact minimal.
To show this latter property of \eqref{eq:ExLin}, note that the matrix representation of the series $(1-\tilde{q})^{-1}$ corresponding to the linearization \eqref{eq:ExLin} (see Remark~\ref{rem:aAdvantage}) is given by $(K_0^{-1} e_1, e_1, K_1 K_0^{-1}, \ldots, K_{\gamma_*} K_0^{-1})$.
From the special structure of $K_0^{-1}$ and $K_{\gamma}$ we see that for 
\begin{equation*}
  K_{\gamma} K_0^{-1} e_1
  =
  \begin{pmatrix}
    0
    \\
  \hat{e}_{\gamma}  
\end{pmatrix}
,\quad
1\leq \gamma \leq \gamma_*
  .  
\end{equation*}
Therefore, with $\{e_{\gamma}\, : 1\leq \gamma \leq \gamma_{*}+1\}$ being the canonical basis of $\CC^{\gamma_*+1}$, we have 
\begin{equation*}
  \mathrm{span}\{e_1, K_{1} K_0^{-1} e_1, \ldots, K_{\gamma_*} K_0^{-1} e_1\}
  =
  \mathrm{span}\{e_1, e_2, \ldots, e_{\gamma_*+1}\}
  =
  \CC^{\gamma_*+1}
  ,
\end{equation*}
which corresponds to condition \eqref{eq:aMinCrit1} in Proposition~\ref{prop:min}.
Similarly, one can show that the condition \eqref{eq:aMinCrit2} is satisfied as well, i.e.,
\begin{equation*}
  \mathrm{span}\{K_0^{-1}e_1, K_0^{-1}K_{1} K_0^{-1} e_1, \ldots, K_0^{-1}K_{\gamma_*} K_0^{-1} e_1\}
  =
  \CC^{\gamma_*+1}
  .
\end{equation*}
We then conclude using Proposition~\ref{prop:min} that the linearization \eqref{eq:ExLin} is indeed minimal.

Below we show that for this choice of linearization conditions (\textbf{M1}) and (\textbf{M2}) hold everywhere in the $\epsB$-bulk.

\subsubsection{Boundedness of $\GSol$ (assumption (\textbf{M1}))}

First, we realize that the solution $\GSol(z)$ of the corresponding \emph{DEL} has the following structure
\begin{equation} \label{eq:Ex1}
  \GSol(z)
  =
  \left(
    \begin{array}{c|ccc}
      \GSol_{11}(z) & 0 & \cdots & 0
      \\ \hline
      0 & & &
      \\
      \vdots & & \widehat{\GSol}(z) &
      \\
      0 & & &
    \end{array}
  \right)
  ,
\end{equation}
where by $\widehat{\GSol}(z)$ we denote the $\gamma_*\times \gamma_*$ submatrix of $\GSol(z)$.
Indeed, by \eqref{eq:ExLin} and \eqref{Gamma for deg 2 polynomial} the right hand side of \eqref{eq:MregA} as well as taking the
 inverse preserves the claimed block structure. 
 Since the solution is obtained via a fixed point argument from an arbitrary starting point 
 by iterating these operations  and then taking the limit $\omega \to 0$, it takes the form \eqref{eq:Ex1}.

Now, for $\GSol = \GSol(z)$ of the form \eqref{eq:Ex1} we write the \emph{DEL} 
\begin{equation*}
  I + (zJ-K_0) \GSol + \sum_{\gamma=1}^{\gamma_*} K_{\gamma}\, \GSol \, K_{\gamma} \, \GSol 
  =
  0
  ,
\end{equation*}
which can be split into two parts
\begin{align} \label{eq:ExMDE1}
   1 + (z-1) \GSol_{11} + \GSol_{11} \Tr\widehat{\GSol}
    &=
    0
    ,
  \\ 
   I - \Xi^{-1} \widehat{\GSol} + \GSol_{11} \widehat{\GSol}
    &=
    0
    . \nonumber
\end{align}
From \eqref{eq:ExMDE1} we obtain that
\begin{equation} \label{eq:ExMDE3}
  \widehat{\GSol}
  =
  (\Xi^{-1} - \GSol_{11}I)^{-1}
  .
\end{equation}
Recall, that by the definition of the $\epsB$-bulk there exists $\eta_0>0$ small enough such that for all $\eta \in [0,\eta_0]$ and   $E \in B_{\epsB}$ we have $\Im \GSol_{11}(E+\imunit \eta) \geq \epsB/2$.
Therefore, since $\Xi$ is self-adjoint, \eqref{eq:ExMDE3} implies that $\| \widehat{\GSol}\| \leq 2/\epsB$ for all $E\in B_{\epsB}$ and $\eta \in [0, \eta_0]$.
Moreover, by plugging \eqref{eq:ExMDE3} into \eqref{eq:ExMDE1} we derive an equation for $\GSol_{11}$
\begin{equation*}
   1 + (z-1) \GSol_{11} + \sum_{\gamma=1}^{\gamma_*} \frac{\GSol_{11}}{\xi_{\gamma}^{-1}-\GSol_{11}}
    =
    0
    ,
  \end{equation*}
  where by $\xi_{\gamma} \in \RR$, $1\leq \gamma \leq \gamma_*$, we denoted the  eigenvalues of $\Xi$.
  Note that if there exists an unbounded solution $m$ of the equation 
  \begin{equation*}
    1 + (z-1) m + \sum_{\gamma=1}^{\gamma_*} \frac{m}{\xi_{\gamma}^{-1}-m}
    =
    0
    ,
  \end{equation*}
  then $m(z)$ can be unbounded only near the point $z=1$.
  In this case 
  \begin{equation*}
    \lim_{z\rightarrow 1} (z-1)m(z)
    =
    \gamma_*-1
    ,
  \end{equation*}
  which implies that $\Im m(z) < 0$ in the neighborhood of $z=1$.
  Therefore, function $\GSol_{11}(z)$, whose imaginary part by Lemma~\ref{lem:MDEexistence} (iii) must be nonnegative on $\CC_+$, has absolute value bounded by some $C>0$ for all $z\in \CC_+$.
  We conclude that the assumption (\textbf{M1}) holds for the linearization \eqref{eq:ExLin} of the  polynomial $1-\tilde{q}(\xb)$ with a constant $C_3=\max\{C,2/\epsB \}$ depending only on the model parameters $\epsB$ and $\Xi$.

  \subsubsection{Boundedness of $\GStOp^{-1}$ (assumption (\textbf{M2}))}

  In order to prove the stability assumption (\textbf{M2}), we will have to extract additional information from the \emph{DEL} \eqref{eq:MDE} by taking its imaginary part at $\eta = 0$
  \begin{equation}\label{eq:ExImMDE}
    \Im \GSol
    =
    \GSol^* \sum_{\gamma = 1}^{\gamma_*} K_{\gamma} \Im \GSol K_{\gamma } \GSol
    .
  \end{equation}
  By using \eqref{eq:ExLin}, \eqref{eq:Ex1} and \eqref{eq:ExMDE3} and comparing the $(1,1)$-components of both sides of \eqref{eq:ExImMDE}, we obtain that for all real $z=E$, $E\in B_{\epsB}$,
  \begin{equation} \label{eq:ExImMDE3}
    |\GSol_{11}|^2 \Tr \widehat{M}^* \widehat{M}
    =
    1
    .
  \end{equation}

  Now, consider the space of $(\gamma_*+1)\times (\gamma_*+1)$ matrices with basis vectors $\{E_{11},E_{12},\ldots\}$, on which the linear operator $\GStOp$ is acting, as $\CC^{(\gamma_*+1)^2}$ with the standard basis $\{e_1,e_2,\ldots\}$.
  On this latter space $\GStOp$ can be represented by the matrix
\begin{equation*}
  \Ab_{\GStOp}
  :=
  I_{(\gamma_*+1)^2} -  \sum_{\gamma=1}^{\gamma_*} \GSol K_{\gamma}  \otimes \GSol^{t}  K_{\gamma}  
  ,
\end{equation*}
or more explicitly
\begin{equation*}
  \Ab_{\GStOp}
  =
  \begin{pmatrix}
    I_{\gamma_*+1} & -\Ab_{12} 
    \\
    -\Ab_{21} & I_{(\gamma_*+1)\gamma_*} 
  \end{pmatrix}
  ,
\end{equation*}
where
\begin{equation*}
  \Ab_{12}
  =
  \GSol_{11} \sum_{\gamma=1}^{\gamma^*} e_{\gamma}^{t} \otimes \GSol^{t} K_{\gamma} 
  ,
  \quad
  \Ab_{21}
  =
  \sum_{\gamma=1}^{\gamma_*} \widehat{\GSol} e_{\gamma}  \otimes \GSol^{t} K_{\gamma} 
  .
\end{equation*}
Note, that
\begin{equation*}
  \det(\Ab_{\GStOp})
  =
  \det(I_{\gamma_*+1} -\Ab_{12} \Ab_{21} )
  ,
\end{equation*}
therefore, in order to prove invertibility of $\GStOp$ it will be enough to show invertibility of
\begin{equation}\label{eq:exampleSmallMatrix}
  I_{\gamma_*+1} -\Ab_{12} \Ab_{21}
  =
  \left(
    \begin{array}{c|ccc}
      1-\GSol_{11}^{2} \Tr(\widehat{\GSol})^2 
      & 0& \cdots& 0
      \\ \hline
      0 & & &
      \\
      \vdots & & I_{\gamma_*} - \GSol_{11}^{2}  (\widehat{\GSol}^{t}\widehat{\GSol}) &
      \\
      0 & & &
    \end{array}
  \right)
  .
\end{equation}

Assume that the upper-left entry is not invertible, i.e.,
\begin{equation*}
  1-\GSol_{11}^{2} \Tr (\widehat{\GSol})^2 
  =
  1-\GSol_{11}^{2} \sum_{\gamma =1}^{\gamma_*} \frac{1}{(\xi_{\gamma}^{-1}-M_{11})^{2}} 
  =
  0
  ,
\end{equation*}
where we used \eqref{eq:ExMDE3}.
In this case, from \eqref{eq:ExImMDE3}, we obtain that for all $1\leq \gamma \leq \gamma_*$
\begin{equation*}
   (\GSol_{11}(\xi_{\gamma}^{-1} - \GSol_{11})^{-1})^2
  =
  |\GSol_{11}(\xi_{\gamma}^{-1} - \GSol_{11})^{-1}|^2
  ,
\end{equation*}
so that $\Im (\GSol_{11}(\xi_{\gamma}^{-1} - \GSol_{11})^{-1}) = \Im (\GSol_{11})\xi_{\gamma}^{-1}|\xi_{\gamma}^{-1} - \GSol_{11}|^{-2} = 0$ for all $1\leq \gamma \leq \gamma_*$, which leads to a contradiction with $\Im \GSol_{11} > \epsB$ for $z=E$, $E\in B_{\epsB}$.

Consider now the case when $ \det (I_{\gamma_*} - \GSol_{11}^{2}  (\widehat{\GSol}^{t}\widehat{\GSol}))=0$, so that the lower-right submatrix of \eqref{eq:exampleSmallMatrix} is singular.
This implies that there exists $\omega \in \CC^{\gamma_*}$, $\|\omega \| = 1$, such that
\begin{equation}
  \label{eq:ExStab1}
  \GSol_{11}^{2}  \omega^* \widehat{\GSol}^{t}\widehat{\GSol} \omega
  =
  1
  .
\end{equation}
We can rewrite the LHS of \eqref{eq:ExStab1} as
\begin{equation*}
  \sum_{k=1}^{\gamma_*} \la \omega, \vb_k \ra \la \overline{\omega}, \vb_k \ra
  =
  \sum_{k=1}^{\gamma_*} (\la \Re \omega, \vb_k \ra )^2 + (\la \Im \omega, \vb_k \ra)^2
  ,
\end{equation*}
where
\begin{equation*}
  \vb_k
  :=
  (\GSol_{11} \widehat{\GSol}_{k1},\GSol_{11} \widehat{\GSol}_{k2}, \ldots, \GSol_{11} \widehat{\GSol}_{k\gamma_*})^{t}
  .
\end{equation*}
Due to \eqref{eq:ExImMDE3}, using triangular and Cauchy-Schwarz inequalities, we have
\begin{equation}\label{eq:ExStab3}
  \Big|\sum_{k=1}^{\gamma_*} (\la \Re \omega, \vb_k \ra )^2 + (\la \Im \omega, \vb_k \ra)^2 \Big|
  \leq
  \|\omega \|^2 \sum_{k=1}^{\gamma_*}\|\vb_k\|^2
  =
  |\GSol_{11}|^2 \Tr \widehat{\GSol}^* \widehat{\GSol}
  =
  1
  .
\end{equation}
Assumption \eqref{eq:ExStab1} implies that the first inequality in \eqref{eq:ExStab3} is in fact an equality and that 
\begin{equation*}
  |\la \Re \omega, \vb_k \ra|^2
  =
  \|\Re \omega \|^2 \| \vb_k \|^2
  ,\quad
   |\la \Im \omega, \vb_k \ra|^2
  =
  \|\Im \omega \|^2 \| \vb_k \|^2
  ,\quad
  1\leq k \leq \gamma_*
  .
\end{equation*}
 Thus there exist $c_1^{(1)},\ldots,c_{\gamma_*}^{(1)},c_1^{(2)},\ldots,c_{\gamma_*}^{(2)} \in \CC$ such that
\begin{equation*}
  \vb_k
  =
  c_k^{(1)} \Re \omega
  =
  c_k^{(2)} \Im \omega
  ,
\end{equation*}
and we see that the rows of the  matrix $\GSol_{11} \widehat{\GSol}$ are linearly dependent.
At the same time we know that since $\Im \GSol_{11}>\epsB$ in the $\epsB$-bulk, by \eqref{eq:ExMDE3}  matrix $\widehat{\GSol}$ must be invertible.
From the obtained contradiction we conclude that $I_{\gamma_*+1} - \Ab_{12} \Ab_{21}$, $\Ab_{\GStOp}$ and $\GStOp$ are all invertible for $z=E$ with $E\in B_{\epsB}$, so that there exists $C>0$ depending only on $\epsB$ and $\Xi$, such that $\| \GStOp^{-1}(E) \| \leq C$  for all $E\in B_{\epsB}$.
Now a simple continuity argument, together with the a priori bound from Lemma~\ref{lem:MDEexistence}, shows that the condition (\textbf{M2}) holds for the model given by \eqref{eq:ExLin} everywhere in the $\epsB$-bulk.

\subsection{Local law for singular values of a product of independent non-Hermitian matrices}
\label{sec:exampl-sing-valu}
Consider $q(y_1,\ldots, y_{\beta_*}, y_1^*, \ldots, y_{\beta_*}^*) = y_1 \cdots y_{\beta_*} y_{\beta_*}^* \cdots y_1^*$.
Then a minimal linearization of the polynomial $\cstarunit - q(\yb,\yb^*)$ is given by
\begin{equation}
  \label{eq:7}
  \Lb
  =
  \left(
    \begin{array}{c|ccccccc}
      1&&&&&&&y_1
      \\ \hline
       &&&&&& y_2 &1
      \\
       &&&&&\iddots&\iddots&
      \\
       &&&&y_{\beta_*}&1&&
      \\ 
       &&&y_{\beta_*}^*&1&&&
      \\
       &&\iddots&1&&&&
      \\
       & y_2^* &\iddots&&&&&
      \\
      y_1^*&1&&&&&&
    \end{array}
  \right) = K_0 \otimes \freeunit+ \Yb
  ,
\end{equation}
or, using the representation \eqref{eq:linearization} and the basis vectors $E_{ij}=e_i e_j^t$, by a set of matrices
\begin{equation}
  \label{eq:6}
  K_0
  =
  E_{1,1} + \sum_{j=2}^{2\beta_*} E_{j,2\beta_*+2-j}
  ,
  \qquad
  L_{\beta}
  =
  E_{\beta, 2\beta_*+1-\beta}
  ,\quad
  1\leq \beta \leq \beta_*
  .
\end{equation}
Here, the corresponding operator $\Gamma$ has the simple form
\begin{equation}\label{Gamma for product}
\Gamma [R] = \mathrm{diag}(r_{2\beta_*+1-i , 2\beta_*+1-i})_{i=1}^{2\beta_*}\,, \qquad R =(r_{i,j})_{i,j=1}^{2\beta_*}\,,
\end{equation}
where $\mathrm{diag}(a) \in \CC^{2\beta_* \times 2\beta_*}$ is the diagonal matrix with vector $a \in \CC^{2\beta_*}$ along its diagonal. 

Before  proving that assumptions (\textbf{M1}) and (\textbf{M2}) hold for this  model, we show first that the solution matrix $\GSol(z)$ has the following structure 
\begin{equation}
  \label{eq:9}
  \GSol (z)
  =
  \sum_{j=1}^{2\beta_*} m_j(z) E_{jj} + \sum_{j=2}^{2\beta_*} m_{\beta_*+1}(z) E_{j,2\beta_*+2-j} -m_{\beta_*+1}(z) E_{\beta_*+1,\beta_*+1}
  ,
\end{equation}
for some $m_j : \CC_+\rightarrow \CC$, $1\leq j \leq 2\beta_*$.  
In order to do so we introduce an auxiliary parameter $\alpha>0$ and consider the linearization $\Lb_\alpha = K_0  \otimes \freeunit + \alpha \Yb$ of the polynomial 
$\cstarunit - \alpha^{2\beta_*}q(\yb,\yb^*)$ with $\Gamma_{\alpha}= \alpha^2 \Gamma$. We use the representation \eqref{eq:41} of the corresponding solution $M_\alpha$ to the  DEL with
$\Gamma_{\alpha}$ and expand 
into geometric series for any fixed $z$ and sufficiently small $\alpha$ as   
\[
M_\alpha(z)= (\id\otimes \tau)\Bigg[\Big(\big(K_0-zJ\big)^{-1}\otimes \freeunit\Big)\sum_{k=0}^\infty (-\alpha)^k\Xb^k\Bigg], \qquad \Xb:=\widehat{\Yb}\,(K_0-zJ\big)^{-1}\otimes \freeunit\,,
 \]
 where  $\widehat{\Yb}$ is defined as 
 all the $y_i$ inside $\Yb$ are replaced by free circular elements $\widehat{y}_i$. Due to its cyclic structure the only powers of $\Xb$ with non-zero elements on the diagonal are integer multiples of $2 \beta_*$ and since $ \Xb^{2\beta_*}$ has $\zeta^{-1}q(\widehat{\yb},\widehat{\yb}^*)$ with $\zeta:=1-z$ in all diagonal entries we conclude that $(K_0-zJ)^{-1}M_\alpha(z)$ has a constant diagonal. Furthermore, the subalgebra of $\CC^{2 \beta_* \times 2 \beta_*}$ of all matrices with the same non-zero entries as $M(z)$ from \eqref{eq:9} is left invariant by matrix inversion, addition of  $K_0 -z J$ (cf. \eqref{eq:6}) as well as application of the operator $\Gamma_\alpha$. Thus similarly to the argument we used in Section~\ref{sec:example} the solution $M_\alpha$ also has only these non-zero entries and thus 
 is of the form \eqref{eq:9} for small enough $\alpha$. Since $M_\alpha$ is analytic in $\alpha>0$ for every $z \in \CC_+$, we conclude that \eqref{eq:9} also holds at $\alpha=1$.

       \subsubsection{Boundedness of $\GSol$ (assumption (\textbf{M1}))}
       We now prove that $\GSol$ is bounded everywhere in the $\epsB$-bulk.
       Using the structure of $\GSol$ \eqref{eq:9}, the \emph{DEL} \eqref{eq:MDE} can be reduced to the following system of equations for $m_{\beta}, 1\leq \beta\leq 2\beta_*$,
          \begin{equation*}
     \left\{
       \begin{array}{ll}
         1-\zeta m_1 + m_{2\beta_*} m_1 = 0, &
         \\
         1 - m_{\beta_*+1} + m_{2\beta_*+1-\beta} m_{\beta} = 0, & 2\leq \beta \leq 2\beta_*,
         \\
         - m_{2\beta_*+2-\beta} + m_{2\beta_*+1-\beta} m_{\beta+1} = 0, & 2\leq \beta \leq 2\beta_*, i \neq \beta_*+1.
       \end{array}
       \right.
     \end{equation*}
     From these equations we obtain that all $m_{\beta}$, $\beta\geq 2$, can be expressed in terms of $m_1$:
     \begin{equation} \label{eq:3}
       \begin{array}{ll}
       m_{\beta_*+1}
       =
         \zeta m_1
         ,&
       \\
       m_{\beta_*+1+\beta}
         =
          m_{\beta_*+1}^{\beta+1}
         =
          (\zeta m_1)^{\beta+1}
         ,&
         0\leq \beta \leq \beta_*-1
         ,
       \\
       m_{\beta}
       =
       m_1 m_{\beta_*+1}^{\beta-1}
       =
        \zeta^{\beta-1}m_1^{\beta}
       ,&
       2\leq \beta \leq \beta_*
       ,
       \end{array}
     \end{equation}
     and  $m_1(z)=\la e_1, M(z) e_1 \ra$ satisfies the following polynomial equation
     \begin{equation}
       \label{eq:11}
       1 - \zeta m_1 + \zeta^{\beta_*} m_1^{\beta_*+1}
       =
       0
       ,\qquad
       \zeta
       =
       1-z.
     \end{equation}
     From \eqref{eq:11} it is easy to see that $|m_1(z)|$ can be unbounded only in the neighborhood of $z=1$.
     Moreover, we will show that there exists $c(\epsB)= c(\epsB,\beta_*)>0$ small enough such that
     \begin{equation}
       \label{eq:2}
       B_{\epsB} \subset [1-C(\beta_*), 1-c(\epsB,\beta_*)]
       ,
     \end{equation}
     where $C(\beta_*)\geq 2$  comes from the boundedness of the support of the density of states.
     In order to prove the upper bound in \eqref{eq:2} we may, without loss of generality, consider only $z\in \CC_+$  with $|\zeta|=|1-z|\leq 4^{-(\beta_*+1)}$ and $\eta=\Im z$ small.
     We will show that for such $z$ the condition  $E=\Re z \in B_{\epsB}$ implies $|\zeta|\geq c(\epsB)$, where $c(\epsB)$ will be specified below.
     Rewrite \eqref{eq:11} as
     \begin{equation}
       \label{eq:14}
       (\zeta m_1)^{\beta_*+1}
       =
       -\zeta (1-\zeta m_1)
       ,
     \end{equation}
     from which it follows that $ |\zeta m_1|^{\beta_*+1} \leq |\zeta|(1+|\zeta m_1|)$, so that $|\zeta m_1| \leq 2|\zeta|^{\frac{1}{\beta_*+1}}$.
     This last bound implies that $|\zeta m_1|< 1/2$, which, together with \eqref{eq:14}, yields
     \begin{equation}
       \label{eq:15}
       |m_1|
       \sim
       |\zeta|^{-\frac{\beta_*}{\beta_*+1}}
       .
     \end{equation}
     Suppose that $|\zeta|^{-\frac{\beta_*}{\beta_*+1}} \geq C'\epsB^{-1}$ with some large constant $C'$.
     For $E$ in the $\epsB$-bulk and $\eta$ small we have $ \Im m_1 \leq 2 \epsB^{-1}$, which together with \eqref{eq:15} gives $ |\Re m_1| \sim |\zeta|^{-\frac{\beta_*}{\beta_*+1}}$.
     In the regime when $\eta \ll |1-E|$ and $|\zeta| \sim |1-E|$, by taking the imaginary part of the equation \eqref{eq:11} and dividing it through by $\Im m_1  (\Re m_1)^{-1}$ we obtain
     \begin{equation}
       \label{eq:17}
       0
       =
       -(1-E)\Re m_1 + (\beta_*+1)(1-E)^{\beta_*}(\Re m_1)^{\beta_*+1} + O\,\bigg(\bigg| \frac{\Im m_1}{\Re m_1}\bigg| + \frac{\eta}{|\zeta|^2}\bigg| \frac{\Re m_1}{\Im m_1}\bigg| \bigg)
       .
     \end{equation}
     Choosing $C'$ sufficiently large and $\eta$ sufficiently small (depending on $\epsB$) the error term becomes negligible, using $\Im m_1\geq \epsB$ since we are in $B_{\epsB}$.
     Using the scaling of $1-E$ and $\Re m_1$ in $\zeta$, we obtain $|(1-E) \Re m_1| \sim |\zeta|^{\frac{1}{\beta_*+1}}$ and  $|(\beta_*+1)(1-E)^{\beta_*}(\Re m_1)^{\beta_*+1}|\sim 1$.
     Since $\zeta$ is small, this leads to a contradiction in \eqref{eq:17}, hence  $|\zeta|^{-\frac{\beta_*}{\beta_*+1}} \geq C'\epsB^{-1}$ cannot hold.
     This finishes the proof of \eqref{eq:2} with $c(\epsB)=(\epsB/C')^{\frac{\beta_*+1}{\beta_*}} $.
     Now, for any $E\in B_{\epsB}$ and $\eta>0$ small enough, \eqref{eq:11} and \eqref{eq:2} imply that $ |m_1| \leq (c(\epsB))^{-1}C$,  which gives an effective bound on $m_1$.
     Boundedness of $|m_1|$ together with \eqref{eq:3} implies that assumption (\textbf{M1}) holds everywhere in  $B_{\epsB}$ with $C_3 = c(\epsB)^{3\beta_*}$.
     
     \subsubsection{Boundedness of $\GStOp^{-1}$ (assumption (\textbf{M2}))}
     In this section we show that assumption (\textbf{M2}) holds everywhere in the $\epsB$-bulk, which together with Theorem~\ref{thm:mainPoly} implies optimal bulk local law for the singular values of a product of matrices $Y_1\cdots Y_{\beta_*}$ satisfying \textup{(\textbf{H1})-(\textbf{H4})}.

     By \eqref{eq:6}, which, in particular, gives that  $\GSol^{t}(z)=\GSol(z)$, and \eqref{eq:9}, matrix $\Ab_{\GStOp} =I- \sum_{\beta=1}^{\beta_*} (\GSol L_{\beta} \otimes \GSol L_{\beta}  + \GSol L_{\beta}^{t}  \otimes \GSol L_{\beta}^{t})$ representing the stability operator $\GStOp$ in the standard basis of $\CC^{(2\beta_*)^2}$ can be written as
          \begin{align*}
            \Ab_{\mathscr{L}}
              =
              I_{(2 \beta_*)^2 }  & - (m_1 E_{1,2\beta_*})^{\otimes 2}
              - \sum_{\beta=2}^{\beta_*} (m_\beta E_{\beta,2\beta_*+1-\beta} + m_{\beta_*+1} E_{2\beta_*+2-\beta,2\beta_*+1-\beta})^{\otimes 2} 
            \\
            &
             - \sum_{\beta=1}^{\beta_*-1}  (m_{2\beta_*+1-\beta} E_{2\beta_*+1-\beta,\beta} + m_{\beta_*+1} E_{\beta+1,\beta})^{\otimes 2}
              - (m_{\beta_*+1} E_{\beta_*+1,\beta_*})^{\otimes 2}      
              ,
          \end{align*}
          where for any $k\in \NN$ and  $R \in \CC^{k\times k}$ we denote $R^{\otimes 2}:= R\otimes R$.
          After removing from $\Ab_{\GStOp}$ rows and columns for all indices such that either row or column of the corresponding index has only one non-zero entry equal to $1$, we obtain that $\det(\Ab_{\mathscr{L}}) = \det(\widehat{\Ab}_{\mathscr{L}})$,
       where
       \begin{equation*}
         \widehat{\Ab}_{\mathscr{L}}
         =
         -I_{2\beta_*} + \sum_{\beta=1}^{2\beta_*} m_{\beta}^{2} E_{\beta,2\beta_*+1-\beta} + m_{\beta_*+1}^2 \sum_{\beta=1}^{\beta_*-1}(E_{\beta+1, \beta } + E_{\beta_*+\beta+1, \beta_*+\beta}) 
         .
       \end{equation*}
              Divide $\widehat{\Ab}_{\mathscr{L}}$ into four blocks of equal size $\Ab_{ij}, 1\leq i,j\leq 2$, so that, e.g., $\Ab_{11}$ denotes the upper-left $\beta_*\times \beta_*$ submatrix of $\Ab$.
       Then by the lower-triangular structure of $\Ab_{11}$ and $\Ab_{22}$ having $-I_{\beta_*}$ on their diagonals, skew-diagonal shape of $\Ab_{12}$ and $\Ab_{21}$, and \eqref{eq:3} we have that $\det (\widehat{\Ab}_{\mathscr{L}})$ is equal to
       \begin{equation*}
         \det (\Ab_{11}-\Ab_{12} \Ab_{22}^{-1} \Ab_{21})
         =
         \det \left(
           \begin{array}{cccccc}
             \upsilon-1&\upsilon& \upsilon & \upsilon&\cdots &\upsilon
             \\
             \omega&\upsilon-1&\upsilon&\upsilon &\cdots & \upsilon
             \\
              & &\ddots& & &\vdots
             \\
                       &&\omega&\upsilon-1&\upsilon&\upsilon
                            \\
             & &  &\omega&\upsilon-1&\upsilon
             \\
              & & &&\omega &\upsilon-1
           \end{array}
         \right) 
       \end{equation*}
with $\upsilon = \zeta^{2\beta_*} m_1^{2(\beta_*+1)}$ and $\omega=(\zeta m_1)^2$.
       One can easily see that the above determinant is equal to the determinant of the following tridiagonal matrix
       \begin{equation*}
        \left(
           \begin{array}{cccccc}
             \upsilon-1&1 &&&&
             \\
             \omega &\upsilon-1-\omega& 1 &&&
             \\
             & \omega &\upsilon-1-\omega& 1&&
             \\
             &&\ddots&\ddots&\ddots&
             \\
              && & \omega &\upsilon-1- \omega&1
             \\
              && & & \omega &\upsilon-1 -\omega
           \end{array}
         \right) 
         ,
       \end{equation*}
       that is equal to
       \begin{equation}
         \label{eq:4}
         (\upsilon-1) \det(\TT_{\beta_*-1}(\upsilon-1-\omega,1,\omega )) - \omega \det(\TT_{\beta_*-2}(\upsilon-1-\omega,  1, \omega )),
       \end{equation}
       where $\TT_{k}(a,b,c)$ denotes a $k\times k$ Toeplitz tridiagonal matrix with $a$ on the main diagonal, and $b$ and $c$ above and below the main diagonal respectively.
       From \eqref{eq:11} we have
       \begin{equation}
         \label{eq:12}
         \upsilon-1-\omega
         =
         -2\zeta m_1
         ,
       \end{equation}
       and thus $(\upsilon-1-\omega)^2=4\omega $.
       Note, that under the condition $a^2=4bc$ the determinant of the Toeplitz tridiagonal matrix takes a particularly simple form
       \begin{equation*}
         \det(\TT_{k}(a,b,c ))
         =
         (k+1) \left(\frac{a}{2}\right)^k
         .
       \end{equation*}
       A simple calculation from \eqref{eq:4} and \eqref{eq:12} gives that
       \begin{equation*}
         \det( \Ab_{\mathscr{L}})
         =
         (\beta_*+1) (-\zeta m_1)^{\beta_*} + (\zeta m_1)^2 \beta_* (-\zeta m_1)^{\beta_*-1}
         .
       \end{equation*}
       Hence, $\det( \Ab_{\mathscr{L}})=0$ implies that (since $\zeta m_1 \neq 0$ in the $\epsB$-bulk)
       \begin{equation}\label{eq:5}
         \zeta m_1
         =
         \frac{\beta_*+1}{\beta_*}
         .
       \end{equation}
       Now if we plug \eqref{eq:5} into \eqref{eq:3} we obtain
       \begin{equation*}
         m_1
         =
         \Big(\frac{\beta_*}{\beta_*+1}\Big)^{\beta_*}\frac{1}{\beta_*}
         ,\qquad
         \zeta
         =
         \frac{(\beta_*+1)^{\beta_*+1}}{{\beta_*}^{\beta_*}}
         ,\qquad
         z
         =
         1-\frac{(\beta_*+1)^{\beta_*+1}}{{\beta_*}^{\beta_*}}
         .
       \end{equation*}
       Since at $z=1-(\beta_*+1)^{\beta_*+1}/{\beta_*}^{\beta_*}$ the imaginary part of $m_1(z)$ vanishes, this  point does not belong to the $\epsB$-bulk.
       This argument, which can be made effective, yields that \textup{(\textbf{M2})} holds.

 \appendix

       \section{Linearizations of noncommutative polynomials: construction and minimization}
\label{sec:aLinCM}

 Let $\Acc$ be a unital $C^*$-algebra and let  $\xb = \{x_1,\ldots,x_{\gamma_*}\}\subset \Acc$ be a family of self-adjoint noncommutative variables.
Let  $\tilde{q}=\tilde{q}(\xb)$ be a self-adjoint polynomial such that $\tilde{q}(0)  =0$.
In this section we  present two methods to construct a (self-adjoint) linearization of $1-\tilde{q}$.

\subsection{Standard algorithm for constructing a symmetric linearization}
\label{sec:abiglinearization}
We present now a simple and fairly standard  procedure for constructing a (self-adjoint) linearization.
Several versions of this algorithm have appeared in the literature, see e.g.  \cite{Ande13}, but for definiteness we present it here 
in a setup the most convenient for us. For simplicity, we will call it \emph{standard linearization}.

Suppose that for  general (even not necessarily self-adjoint) polynomials $a_1,\ldots,a_k \in \CC \la \xb \ra$ we have matrices 
\begin{equation*}
  \left(
    \begin{array}{c|c}
      0 & d_i
      \\ \hline
      b_i & U_i
    \end{array}
  \right) \in \CC \la \xb \ra^{m_i \times m_i},\quad U_i \in \CC \la \xb \ra^{(m_i-1) \times (m_i-1)},\quad 1\leq i\leq k,
\end{equation*}
such that
\begin{equation*}
  - d_i \,U_i^{-1}b_i
  =
  a_i
  ,\quad
  1\leq i\leq k
  .
\end{equation*}
We now show how to construct a linearization for a scalar multiple of $a_i$, for sums of $a_i$'s and for the real part of $a_i$. 
One can easily check that the following rules hold:
\begin{itemize}
\item[(R1)] $\left(
    \begin{array}{c|c}
      0 & d
      \\ \hline
      b & U
    \end{array}
  \right) 
  :=
  \left(
    \begin{array}{c|c}
      0	&  \sqrt{|\zeta|}d_1		
      \\ \hline
      \sqrt{|\zeta|} b_1	& \zeta^{-1} |\zeta| U_1				 		
    \end{array}
  \right)
  \,\quad \text{ gives } -d\,U^{-1}b = \zeta  \,a_1 \text{ for } \zeta \neq 0\, ,
  $
\item[(R2)] $\left(
    \begin{array}{c|c}
      0 & d
      \\ \hline
      b & U
    \end{array}
  \right) 
  :=
  \left(
    \begin{array}{c|ccc}
      0		&  d_1	& \cdots& d_k	
      \\ \hline
      b_1	& U_1	& 		& 		
      \\
      \vdots	&		&\ddots	& 		
      \\
      b_k	&		& 		&  U_k 	
    \end{array}
  \right)
  \,\quad \text{ gives } -d\,U^{-1}b = a_1+a_2+ \dots +a_k\, ,
  $
\item[(R3)] $\left(
    \begin{array}{c|c}
      0 & d
      \\ \hline
      b & U
    \end{array}
  \right) 
  :=
  \left(
    \begin{array}{c|cc}
      0	&  d_1	&  b_1^*	
      \\ \hline
      d_1^*	& 0		& U_1^*	 		
      \\
      b_1	&	U_1	& 0		 		
    \end{array}
  \right) 
  \,\quad \text{ gives } -d\,U^{-1}b = a_1+a_1^*\, .
  $
\end{itemize}
Given a self-adjoint polynomial $\tilde q(\xb)$ with $\tilde q(0)=0$, we now construct its linearization, i.e.  $\ell$ and $\widehat{\Lb}$, by using the following procedure:
\begin{enumerate}
\item With $\tilde{q}_1$ denoting the linear part of $\tilde{q}$, put $\lambda=\cstarunit - \tilde{q}_1(\xb)$.
\item  Write $\tilde{q}-\tilde{q}_1$ as the sum of monomials of type $\zeta_{\alpha_1\cdots\alpha_k} x_{\alpha_1} \cdots x_{\alpha_k}$ of degree at least two with $(\alpha_1, \ldots, \alpha_k)\in \llbr \gamma_* \rrbr^{k}$ and $\zeta_{\alpha_1\cdots\alpha_k} \in \CC$; for definiteness order the terms in the sum with respect to their degree from lowest to highest, and within each group of monomials of the same degree order them lexicographically with respect to $(\alpha_1,\ldots,\alpha_k)$.
\item   For each such  monomial construct a linearization of $x_{\alpha_1} \cdots x_{\alpha_k}$ using the basic linearization rule
  \begin{equation} \label{eq:abasicblock}
    \left(
      \begin{array}{c|c}
        0 & d
        \\ \hline
        b & U
      \end{array}
    \right) 
    :=
    \left(
      \begin{array}{c|cccc}
        0	& 0 & \cdots & 0 &  x_{\alpha_1}		
        \\ \hline
        0	&  & & x_{\alpha_2} & -\cstarunit
        \\
        \vdots& &\iddots & \iddots &0
        \\
        0 & x_{\alpha_{k-1}} & \iddots &\iddots &
        \\
        x_{\alpha_k} & -\cstarunit &0  & \cdots & 0
      \end{array}
    \right)
    \,\quad \text{ gives } -d\,U^{-1}b = x_{\alpha_1}x_{\alpha_2}\cdots x_{\alpha_k} \, .
  \end{equation}
\item Then use rule (R1) to multiply the monomials by non-zero coefficients $\zeta_{\alpha_1\cdots\alpha_k}$.
\item Next, use rule (R2) to obtain a linearization (not Hermitian at this point) of the sum of monomials.
\item  Finally, rule (R3) applied to the linearization obtained on the previous stage gives a symmetric linearization.
\end{enumerate}

Note, that if we write $\Lb$  obtained along this procedure in the form \eqref{eq:aLsum3}, then
\begin{equation}\label{eq:apermut}
  K_0
  =\left(
    \begin{array}{c|ccc}
      1 & 0 & \cdots & 0
      \\ \hline
      0 & & &
      \\
      \vdots & & \Theta \widehat{K}_0 &
      \\
      0 & & &
    \end{array}
  \right)
  ,
\end{equation}
where $\widehat{K}_0$ is a permutation matrix and $\Theta=\mathrm{Diag}(e^{i \theta_1}, \ldots, e^{i \theta_{m-1}} )$ for some $\theta_1, \ldots, \theta_{m-1} \in \RR$.

\subsection{Minimal linearization}\label{sec:minimallinearization}

The procedure described in the previous section is only one of many ways of constructing a (self-adjoint) linearization.
For example, we could repeat (R3) as many times as we wish creating  a linearization of higher dimension which would still have all the required properties.
In this section we will show how one can reduce the dimension of the linearization, which can be particularly useful if we want to check numerically that the conditions \textup{(\textbf{M1})-(\textbf{M2})} are satisfied for some particular polynomial.
Symmetric minimal linearizations have been constructed before, see \cite[Lemma~4.1 (3)]{HeltMaccVinn06}.
We present a particularly direct construction here.

\begin{defn}[Minimal linearization]
  We say that linearization $\Lb$ of a polynomial  $\cstarunit - \tilde{q}(\xb) \in \CC \la \xb \ra$ is \emph{minimal} if it has the smallest dimension among all linearizations of $\cstarunit - \tilde{q}(\xb)$.
\end{defn}
In our construction of the minimal linearization we will need the notion of a \emph{matrix representation of a series}.
\begin{defn}[Matrix representation of a series] \label{def:matRepr}
  Let $s=s(x_1,\ldots,x_{\gamma_*})$ be a formal series in noncommutative variables $x_{1},\ldots,x_{\gamma_*}$
  \begin{equation*}
    s
    =
    s_{0} + \sum_{k=1}^{\infty} \sum_{(\gamma_1,\ldots,\gamma_k)\in \llbr \gamma_* \rrbr^k} s_{\gamma_1 \ldots \gamma_k} x_{\gamma_1}\cdots x_{\gamma_k}
    .
  \end{equation*}
  Let $v_1,v_2\in \CC^{m}$ and $V_1,\ldots,V_{\gamma_*}\in\CC^{m\times m}$.
  We say that $(v_1,v_2, V_1,\ldots V_{\gamma_*})$ is a \emph{matrix representation} of $s$ if for any $k\in \NN$ and any $(\gamma_1,\ldots,\gamma_k)\in \llbr \gamma_* \rrbr^k$
  \begin{equation*}
    s_0
    =
    \la v_1,v_2 \ra
    ,\quad
    s_{\gamma_1\ldots \gamma_k}
    =
    \la v_1, V_{\gamma_1}\cdots V_{\gamma_k} v_2 \ra
    .
  \end{equation*}
  We call $m$ the \emph{dimension} of the linearization.
\end{defn}
\begin{rem}
  The concept of matrix representation of a formal series introduced in Definition~\ref{def:matRepr} is also known in the literature (see e.g. Chapter 1, Section 5 in \cite{BersReutBook}) as linear representation or linearization.
  In order to avoid confusion in this paper we will reserve the term \emph{linearization} only for objects introduced in Definition~\ref{def:saLin}.
\end{rem}
Similarly as for linearizations, we can define a minimal matrix representation of a series.
\begin{defn}[Minimal matrix representation of a series]
  Matrix representation of a series is called \emph{minimal} if it has the smallest dimension among all possible matrix representations of this series.
\end{defn}
\begin{rem}\label{rem:aAdvantage}
  The advantage of introducing the minimal matrix representation is the following.
  On one hand,  it is very easy to see that
  \begin{equation}\label{eq:aLsum4}
    \Lb = K_0\otimes \cstarunit - \sum_{\gamma=1}^{\gamma_*} K_{\gamma}\otimes x_{\gamma}
  \end{equation}
  is a linearization of $\cstarunit - \tilde{q}(\xb) \in \CC \la \xb \ra$ if and only if
  \begin{equation}\label{eq:aMatRep}
    \big(K_0^{-1}e_1, e_1, K_1 K_0^{-1}, \ldots,  K_{\gamma_*} K_0^{-1} \big)
  \end{equation}
  gives a matrix representation of $(1-\tilde{q})^{-1} := 1 + \sum_{k=1}^{\infty} \tilde{q}^{k}$.
  Indeed, by the Schur complement formula \eqref{eq:schur} we have that
  \begin{equation}\label{eq:aGenRes11o}
    \la K_0^{-1} e_1\otimes \cstarunit, (I \otimes \cstarunit - \sum_{\gamma=1}^{\gamma_*} K_{\gamma} K_0^{-1}\otimes x_{\gamma} )^{-1} \, e_1\otimes \cstarunit \ra_{\Acc}
    =
    \frac{1}{\cstarunit - \tilde{q}(\xb) }
    ,
  \end{equation}
  and thus if we assume that $\sum_{\gamma=1}^{\gamma_*} \| K_{\gamma} K_0^{-1}\| \|x_{\gamma}\|_{\Acc} \leq 1/2$ and expand both the LHS and the RHS of \eqref{eq:aGenRes11o} into a power series with respect to $x_{\gamma}$'s, we will see that the coefficients in the expansion of $(1-\tilde{q})^{-1}$ are given by the matrix representation \eqref{eq:aMatRep}.

  On the other hand, there is a simple characterization of the minimal matrix representation of a series, which is stated in the following proposition.
  Therefore, one can use minimization of a special type of matrix representations of $(1-\tilde{q})^{-1}$ in order to construct a minimal linearization of the polynomial $\cstarunit - \tilde{q}(\xb)$.
\end{rem}
\begin{pr}[\cite{BersReutBook}, Proposition 2.1]\label{prop:min}
  Let $s=s(x_1,\ldots,x_{\gamma_*})$ be a series in noncommutative variables $x_{1},\ldots,x_{\gamma_*}$ and let $(v_1,v_2, V_1,\ldots V_{\gamma_*})$ with $v_1,v_2\in \CC^{m}$ and $V_1,\ldots,V_{\gamma_*}\in\CC^{m\times m}$ be one of its matrix representations.
  The matrix representation $(v_1,v_2, V_1,\ldots V_{\gamma_*})$ is minimal if and only if
  \begin{align}\label{eq:aMinCrit1}
    \mathrm{span} \Big(\{v_2\}\cup \bigcup_{k=1}^{\infty} \bigcup_{(\alpha_1,\ldots,\alpha_k)\in \llbr \gamma_* \rrbr^{k}} V_{\alpha_1}\cdots V_{\alpha_k} q_2 \Big)
    =
    \CC^{m}
    ,\\ \label{eq:aMinCrit2}
    \mathrm{span} \Big(\{v_1\}\cup \bigcup_{k=1}^{\infty} \bigcup_{(\alpha_1,\ldots,\alpha_k)\in \llbr \gamma_* \rrbr^{k}} V_{\alpha_1}^*\cdots V_{\alpha_k}^*q_1 \Big)
    =
    \CC^{m}
    .
  \end{align}
\end{pr}
In the next lemma we will show how to construct a linearization $\Lb$ of the form \eqref{eq:aLsum4}, such that the corresponding matrix representation of $(1-\tilde{q})^{-1}$ \eqref{eq:aMatRep} satisfies \eqref{eq:aMinCrit1}-\eqref{eq:aMinCrit2}.
This would imply that this linearization is minimal, since otherwise it would be possible to construct a minimal representation of $(1-\tilde{q})^{-1}$ with dimension smaller than minimal.
The matrices $\KK_{\gamma}$ below faithfully represent the collection of  self-adjoint matrices  $K_{\gamma}$ on a smaller space $\widetilde{U}$, which is the natural smallest space.

Before stating the next lemma let us introduce some notation that will be used to describe the minimization algorithm.
Denote by $\Icc$ the set of multi-indices
\begin{equation}\label{eq:aIcc}
  \Icc
  :=
  \{\emptyset\}\cup \bigcup_{k=1}^{\infty} \big\{(\alpha_1,\ldots \alpha_k)\in \llbr \gamma_* \rrbr^k \big\}
  .
\end{equation}
For any $k\in \NN$,  multi-index $\overline{\alpha}:=(\alpha_1, \ldots, \alpha_k)\in \llbr  \gamma_* \rrbr^k$ and a family of matrices $\{R_{\alpha} \, : \, \alpha \in \llbr  \gamma_* \rrbr \}$ we will denote
\begin{equation}\label{eq:Ralpha}
  R_{\emptyset}
  :=
  I
  ,\qquad
  R_{\overline{\alpha}}
  :=
  R_{\alpha_1}\cdots R_{\alpha_k}
  .
\end{equation}
For any two multi-indices $\overline{\alpha}\in\llbr \gamma_* \rrbr^k$ and $\overline{\beta}\in\llbr \gamma_* \rrbr^l$ we will denote by $\overline{\alpha \beta}$ the concatenation of $\overline{\alpha}$ and $\overline{\beta}$, i.e., $\overline{\alpha \beta}:= (\alpha_1, \ldots, \alpha_k,\beta_1, \ldots, \beta_l)$, and by $\overline{\alpha}^t$ the multi-index taken in the reversed order, i.e., $\overline{\alpha}^t:=(\alpha_k, \ldots, \alpha_1)$.
Finally for any multi-index $\overline{\alpha}$ of length $k$ and a linearization $\Lb$ of the form \eqref{eq:aLsum4} we will denote
\begin{equation*}
  \xi_{\overline{\alpha}}
  :=
  K_0^{-1} K_{\alpha_1} \ldots K_0^{-1} K_{\alpha_k} K_0^{-1} e_1
  .
\end{equation*}

\begin{lem}[Minimization algorithm] \label{lem:aminimization}
  Let $\tilde{q}\in \CC\la \xb \ra$ be self-adjoint such that $\tilde{q}(0)=0$ and let $\Lb=K_0\otimes \cstarunit - \sum_{\gamma=1}^{\gamma_*} K_{\gamma}\otimes x_{\gamma}$  be an arbitrary $n$-dimensional (self-adjoint) linearization of $\cstarunit - \tilde{q}(\xb)$ with $K_0$ invertible.
  Denote $A_{\gamma}:= K_{\gamma} K_{0}^{-1}$.
  \begin{enumerate}
    \setlength\itemsep{0em}      
  \item Define a subspace $U\subset \CC^n$ 
    \begin{equation}\label{eq:Udef}
      U
      :=
      \mathrm{span} \Big( \bigcup_{\overline{\alpha}\in \Icc} A_{\overline{\alpha}} \, e_1 \Big)
      \subset
      \CC^{n}
      .
    \end{equation}
  \item  Denote by $P_U:\CC^n\rightarrow U$ the orthogonal projection onto $U$ and define a subspace $\tilde{U}\subset U$ by
    \begin{equation}\label{eq:Utilde}
      \tilde{U}
      :=
      \mathrm{span} \Big( \bigcup_{\overline{\alpha}\in \Icc} (P_{U} A^* P_{U})_{\overline{\alpha}}K_0^{-1}e_1 \Big)
      \subset
      \CC^{n}
      .
    \end{equation}
    Let $m:=\mbox{dim}\; \tilde U$ be the dimension of $\tilde U$.
  \item   Choose a basis of $\tilde{U}$ in the form $ \{(P_{U} A^{ * } P_{U})_{\overline{\beta_i}}K_0^{-1}e_1, 1\leq i \leq m \}$ for some multi-indices $\overline{\beta}_{i}$ with $\overline{\beta}_1= \emptyset$.
  \item   Define $\KK_0 = (\KK_0(i,j))_{i,j=1}^m \in \CC^{m\times m}$ and $\KK_{\gamma} = (\KK_{\gamma}(i,j))_{i,j=1}^m \in \CC^{m\times m}$ by
    \begin{equation*}
      \KK_0(i,j)
      =
      \la \xi_{\overline{\beta_i}} ,  K_0 \, \xi_{\overline{\beta_j}} \ra
      ,\qquad
      \KK_{\gamma}(i,j)
      =
      \la \xi_{\overline{\beta_i}} ,  K_{\gamma} \, \xi_{\overline{\beta_j}} \ra
      .
    \end{equation*}
  \item   Take an arbitrary unitary matrix $W\in \CC^{m\times m}$ such that $    W e_1 =(\KK_0  e_1)/ \| \KK_0 \, e_1 \|_2$  and define
    \begin{equation}\label{eq:aKmin}
      \Qmin_0
      :=
      \frac{1}{ \| \KK_0 \, e_1 \|_2^2} W^* \, \KK_0 \, W
      ,\quad
      \Qmin_{\gamma}
      :=
      \frac{1}{ \| \KK_0 \, e_1 \|_2^2} W^* \, \KK_{\gamma} \, W
      .
    \end{equation}
  \end{enumerate}
  Then $\Lmin:= \Qmin_0\otimes \cstarunit - \sum_{\gamma=1}^{\gamma_*} \Qmin_{\gamma}\otimes x_{\gamma}$
  gives a minimal (self-adjoint) linearization of $\cstarunit - \tilde{q}(\xb)$.
\end{lem}
\begin{proof}
  Let us take a matrix representation of $(1 - \tilde{q})^{-1}$ given by $(K_{0}^{-1} e_1, e_1,  A_{1}, \ldots,A_{\gamma_*})$ (see Remark~\ref{rem:aAdvantage}).
  Denote 
  \begin{equation*}
    l_{\emptyset}
    :=
    P_U K_0^{-1} e_1
    ,\quad
    r_{\emptyset}
    :=
    P_{\tilde{U}} e_1
    ,\quad
    Q_{\gamma}
    :=
    P_{\tilde{U}} A_{\gamma} P_{\tilde{U}}
    ,
  \end{equation*}
  where $P_{\tilde{U}} \, : \, \CC^{n}\rightarrow \tilde{U}$ is an orthogonal projection onto $\tilde{U}$.
  Then $(l_{\emptyset}, r_{\emptyset}, Q_1, \ldots,Q_{\gamma_*})$ also gives a matrix representation of $(1-q)^{-1}$.
  Indeed, for any $v\in \CC^{n}$ and $\overline{\alpha}\in \Icc$ we have that
  \begin{equation}\label{eq:aKR0}
    (P_{U} A P_{U})_{\overline{\alpha}} \,v
    =
    A_{\overline{\alpha}} \, P_{U} \,v
    ,\qquad
    Q^*_{\overline{\alpha}} \, v
    =
    (P_{\tilde{U}} A^* P_{\tilde{U}})_{\overline{\alpha}} \, v
    =
    (P_{U} A^* P_{U})_{\overline{\alpha}}\, P_{\tilde{U}} \, v
    ,
  \end{equation}
  where the first equality in \eqref{eq:aKR0} follows from the definition of $U$ \eqref{eq:Udef} and the fact that $A_{\gamma} ( U) \subset U$, whereas the second is due to the definition of $\tilde{U}$ \eqref{eq:Utilde} and $P_{U} A^*_{\gamma} P_{U} (\tilde{U})\subset \tilde{U}$.
  In particular, we have that
  \begin{equation}\label{eq:aKR1}
    (P_{U} A P_{U})_{\overline{\alpha}} \, e_1
    =
    A_{\overline{\alpha}} \, e_1
    ,\qquad
    Q^*_{\overline{\alpha}} \, l_{\emptyset}
    =
    (P_{U} A^* P_{U})_{\overline{\alpha}}\,  K_0^{-1} e_1
    ,
  \end{equation}
  which implies
  \begin{equation*}
    \la l_{\emptyset}, Q_{\overline{\alpha}} r_{\emptyset} \ra
    =
    \la (P_{U} A^* P_{U})_{\overline{\alpha}^t}\,  K_0^{-1} e_1 , e_1 \ra
    =
    \la K_0^{-1} e_1 , (P_{U} A P_{U})_{\overline{\alpha}} \, e_1 \ra
    =
    \la  K_0^{-1} \, e_1,  A_{\overline{\alpha}}\,e_1 \ra
    .
  \end{equation*}
  This means that for any $\overline{\alpha}\in \Icc$
  \begin{equation}\label{eq:aKR3}
    \la l_{\emptyset}, Q_{\overline{\alpha}} r_{\emptyset} \ra
    =
    \la  K_0^{-1} \, e_1,  A_{\overline{\alpha}}\,e_1 \ra
  \end{equation}
  and we conclude that $(l_{\emptyset}, r_{\emptyset}, Q_{1}, \ldots Q_{\gamma_*})$ gives a matrix representation of $(1-\tilde{q})^{-1}$.
  In other words, we can still construct a matrix representation of $(1-\tilde{q})^{-1}$ if we restrict $A_{\gamma}$ to an $m$-dimensional subspace $\tilde{U}\subset \CC^{n}$.

  Moreover, 
  \begin{equation}\label{eq:aFullSpan}
    \mathrm{span} \Big(\bigcup_{\overline{\alpha}\in \Icc}  Q_{\overline{\alpha}} \, r_{\emptyset}  \Big)
    =
    \tilde{U}
    ,\qquad
    \mathrm{span} \Big(\bigcup_{\overline{\alpha}\in \Icc}  Q^*_{\overline{\alpha}} \, l_{\emptyset}  \Big)
    =
    \tilde{U}
    .
  \end{equation}
  To see this, assume that there exists $\tilde{u}\in \tilde{U}$ such that $\tilde{u}\perp Q_{\overline{\alpha}}\, r_{\emptyset}$ for all $\overline{\alpha}\in \Icc$.
  Then using \eqref{eq:aKR0} we get that
  \begin{equation*}
    0
    =
    \la \tilde{u}, Q_{\alpha_1}\cdots Q_{\alpha_k} r_{\emptyset} \ra
    =
    \la \tilde{u}, A_{\alpha_1} \cdots A_{\alpha_k} e_1 \ra
    ,
  \end{equation*}
  and since this holds for every multi-index, this implies that $\tilde{u}\perp U$ and thus $\tilde{u}=0$.
  The second equality in \eqref{eq:aFullSpan} can be obtained similarly.
  
  Now we will show that matrices $\KK_{0}$ and $\KK_{\gamma}$ represent $Q_0|_{\tilde{U}}: \tilde{U}\rightarrow \tilde{U}$ and $Q_{\gamma}|_{\tilde{U}}: \tilde{U}\rightarrow \tilde{U}$ in a properly chosen basis. To this end,  for any multi-index $\overline{\alpha}\in \Icc$ denote
  \begin{equation*}
    l_{\overline{\alpha}}
    :=
    Q^*_{\overline{\alpha}}\,l_{\emptyset}
    ,\quad
    r_{\overline{\alpha}}
    :=
    Q_{\overline{\alpha}}\, r_{\emptyset}
    ,
  \end{equation*}
  and note that due to \eqref{eq:aKR1} $\{l_{\overline{\beta}_i}\, : \, 1\leq i \leq m \}$ gives a basis of $\tilde{U}$ for some set of multi-indices $\{\overline{\beta}_i,1\leq i \leq m \}$ with $\overline{\beta}_1=\emptyset$.
  
  Now we show that $\{r_{\overline{\beta}_i}\, : \, 1\leq i \leq m \}$ is linearly independent, hence it also forms a basis of $\tilde{U}$.
  Suppose there exist $c_1,\ldots, c_m \in \CC$ such that $\sum_{j=1}^m c_j r_{\overline{\beta}_j}=0$.
  This means, that for all $\overline{\alpha}\in \Icc$
  \begin{equation*}
    \Big\la l_{\overline{\alpha}} ,\sum_{j=1}^m c_j r_{\overline{\beta}_j}  \Big\ra
    =
    0
    .
  \end{equation*}
  Using \eqref{eq:aKR3} and the straightforward indentity  $ A^{ * }_{\overline{\alpha}} K_0^{-1}= K_0^{-1}A_{\overline{\alpha}}$ that is valid for all $\overline{\alpha}\in \Icc$, it is easy to see that  for all $\overline{\alpha}, \overline{\beta} \in \Icc$
  \begin{equation}\label{eq:aLinInd}
    0
    =
    \Big\la l_{\overline{\alpha}} ,\sum_{j=1}^m c_j r_{\overline{\beta}_j}  \Big\ra 
    =
    \sum_{j=1}^m  c_j \la l_{\overline{\alpha}}, r_{\overline{\beta}_j} \ra
    =
    \overline{\Big\la \sum_{j=1}^m c_j l_{\overline{\beta}_j} , r_{\overline{\alpha}}, \Big\ra}
    \quad
    \mbox{ for all $\overline{\alpha}\in \Icc$ }
    .
  \end{equation}
  Since $\{l_{\overline{\beta}_i}, 1\leq i \leq m \}$ is a basis of $\tilde{U}$ and $\tilde{U}$ is generated by $\{r_{\overline{\alpha}}, \overline{\alpha}\in \Icc\}$, we have from \eqref{eq:aLinInd} that $\sum_{j=1}^m c_j l_{\overline{\beta}_j}=0$, which implies that $\sum_{j=1}^m |c_j|=0$.
  This shows that $\{r_{\overline{\beta}_j},1\leq j\leq m\}$ is linearly independent and thus forms a basis of $\tilde{U}$.

  Define now $n\times m$ matrices $B_{L}$ and $B_{R}$, whose columns are the basis vectors $l_{\overline{\beta}_i}$ and $r_{\overline{\beta}_i}$ correspondingly, i.e., $B_{L}:= (l_{\overline{\beta}_i}\, : \, 1\leq i \leq m)$ and $B_{R}:= (r_{\overline{\beta}_i}\, : \, 1\leq i \leq m)$.
  Then from \eqref{eq:aKR3} we have that
  \begin{equation}\label{eq:aKB}
    \KK_0
    =
    B_{L}^{*} B_{R}
    ,\qquad
    \KK_{\gamma}
    =
    B_{L}^{*} Q_{\gamma} B_{R}
    .
  \end{equation}
  Matrix $\KK_0$ is obviously invertible from this construction, since the columns of $B_{L}$ and $B_{R}$ form two bases of $\tilde{U}$, thus $B_{R}=B_{L} T$ for some invertible $T\in \CC^{m\times m}$.
  On the other hand, since $P_{\tilde{U}}^{B}:= B_{R}(B_{L}^* B_{R})^{-1} B_{L}^*$ is a projection onto $\tilde{U}$, we have
  \begin{equation}\label{eq:aKR4}
    \la e_1, (\KK \KK_{0}^{-1})_{\overline{\alpha}}\, \KK_{0}e_1 \ra
    =
    \la   B_{L}\, e_1, P_{\tilde{U}}^{B} \,  Q_1\, P_{\tilde{U}}^{B}\cdots P_{\tilde{U}}^{B}\, Q_{\alpha_k} \, P_{\tilde{U}}^{B}\, B_{R} \, e_1 \ra
    =
    \la l_{\emptyset}, Q_{\overline{\alpha}}  r_{\emptyset} \ra
  \end{equation}
  for any $\overline{\alpha} \in \Icc$, which implies that $(\KK_{0}e_1,  e_1, \KK_{1}\KK_{0}^{-1} , \ldots,  \KK_{\gamma_*}\KK_{0}^{-1})$ is a matrix representation of $(1-\tilde{q})^{-1}$ of dimension $m$.

  In the last step we make a change of basis that allows us to replace $\KK_{0}e_1$ by $e_1$.
  By the choice of $W$ we have that $W^* \KK_0 e_1 = \| \KK_0 e_1 \|_2 e_1$, so that
  \begin{equation}\label{eq:aAlmost}
    \la  e_1 , (\KK \KK_{0}^{-1})_{\overline{\alpha}} \, \KK_0 e_1 \ra
    =
    \| \KK_0 e_1 \|_2^2 \la  e_1 ,  W^* \KK_0^{-1} W \, (W^* \KK W W^*\KK_0^{-1} W )_{\overline{\alpha}} \, e_1 \ra
    ,
  \end{equation}
  and  $\| \KK_0 e_1 \|_2^2$ will be absorbed by one of the ${\Qmin_0}^{-1}$ if we define $\Qmin_0$ and $\Qmin_{\gamma}$ via \eqref{eq:aKmin}.
  Therefore, from the definition of $\Qmin$ \eqref{eq:aKmin}, \eqref{eq:aAlmost}, \eqref{eq:aKR4} and \eqref{eq:aKR3} we obtain that
  \begin{equation}\label{eq:aDone}
    \la \Qmin_0^{-1} e_1, (\Qmin \Qmin_{0}^{-1})_{\overline{\alpha}} \, e_1 \ra
    =
    \la K_0^{-1} e_1, A_{\overline{\alpha}} \, e_1 \ra
    ,
  \end{equation}
  and we conclude that $(\Qmin_0^{-1} e_1, e_1, \Qmin_{1} \Qmin_{0}^{-1}, \ldots, \Qmin_{\gamma_*}\Qmin_{0}^{-1})$ is a matrix representation of $(1-\tilde{q})^{-1}$.

  Moreover,
  \begin{equation}\label{eq:aKminSpan}
    \mathrm{span} \Big(\bigcup_{\overline{\alpha}\in \Icc}  (\Qmin \Qmin_0^{-1})_{\overline{\alpha}} \, e_1  \Big)
    =
    \CC^{m}
    ,\qquad
    \mathrm{span} \Big(\bigcup_{\overline{\alpha}\in \Icc}  (\Qmin \Qmin_0^{-1})^*_{\overline{\alpha}} \, \Qmin_{0}^{-1}\, e_1  \Big)
    =
    \CC^{m}
    .
  \end{equation}
  Indeed, suppose that there exist $c_1,\ldots, c_m \in \CC$ satisfying $\sum_{i=1}^{m} |c_i|>0$, such that
  \begin{equation*}
    \sum_{i=1}^m c_i \, (\Qmin \Qmin_0^{-1})_{\overline{\beta}_i} \,  e_1
    =
    0
    .
  \end{equation*}
  Then for any $\overline{\alpha}\in \Icc$
  \begin{equation*}
    \Big\la  (\Qmin \Qmin_0^{-1})^*_{\overline{\alpha}} \Qmin_0^{-1} ,  \sum_{i=1}^m c_i \, (\Qmin \Qmin_0^{-1})_{\overline{\beta}_i} \,  e_1 \Big\ra
    =
    0
  \end{equation*}
  which by \eqref{eq:aDone} and \eqref{eq:aKR3}  means that $    \la  l_{\overline{\alpha}},\sum_{i=1}^m c_i \, r_{\overline{\beta}_i}\ra = 0$ and contradicts to the fact that $\{r_{\overline{\beta}_i}\, 1\leq i \leq m\}$ is a basis of $\tilde{U}$ and $\{r_{\overline{\alpha}}\,: \, \overline{\alpha}\in \Icc\}= \tilde{U}$.
  Therefore, $\{ (\Qmin \Qmin_0^{-1})_{\overline{\beta}_i} \, e_1 , 1\leq i \leq m \}$ is linearly independent,  which implies the first equality in  \eqref{eq:aKminSpan}.
  The second equality in  \eqref{eq:aKminSpan} can be shown using a similar argument.

  Now we can finish the proof of Lemma~\ref{lem:aminimization}.
  By construction matrices $\Qmin_0$ and $\Qmin_{\gamma}$ are Hermitian.
  Moreover, by \eqref{eq:aKB} and \eqref{eq:aKmin} $\Qmin_0$ is invertible and by \eqref{eq:aDone}
  \begin{equation}
    \label{eq:111}
    \la e_1, \Qmin_0^{-1} \, e_1 \ra
    =
    \la e_1, K_0^{-1} e_1 \ra = 1
    . 
  \end{equation}
  It remains to show that $\Lb_{\mathrm{m}}:=\Qmin_0 \otimes \cstarunit - \sum_{\gamma=1}^{\gamma_*} \Qmin_{\gamma}\otimes x_{\gamma}$ is minimal and satisfies 
   \eqref{eq:Lin11}.
  
  Similarly as in Remark~\ref{rem:aAdvantage}, if we assume that $\sum_{\gamma=1}^{\gamma_*} \| \Qmin_{\gamma} \Qmin_0^{-1}\| \|x_{\gamma}\|_{\Acc} \leq 1/2$, then by \eqref{eq:aDone} and \eqref{eq:aGenRes11o}
  \begin{equation}\label{eq:aFinale1}
    \Big\la  e_1 \otimes \cstarunit, \Big( \Qmin_0 \otimes \cstarunit - \sum_{\gamma=1}^{\gamma_*} \Qmin_{\gamma}  \otimes x_{\gamma}\Big)^{-1} \, e_1\otimes \cstarunit \Big\ra_{\Acc}
    =
    \frac{1}{\cstarunit - \tilde{q}(\xb)}
    .
  \end{equation}
  On the other hand, if similarly to \eqref{eq:Lexpr1} we write
  \begin{equation*}
    \Lb_{\mathrm{m}}
    =
    \begin{pmatrix}
      \lambda_{\mathrm{m}} & \ell_{\mathrm{m}}^*
      \\
      \ell_{\mathrm{m}} & \widehat{\Lb}_{\mathrm{m}}
    \end{pmatrix}
      ,
    \end{equation*}
then by the Schur complement formula
  \begin{equation*}
    \frac{1}{[\Lb_{\mathrm{m}}^{-1}]_{11}}
    =
    \lambda_{\mathrm{m}} - \ell_{\mathrm{m}}^* \widehat{\Lb}_{\mathrm{m}}^{-1} \ell_{\mathrm{m}}
    ,
  \end{equation*}
  which together with \eqref{eq:aFinale1} implies  \eqref{eq:Lin11}.

  Finally, minimality of $\Lb_{\mathrm{m}}$ follows from \eqref{eq:aKminSpan}.
  Indeed, \eqref{eq:aKminSpan} implies that
  \begin{equation*}
    \big(\Qmin_0^{-1} e_1, e_1, \Qmin_{1}\Qmin_0^{-1}, \ldots, \Qmin_{\gamma_*} \Qmin_0^{-1}\big)
  \end{equation*}
  is a matrix representation of $(1-\tilde{q})^{-1}$ of the lowest possible dimension.
  If we assume that there exist a linearization $\Lb'=K'_0\otimes \cstarunit - \sum_{\gamma=1}^{\gamma_*} K'_{\gamma} \otimes x_{\gamma}$ with dimension smaller than $m$, then
  \begin{equation*}
    \big((K'_0)^{-1} e_1, e_1, K'_{1} (K'_0)^{-1}, \ldots, K'_{\gamma_*} (K'_0)^{-1}\big)
  \end{equation*}
  would give a matrix representation of $(1-\tilde{q})^{-1}$ of dimension smaller than $m$, which would lead to a contradiction.
  We conclude that $\Lb_{\mathrm{m}}$ is a minimal linearization of $\cstarunit - \tilde{q}(\xb)$.
\end{proof}

\subsection{Numerical   comparison of two linearizations}
\label{sec:aLinComp}

In the next two tables we show how the dimension of the standard 
linearization from Appendix~\ref{sec:abiglinearization} relates to the dimension the minimal linearization for polynomials having different 
 degrees  and structures.

The first table (Figure~\ref{fig:table1}) shows how the dimensions of the two different linearizations  depend on the degree of the polynomial.
For a given degree, we generated  random samples of noncommutative polynomials  in two noncommutative variables 
by choosing the coefficients of all possible  monomials up to the given degree  independently (up to symmetry constraints) and uniformly  from an interval.

\begin{figure}[!h] 
  \centering
  \begin{tabular}{|C{1cm}|C{2cm}|C{2cm}|}
  \hline
   \mbox{degree}
  & $\gamma$  
  & $\delta$  
  \\ \hline\hline
   1 & 1 & 1
  \\ \hline
   2 & 9 & 3
  \\ \hline
   3 & 41  & 5
  \\ \hline
   4 & 137 & 9
  \\
  \hline
   5 & 393 & 13
  \\
    \hline
   6 & 1033 & 21
  \\
  \hline
   7 & 2563 & 29
  \\
  \hline
   8 & 6153 & 45
  \\
  \hline

  \end{tabular}
  \caption{For random polynomials with a given degree, 
  $\gamma$ and $\delta$ are the average dimensions of the standard and the minimal linearizations, respectively.}
  \label{fig:table1}
\end{figure}

The second table, Figure~\ref{fig:table2},  illustrates how the dimension of the minimal linearization may depend on the structure of the polynomial.
We again generated samples of  polynomials in two noncommutative variables.
This time each sample is characterized by two given numbers, the lowest and highest degree of the monomials allowed in the polynomials.
The coefficients of the monomials are given as a product of independent (up to symmetry constraints) random variable uniformly distributed on an interval and a
 0--1  Bernoulli random variable with parameter chosen is such a way that the standard linearization for all four samples have approximately the same dimension
 around 2000. In other words, the Bernoulli variable picks an appropriate subset of the all possible monomials and then we further randomize
 its coefficient. This random preselection is necessary to keep the calculation at manageable length.

\begin{figure}[!h]
  \centering
  \begin{tabular}{|C{2cm}|C{2cm}|C{2cm}|C{2cm}|}
  \hline
   \mbox{Min degree}
  & \mbox{Max degree}
  & $\gamma$ 
  & $\delta$ 
  \\ \hline\hline
   1 & 7 & 2113 & 29
  \\ \hline
   6 & 8 & 2076 & 44,9
  \\ \hline
   9 & 10 & 2082,8 & 86,2
  \\ \hline
   11 & 12 & 1930 & 150,9
  \\
  \hline
  \end{tabular}
  \caption{ For random polynomials consisting of monomials with a given minimal and maximal degree, 
  $\gamma$ and $\delta$ are the average dimensions of the standard and the minimal linearizations, respectively.}
    \label{fig:table2}
\end{figure}

Above results suggest that the minimal linearization provides a  substantial  reduction in the size of the linearization for a typical polynomial with no restriction on its
 structure, and this reduction becomes less significant if we restrict the polynomial to have only monomials of higher degrees.  
 In other words, minimal linearization is the most advantageous over the customary one if the polynomial is the sum of many monomials.
Note that randomization
 excludes the polynomials of very special structure, for example, high powers of linear combinations of noncommutative variables, which may behave very differently.

\section{Properties of semicircular noncommutative random variables}\label{sec:semicircular}

\label{sec:freeIntro}
The aim of this section is to recall some basic definitions related to the $C^*$-probability spaces and semicircular random variables that are used throughout the paper.
For a more complete introduction to the subject we refer  the  reader to \cite[Section 5]{AndeGuioZeitBook}.

\begin{defn}[$C^*$-algebra and $C^*$-probability space]
  We call $\Acc$ a \emph{(unital) $C^*$-algebra}, if
  \begin{itemize}
      \setlength{\itemsep}{0em}
  \item[(i)]   $\Acc$ is a (unital) algebra endowed with an  involution $^*$ and norm $\|\,\cdot \,\|_{\Acc}$ satisfying
  \begin{equation*}
    \|ab\|_{\Acc}\leq \|a\|_{\Acc}\|b\|_{\Acc}
    ,\quad
    \|a^* a \|_{\Acc}
    =
    \|a\|_{\Acc}^2
  \end{equation*}
  for any $a,b\in \Acc$;
  \item[(ii)] $(\Acc,\|\cdot\|_{\Acc})$ is a Banach space.
  \end{itemize}
  If $\Acc$ is a unital $C^*$-algebra and $\tau:\Acc \rightarrow \CC$ is a linear complex-valued functional such that
  \begin{equation*}
    \tau(a^*a)
    \geq
    0
    ,\quad
    \tau(\cstarunit)
    =
    1
  \end{equation*}
  for any $a\in \Acc$ and the unit element $\cstarunit\in\Acc$, then we call $(\Acc, \tau)$ a \emph{$C^*$-probability space}.
  We will always assume that the state $\tau$ is tracial ($\tau(ab)=\tau(ba)$ for all $a,b\in \Acc$) and faithful ($\tau(a^*a)=0$ implies that $a=0$).
\end{defn}

We call the elements of a $C^*$-probability space $\Acc$ \emph{non-commutative random variables}.
A family $\{a_1,\ldots,a_k\}\subset \Acc$ of non-commutative random variables is characterized by its \emph{non-commutative distribution}, a map $\mu_{a_1,\ldots,a_k}\, : \, \CC\la x_1,\ldots,x_k \ra \rightarrow \CC$ given by
\begin{equation*}
  \mu_{a_1,\ldots,a_k} (P)
  =
  \tau(P(a_1,\ldots,a_k))
  ,\quad
  P\in \CC\la x_1,\ldots,x_k \ra
  ,
\end{equation*}
where we  recall that   $\CC\la x_1,\ldots,x_k \ra$ denotes the set of (noncommutative) polynomials in $x_1,\ldots,x_k$.

A family $\Acc_1,\ldots,\Acc_n \subset \Acc$ of subalgebras of $\Acc$, each containing $\cstarunit$, is called \emph{freely independent} if
\begin{equation*}
  \tau(a_1a_2\cdots a_k)
  =
  0  
\end{equation*}
for any $(i_1,\ldots,i_k)\in \{1,\ldots,n\}^k$ and $a_1\in \Acc_{i_1},\ldots,a_k\in \Acc_{i_k}$ with $\tau(a_j)=0$ and $i_1\neq i_2, i_2\neq i_3,\ldots, i_{k-1}\neq i_k$.
Noncommutative variables $a_1,\ldots,a_n$ are freely independent if the  subalgebras generated by $a_1,\ldots,a_n$ are freely independent.

A freely independent family of noncommutative variables $\semic_1,\ldots,\semic_k $ from the $C^*$-probability space $(\Acc,\tau)$ is called a \emph{semicircular system} if $\semic_i^*=\semic_i$ and
\begin{equation} \label{eq:aSemicMoments}
  \tau(\semic_i^{j})
  =
  \left\{
    \begin{array}{ll}
      0, & j\quad \mbox{ is odd,}
      \\
      C_{\ell}, & j=2\ell, \mbox{ even,}
    \end{array}
    \right.
  \end{equation}
  where $C_{\ell}$ is the $\ell$-th Catalan number,  $C_{\ell}=\frac{1}{\ell+1} \binom{2\ell}{\ell}$.
  We will denote by $\Scc$ the unital (with unit element $\freeunit$) $C^*$-algebra generated by $\{\semic_1,\ldots,\semic_{k}\}$, and $(\Scc, \tau)$ will be the corresponding $C^*$-probability space.
  The spectrum of a semicircular element $\semic$ is equal to the interval $[-2,2]$, in particular we have that $\|\semic\|_{\Scc}=2$.

\bibliographystyle{abbrv}
\bibliography{bib}

\end{document}